\documentclass[brochure,english,12pt]{smfbourbaki}

\usepackage[T1]{fontenc}
\usepackage{lmodern,amssymb,bm,bbm,mathrsfs,enumitem}

\usepackage[all]{xy}

\usepackage{babel}

\usepackage[utf8]{inputenc}

\usepackage[colorlinks=true, linkcolor=blue, citecolor=red, urlcolor=blue]{hyperref}

\usepackage[
backend=biber,
style=authoryear, 
citestyle=authoryear-comp,
maxnames=7,
sortcites=false ]{biblatex}
\usepackage{csquotes}

\newtoggle{tnbcbx@howcited}
\DeclareEntryOption[boolean]{howcited}[true]{\settoggle{tnbcbx@howcited}{#1}}
\DeclareBibliographyOption[boolean]{howcited}[true]{\settoggle{tnbcbx@howcited}{#1}}
\DeclareTypeOption[boolean]{howcited}[true]{\settoggle{tnbcbx@howcited}{#1}}

\newbibmacro{howcited}{\iftoggle{tnbcbx@howcited}
    {\iffieldundef{shorthand}
       {}
       {\setunit{\addspace}\printtext[parens]{\bibstring{citedas}\setunit{\addcolon\space}\printfield{shorthand}\setunit{\addslash}}}}
    {}}

\renewbibmacro{finentry}{\usebibmacro{howcited}\finentry}

\DefineBibliographyStrings{french}{citedas    = {cité ci-dessus avec l'acronyme}}
\DefineBibliographyStrings{english}{citedas    = {heretofore cited as}}

\DefineBibliographyExtras{french}{\restorecommand\mkbibnamefamily}

\DeclareDelimFormat{nameyeardelim}{\addcomma\space}

\DeclareNameAlias{sortname}{given-family}

\renewcommand{\bibnamedash}{\leavevmode\raise3pt\hbox to3em{\hrulefill}\space}

\AtEveryBibitem{\clearfield{issn} \clearfield{isbn} \clearfield{doi} \clearlist{language} \ifentrytype{online}{}{\ifentrytype{unpublished}{}{\clearfield{url}
  }
  }
}

\renewbibmacro{in:}{\ifentrytype{article}{}{\printtext{\bibstring{in}\intitlepunct}}}

\DeclareFieldFormat[article,periodical,inreference]{number}{\mkbibparens{#1}}
\DeclareFieldFormat[article,periodical,inreference]{volume}{\mkbibbold{#1}}
\renewbibmacro*{volume+number+eid}{\printfield{volume}\setunit*{\addthinspace}\printfield{number}\setunit{\addcomma\space}\printfield{eid}}

\DeclareFieldFormat[article,inbook,incollection]{title}{\enquote{#1}\addcomma}

\numberwithin{equation}{defi}

\DeclareMathOperator{\sS}{\Delta^{\op} Set} \DeclareMathOperator{\htp}{\mathrm H}
\newcommand{\tp}{\mathrm{top}}

\DeclareMathOperator{\Ab}{\mathrm{Ab}}

\DeclareMathOperator{\Sh}{Sh}
\DeclareMathOperator{\PSh}{PSh}

\DeclareMathOperator{\HM}{HM}

\DeclareMathOperator{\iS}{\mathscr S} \DeclareMathOperator{\iC}{\mathscr C}
\DeclareMathOperator{\ihtp}{\mathscr H}
\DeclareMathOperator{\iSH}{\mathscr{S\!H}}
\DeclareMathOperator{\iSHcell}{\mathscr{S\!H}^{\mathrm{cell}}}

\DeclareMathOperator{\iDer}{\mathscr D}
\DeclareMathOperator{\iDA}{\mathscr D_{\AA^1}}
\DeclareMathOperator{\iDM}{\mathscr{DM}}

\DeclareMathOperator{\iPSh}{\mathscr{PS}h}

\DeclareMathOperator{\Fun}{\mathscr Fun}

\DeclareMathOperator{\Ho}{Ho}

\DeclareMathOperator{\Ker}{Ker}

\DeclareMathOperator{\GW}{GW} \DeclareMathOperator{\K}{K} \DeclareMathOperator{\MW}{MW}
\DeclareMathOperator{\W}{W} \newcommand{\KMW}{\mathrm K^{\mathrm{MW}}}
\newcommand{\KM}{\mathrm K^{\mathrm M}}
\DeclareMathOperator{\holim}{holim}
\DeclareMathOperator{\hocolim}{hocolim}
\newcommand {\piS}[1]{\pi^{\mathrm S}_{#1}\!} \newcommand {\hpiS}[1]{\hat \pi^{\mathrm S}_{#1}\!} \DeclareMathOperator{\upi}{\underline \pi}
\newcommand {\piA}[1]{\underline \pi^{\AA^1}_{#1}\!}
\newcommand {\piAS}[1]{\underline \pi^{\AA^1}_{#1}\!} \newcommand{\uKMW}{\underline{\mathrm K}^{\mathrm{MW}}}
\newcommand{\uKM}{\underline{\mathrm K}^{\mathrm M}}
\DeclareMathOperator{\uH}{\underline H}

\DeclareMathOperator{\Sus}{\Sigma^\infty}
\DeclareMathOperator{\Lop}{\Omega^\infty}
\DeclareMathOperator{\Tot}{Tot} \DeclareMathOperator{\CB}{CB}

\newcommand{\qd}[1]{\langle#1\rangle}

\DeclareMathOperator{\tdeg}{\widetilde{deg}}
\DeclareMathOperator{\car}{char}

\newcommand{\topo}{\mathrm{cl}} 

\newcommand{\cX}{\mathcal X}
\newcommand{\cY}{\mathcal Y}
\newcommand{\cZ}{\mathcal Z}
\newcommand{\cG}{\mathcal G}

\newcommand{\E}{\mathbf E}

\newcommand{\sX}{\mathbf X}
\newcommand{\sY}{\mathbf Y}

\newcommand{\un}{\mathbbm 1}
\newcommand{\MGL}{\mathbf{MGL}}
\newcommand{\MU}{\mathbf{MU}}
\newcommand{\sH}{\mathbf H}
\newcommand{\wMU}{\widehat{\mathbf{MU}}}
\newcommand{\BP}{\mathbf{BP}}
\newcommand{\MSL}{\mathbf{MSL}}
\newcommand{\BPGL}{\mathbf{BPGL}}
\newcommand{\KQ}{\mathbf{KQ}}
\newcommand{\kq}{\mathbf{kq}}

\newcommand{\KGL}{\mathbf{KGL}}
\newcommand{\sGW}{\mathbf{GW}}  \newcommand{\HB}{\mathbf H_{\mathrm B}}
\newcommand{\Het}{\mathbf H_{\et}}

\newcommand{\HMot}{\mathbf H_{\mathrm{M}}}
\newcommand{\M}{\mathbf M}

\newcommand{\HH}{\mathrm H}
\newcommand{\HA}{\mathrm H_{\AA^1}}

\newcommand{\GL}{\mathrm{GL}}
\newcommand{\SL}{\mathrm{SL}}
\newcommand{\Orth}{\mathrm{O}}

\newcommand{\BGL}{\mathrm{BGL}}
\newcommand{\BetO}{\mathrm{B_{\et}O}}

\newcommand{\etat}{\eta^{\mathrm{top}}}

\newcommand{\Sm}{\mathrm{Sm}}

\DeclareMathOperator{\SH}{SH}

\DeclareMathOperator{\Sq}{Sq} 

\DeclareMathOperator{\Id}{Id}

\DeclareMathOperator{\Hom}{Hom}
\DeclareMathOperator{\Aut}{Aut}
\DeclareMathOperator{\Map}{Map}
\DeclareMathOperator{\End}{End}
\DeclareMathOperator{\Ext}{Ext}
\DeclareMathOperator{\uHom}{\underline{Hom}} 

\DeclareMathOperator{\op}{op}
\DeclareMathOperator{\Pic}{Pic}
\DeclareMathOperator{\CH}{\mathrm{CH}}

\DeclareMathOperator{\spec}{Spec}

\newcommand{\ilim}[1]{\varinjlim_{#1}}

\DeclareMathOperator{\Th}{\mathrm{Th}}

\newcommand{\tw}[1]{\{#1\}}

\newcommand{\NN} {\mathbb N}
\newcommand{\ZZ} {\mathbb Z}

\newcommand{\QQ} {\mathbb Q}
\newcommand{\RR} {\mathbb R}
\newcommand{\CC} {\mathbb C}
\newcommand{\FF} {\mathbb F}

\newcommand{\cO}{\mathcal O}

\newcommand{\cL}{\mathcal L} 

\renewcommand{\AA} {\mathbb A}
\newcommand{\PP} {\mathbb P}

\newcommand{\GG} {\mathbb{G}_m}

\newcommand{\nis}{\mathrm{Nis}}
\newcommand{\zar}{\mathrm{Zar}}
\newcommand{\et}{\mathrm{\acute et}}

\date{Juin 2025}
\bbkannee{77\textsuperscript{e} année, 2024--2025}  \bbknumero{1241}                                    

\title{Motivic homotopy theory and stable homotopy groups}
\subtitle{after Morel--Voevodsky, Isaksen--Wang--Xu, ... }

\author{Frédéric Déglise}
\address{ENS de Lyon, UMPA, UMR 5669, \\ 46 all{\'e}e d'Italie, 69364 Lyon Cedex 07}
\email{frederic.deglise@ens-lyon.fr}

\begin{document}

\maketitle
\setcounter{tocdepth}{2}

\tableofcontents

\section*{Introduction}

The idea of applying topological methods to algebraic geometry dates back at least to Lefschetz,
 who in 1924 envisioned extending the nascent techniques of Analysis Situs to the study of algebraic varieties.
 Through the use of Lefschetz pencils, he introduced a topological viewpoint on phenomena such as degenerations and the behavior of varieties near their boundaries
 ---  ideas that already foreshadowed the modern concept of rational equivalence, which fundamentally encodes the deformation of cycles along the projective line \(\mathbb{P}^1\),
 and suggested a deep connection between geometry and homotopy theory.
 This foundational perspective was later profoundly reshaped by Grothendieck, who elevated the use of sheaves and topos theory
 as central tools for structuring and understanding cohomological invariants, culminating in his visionary and revolutionary, yet unfinished, theory of motives.

Beilinson revived this vision by formulating the concept of motivic cohomology, conceived as a universal cohomology theory for algebraic varieties,
 akin to singular cohomology in topology. His conjectures, especially those concerning the relation with algebraic K-theory and cycles,
 became a major driving force in the development of the theory. On the topological side, the advent of stable homotopy theory,
 initiated by Adams and his successors, introduced a new language centered around generalized cohomology and orientation theory;
 Brown representability, characteristic classes, and formal group laws became essential tools, shaping the modern understanding of stable phenomena.
 These culminated in the development of chromatic homotopy theory,
 a guiding philosophy that now serves as a central framework for understanding the layered structure of stable homotopy.

\bigskip

Motivic homotopy theory lies at the crossroads of these ideas and structures.
 Introduced by Voevodsky and later developed systematically with Morel, it seeks to import the methods of algebraic topology
 into the realm of algebraic geometry, first by defining homotopies using the affine line \(\mathbb{A}^1\), and second by building their homotopy theory
 within the topos of sheaves over smooth schemes, building on the foundational work of Illusie, Joyal, and Jardine.
 Motivic homotopy theory draws inspiration from several sources: from the theory of motives, with its Tate twists and philosophy of weights;
 from motivic cohomology, grounded in the theory of algebraic cycles and Chow groups; and from stable homotopy theory,
 especially through the examples of cobordism and Morava \(K\)-theories.
 Over time, it has developed into a rich and coherent framework that unifies these diverse perspectives.

At the heart of this unifying approach lies the theory of motivic homotopy sheaves, whose internal structure is governed by unramified cohomology.
 While the latter was originally developed by Gersten, Bloch–Ogus, and others,
 it reappears in motivic homotopy theory as an intrinsic phenomenon, encoded in the properties of motivic homotopy sheaves with respect Gersten resolutions.
 These ideas are central to Voevodsky’s approach to motivic complexes and play a foundational role in Morel’s generalization to the full \(\mathbb{A}^1\)-homotopy framework over a field.
 Through this perspective, motivic homotopy sheaves reflect the local-to-global nature of algebraic phenomena and serve as a key organizing tool in both the unstable and stable settings.
 Remarkably, they are also accessible to explicit computation, a feature we will illustrate throughout these notes.

\bigskip

The impact of motivic homotopy theory has been both broad and deep. It played a central role in Voevodsky’s proof of the Milnor and Bloch–Kato conjectures,
 which establish powerful bridges between motivic cohomology and étale cohomology, Galois cohomology and and Milnor K-theory.
 It also led to the development of the theory of motivic complexes and the six functors formalism, and to Ayoub's theory of the motivic Galois group.
 Morel's foundational work introduced new quadratic invariants --- such as the Chow–Witt groups --- which opened new perspectives on quadratic enumerative geometry
 lead by Levine and Wickelgren, giving for example the emerging notion of quadratic \(L\)-functions.
 The decomposition of stable motivic homotopy theory into \(\pm\)-parts has revealed previously unseen structures in algebraic geometry,
 particularly over real fields, where motivic realization functors extend the classical links with complex geometry into genuinely new territory
 connected to real algebraic geometry.

Motivic obstruction theory has also provided fresh approaches to the classification of algebraic vector bundles,
 through classifying spaces and characteristic classes valued in Chow–Witt groups, as developed notably by Asok and Fasel (building on an idea of Morel)
 in their work on Murthy's splitting conjecture.
 More broadly, motivic homotopy invariants often mirror those in classical homotopy theory. 
 Over time, increasingly rich connections between motivic and classical stable homotopy theory have been uncovered,
 thanks in part to the structural insights made possible by the Milnor and Bloch–Kato conjectures.
 These links have profoundly transformed the computation of classical stable stems,
 notably through the motivic approach advanced in its latest stage by \cite{IWX},
 which will form the central focus of the final part of these notes.
 Let us finally mention that the motivic approach, via synthetic homotopy theory,
 also led to a proposed resolution of the last remaining case of the Kervaire invariant one problem, in recent work by \cite{LWX}.

\bigskip

The present text is structured to support a gradual conceptual progression and to highlight the central ideas of motivic homotopy theory.
 The first part lays out guiding principles, beginning with complex and real realization functors,
 naïve \(\mathbb{A}^1\)-homotopies, and the choice of topology, gradually introducing model categories and \(\infty\)-categories only as needed.
 The second part is devoted to unstable motivic homotopy theory, using the language of sheaves and localizations more extensively, and motivating the definition of motivic homotopy sheaves.
 The third part presents stable homotopy theory via \(\mathbb{P}^1\)-stabilization, and studies motivic stable stems through the lens of Morel's degree, slice filtration and higher homotopy tools.
 The final part is devoted to recent developments, in particular the computations of motivic and classical stable stems based on the motivic Adams spectral sequence
 and the deformation of homotopy theories via the motivic class \(\tau\), drawing from the work of \cite{IWX}.

This text is written to be accessible at multiple levels.
 The first two parts are aimed at readers with a basic background in algebraic topology and algebraic geometry.
 The latter sections require more familiarity with stable homotopy theory,
 and in particular the last part fully adopts the language of $\infty$-categories.
 A brief review of this formalism is included at the end of the first section.
 We have tried throughout the text to maintain a pedagogical style, giving all necessary definitions and providing references to the literature.

We also refer the reader to several excellent surveys on motivic homotopy theory, each offering a distinct point of view;
 see, e.g., \cite{AE}, \cite{WW}, \cite{AsokOst}.

\section*{Conventions}

Throughout these notes, \( k \) denotes a fixed base field.

We work in the language of schemes: all schemes are assumed to be of finite type over \( k \). By a (algebraic) variety over \( k \), we mean a quasi-projective scheme, that is, a scheme locally defined by polynomial equations in an affine or projective space. Readers unfamiliar with the language of schemes may safely replace \( k \)-schemes with algebraic varieties over \( k \); the general theory remains unaffected.

Unless stated otherwise, the term ``monoidal'' means ``symmetric monoidal''. Units of monoidal categories are typically denoted by \( \un \).

A ``space'' always refers to a simplicial set. A ``presheaf'' (resp. ``sheaf'') means, without explicitly stated otherwise, a presheaf (resp. Nisnevich sheaf) over the category \( \Sm_k \) of smooth schemes over \( k \). A ``\( k \)-space'' is a presheaf of simplicial sets on \( \Sm_k \).

In the last section, all formal group laws are assumed to be commutative.

Concerning foundational choices: the unstable part of the theory is formulated, as far as possible, using a combination of model categories and \(\infty\)-categories,
 in order to help the reader navigate the existing literature. The stable part, however, relies more systematically on the \(\infty\)-categorical formalism.

\section*{Acknowledgments}

I am grateful to Aravind Asok, Dan Isaksen, and Oliver Röndigs for answering many questions
 and for enlightening exchanges. Joseph Ayoub kindly shared with me beautiful notes on his course in motivic homotopy.
 I would also like to thank the collaborators of Nicolas Bourbaki, as well as Aravind Asok, Adrien Dubouloz, Jean Fasel, Niels Feld, Paul Arne Østvær, Rakesh Pawar, Oliver Röndigs,
 and Raphaël Ruimy for their attentive reading and valuable suggestions on an earlier version of these notes.

\section{A few guiding principles}

\subsection{Complex and real homotopy types of algebraic varieties} \label{sec:real}

The main motivation behind motivic homotopy theory is to extend the invariants of
 classical homotopy theory to algebraic varieties.
 As envisioned by Lefschetz one century ago,
 the main guide to do so is to use the topological space underlying complex algebraic varieties.
 In fact, one of the appealing features of motivic homotopy is that it also naturally incorporate
 the homotopy of real algebraic varieties.

Let us consider a real or complex embedding $\sigma:k \rightarrow E=\RR, \CC$
 of our base field $k$, and $X$ be an algebraic $k$-variety.
 When $\sigma$ is a complex embedding,
 we let $X^\sigma(\CC)=\Hom_k(\spec(\CC),X)$ be the set $\CC$-points of $X$,
 where $\spec{\CC}$ is viewed as a $k$-scheme via $\sigma$,
 endowed with its natural analytic topology --- coming from its canonical structure of complex analytic variety: see \cite[XII, Th. 1.1]{SGA1}.
 Similarly, when $\sigma$ is a real embedding, we let $X^\sigma(\RR)=\Hom_k(\spec(\RR),X)$ endowed with
 canonical euclidean topology (similarly coming from its structure of real analytic variety).
 In any case, one can define the \emph{$\sigma$-Betti homotopy type} of the $k$-scheme $X$,
 that is the isomorphism class of the topological space underlying $X^\sigma(E)$ in the homotopy category $\htp^\tp$.
 It is instructive to determine this purely topological invariant of $k$-schemes in a few cases,
 left as exercises to the reader.
\begin{center}
\begin{tabular}{|l|l|l|}
\hline
$k$-schemes & real case & complex case \\
\hline
$\AA^n$ & $*$ & $*$ \\
$\PP^1$ & $S^1$ & $S^2$ \\
$\GG$ & $S^0$ & $S^1$ \\
$\AA^n-\{0\}$ & $S^{n-1}$ & $S^{2n-1}$ \\
$x^2=y^3$ (cuspidal curve) & * & * \\
$y^2=x^2(x+1)$ (nodal curve) & $S^1$ & $S^1$ \\
$Q_{2n-1}:=V_{\AA^{2n}}(\sum_{i=1}^n x_iy_i=1)$ & $S^{n-1}$ & $S^{2n-1}$ \\
$Q_{2n}:=V_{\AA^{2n+1}}(\sum_{i=1}^n x_iy_i=z(1+z))$ & $S^{n}$ & $S^{2n}$ \\
\hline
\end{tabular} \\
$\sigma$-Betti homotopy type
\end{center}
Therefore, with the aim to define a motivic homotopy type
 which admits both a real and complex realization,
 this table tells us two things: 
\begin{enumerate}
\item the affine line should be contractible;
\item there are several algebraic models whose motivic homotopy type should look like a sphere.
\end{enumerate}

\subsection{Naive $\AA^1$-homotopies.}

As suggested by the considerations of the previous subsection, 
 Voevodsky's fundamental idea emerges:
 to use the affine line $\AA^1_k$ to parameterize motivic homotopies.
 This echoes the notion of rational equivalence on algebraic cycles,
 already considered by Lefschetz, except one uses the projective line in the case of cycles.
 One introduces the following definition:\footnote{In \cite[p. 88-89]{MV}, it was called a \emph{strict} $\AA^1$-homotopy
 equivalence. The following terminology seems to be the more commonly used now.}
\begin{defi}\label{df:naive-A1}
One defines the \emph{naive $\AA^1$-homotopy equivalence relation} on morphisms of $k$-schemes
 as the transitive relation generated by the following symmetric and reflexive relation
 between two morphisms $f,g:Y \rightarrow X$ of algebraic $k$-varieties:
$$
\exists H:\AA^1 \times Y \rightarrow X \mid
 H \circ s_0=f, H \circ s_1=g
$$
where $s_0$ and $s_1$ are respectively the zero and unit sections of $\AA^1_k$.

A $k$-scheme $X$ is \emph{naively $\AA^1$-contractible} if it admits
 a rational point $x$ and the identity map $\Id_X$ is naively $\AA^1$-equivalent
 to the composition $X \rightarrow \spec(k) \xrightarrow x X$.
\end{defi}
In particular, both the algebraic affine $k$-variety $\AA^n_k$ 
 and the cuspidal curve $V:x^2=y^3$ are naively $\AA^1$-contractible.

This equivalence relation is compatible with composition.
 Therefore one can define the \emph{naive motivic homotopy category} as the category
 whose objects are smooth $k$-schemes and morphisms are given by naive $\AA^1$-homotopy
 equivalence classes simply denoted by $[X,Y]^N$. As in topology,
 we will see that this definition is not strong enough to define a suitable motivic homotopy
 theory (see in particular \ref{ex:naive&weak-A1}). But remarkably,
 it is already possible to make computations with this definition.
 Let us first introduce the pointed version.
\begin{defi}\label{df:pointed-naive}
A base point of a $k$-scheme is a rational point $x \in X(k)$.
 One defines \emph{naive pointed $\AA^1$-homotopy equivalence relation}
 between two pointed morphisms $f,g:(Y,y) \rightarrow (X,x)$ as above,
 but requiring that the homotopies $H(t,-)$ are all pointed maps.

We let $[(Y,y),(X,x)]_\bullet^N$ be the naive $\AA^1$-homotopy classes of pointed maps
 modulo this equivalence relation.
\end{defi}

\subsubsection{Witt monoid}\label{num:GW}
 One of the striking aspects of motivic homotopy theory,
 as discovered by Morel, is its surprising connection to the theory of quadratic forms
 --- more accurately, inner products, to accommodate characteristic $2$.
 Let us recall some basic definitions of this rich and fascinating subject
 --- the reference book \cite{MH} will be well-suited to our point of view.

The isomorphism classes of finite-dimensional non-degenerate symmetric 
 bilinear forms over $k$ form a monoid under orthogonal sum which we denote
 by $(\MW(k),+)$ and refer to as the \emph{Witt monoid}. 
 Note moreover that the tensor product induces a semi-ring structure on $\MW(k)$, denoted simply by
 $(\MW(k),+,\times)$. Let $Q(k):=k^\times/(k^\times)^2$ be the \emph{quadratic classes} of units of $k$,
 equipped with its group structure.
 Then one obtains a canonical morphism of (multiplicative) monoids: $(Q(k),\times) \rightarrow (\MW(k),\times)$
 which to a unit $u$ associates the symmetric bilinear form $u.xy$, whose isomorphism class is denoted by $\qd u$.
 Standard notations in this context are:
\begin{itemize}
\item $\qd{a_1,\hdots,a_n}=\qd{a_1}+\hdots+\qd{a_1}$, for the obvious diagonal symmetric bilinear form;
\item $h=\qd{1,-1}$, the \emph{hyperbolic form}.
\end{itemize}
One then defines two associated rings:
\begin{itemize}
\item the \emph{Witt ring} $\W(k)$ of $k$ which is the quotient of $(\MW(k),+,\times)$ with respect to the classical Witt equivalence relation; 
\item the \emph{Grothendieck-Witt ring} $\GW(k)$ of $k$ which is the group completion of the additive monoid $(\MW(k),+)$ equipped
 with the induced ring structure.
\end{itemize}
Note that one deduces from these definitions that $\W(k)=\GW(k)/(h)$.

One has two universal invariants associated with an element of $\MW(k)$,
 say given by the isomorphism class of $(V,\varphi)$:
 the \emph{rank} $\dim_k(V)$, and the \emph{discriminant} $\mathrm{disc}(V,\varphi) \in Q(k)$.

Note also that in characteristic not $2$, the monoid $(\MW(k),+)$ is cancellative,
 leading to a monomorphism: $\MW(k) \rightarrow \GW(k)$. In characteristic $2$, we let $\MW^s(k)$
 be the universal cancellative monoid associated with $\MW(k)$, so that $\MW^s(k) \rightarrow \GW(k)$
 is the universal monomorphism.

\subsubsection{Cazanave's theorem}
Let us now consider our motivic sphere $\PP^1_k$, pointed by $\infty$ according to Definition \ref{df:pointed-naive}.
 A pointed endomorphism of $\PP^1_k$ is represented by a rational function $f=\frac A B$
 where $A, B \in k[t]$ are coprime monic polynomials such that $n=\deg(A)>\deg(B)$.
 In this situation, one classically associates to $(A,B)$ the Bezout matrix $\mathrm{Bez}(A,B)$,
 which is a symmetric bilinear form whose determinant is given by $(-1)^{n(n+1)/2}.\mathrm{res}(A,B)$.
 We then define a canonical application:
\begin{equation}\label{eq:Caza}
\End_\bullet(\PP^1_k) \rightarrow \MW^s(k) \times_{Q(k)} k^\times,
f \mapsto \big(\mathrm{Bez}(A,B),(-1)^{n(n+1)/2}.\mathrm{res}(A,B)\big)
\end{equation}
The following theorem was obtained in \cite[Cor. 3.10, Th. 4.6]{Caza2}.
\begin{theo}\label{thm:Caza}
Consider the above notation. Then the above map induces a bijection of pointed sets:
$$
[\PP^1_k,\PP^1_k]_\bullet^N \rightarrow \MW^s(k) \times_{Q(k)} k^\times
$$
and in fact an isomorphism of semi-ring,
 for suitably defined addition on the left hand-side,
 multiplication being composition.
\end{theo}
This result will allow us to measure the difference between
 naive $\AA^1$-homotopies and weak $\AA^1$-homotopies;
 see Example~\ref{ex:naive&weak-A1}.

\subsubsection{Presentations}\label{num:present-MW}
The additive group of the Witt ring $(\W(k),+)$ has a well-known presentation, due to Witt,
 as the abelian group generated by $\qd{\bar u}$ for a quadratic class $\bar u \in Q(k)$
 modulo the relations:
\begin{equation}\label{eq:rel-W}
h=0, \quad  
\qd{\bar u,\bar v}=\qd{\bar u+\bar v,(\bar u+\bar v)\bar u\bar v}, \bar u, \bar v \in Q(k), u+v \neq 0
\end{equation}
One also gets the presentation of  $(\GW(k),+)$ as the abelian group genereated by $\qd{\bar u}$
 for a quadratic class $\bar u \in Q(k)$ modulo the single relations
\begin{equation}\label{eq:rel-MW}
\qd{\bar u,\bar v}=\qd{\bar u+\bar v,(\bar u+\bar v)\bar u\bar v}, \bar u, \bar v \in Q(k), u+v \neq 0
\end{equation}
One deduces from the above theorem the following very explicit computation of naive $\AA^1$-homotopy classes.
\begin{coro}
The pointed set $[\PP^1_k,\PP^1_k]_\bullet^N$ has a structure of abelian groups which 
 is generated by the symbols $(u)$ for the class of a unit $\bar u \in k^\times/(\pm 1)$,
 corresponding to the endomorphism $(x:y) \mapsto (ux:y)$,\footnote{with the convention that $(0:1)=\infty$}
 modulo the unique relation:
$$
(\bar u,\bar v)=(\bar u+\bar v,(\bar u+\bar v)\bar u\bar v), u+v \neq 0
$$
where $(\bar u,\bar v)=(\bar u)+(\bar v)$.
\end{coro}

\subsection{The choice of topology}

\subsubsection{}\label{sec:advantage-Nis}
 The category of algebraic varieties is too coarse to allow for homotopy theoretic constructions
 such as path spaces, mapping cylinders, etc.
 As stated in the introduction the solution is to embed $k$-schemes
 into an appropriate \emph{topos}, that is using sheaves for an appropriate
 Grothendieck topology.

For motivic homotopy theory, the Zariski topology is not strong enough.
 Indeed, smooth $k$-varieties look like $\AA^n_k$ only \'etale locally.
 On the other hand, the \'etale topology is too strong: algebraic K-theory,
 as well as motivic cohomology, do not satisfy \'etale descent.\footnote{For motivic cohomology,
 the failure of \'etale descent is quantified by the Beilinson-Lichtenbaum conjecture,
 now a theorem of Voevodsky. See \cite[Conj. 1.16]{RiouBK},
 and also Section \ref{sec:hmot-torsion}.}
 The \emph{Nisnevich topology} is a topology intermediate between the Zariski and \'etale topologies that captures some good features
 of both.\footnote{It was introduced by Nisnevich in \cite{Nis}.}
 whose covers are given by the \'etale surjective families $(p_i:X_i \rightarrow X)_{i \in I}$
 such that for any $x \in X$, there exists $i \in I$, and $x_i \in X_i$ such that $p_i(x_i)=x$
 and the induced extension of residue fields $\kappa(x_i)/\kappa(x)$ is trivial.

Here are the main advantages of the Nisnevich topology:
\begin{enumerate}
\item the cohomology of a point (spectrum of a field) is trivial;
\item algebraic K-theory does satisfy Nisnevich descent (\cite{TT});
\item a closed immersion $Z \subset X$ of smooth $k$-schemes
 looks Nisnevich-locally like the $0$-section of an affine space.\footnote{i.e.,
 for all $x \in Z$, one gets a (non-canonical) isomorphism $X_{(x)}^h \simeq (\AA^c_{\kappa(x)})_{(0)}^h$,
 where $X_{(x)}^h$ is the Nisnevich localization of $X$ at $x$, in other words, the spectrum
 of the henselization of the local ring of $X$ at $x$.} 
\end{enumerate}
In particular, the Grothendieck site chosen for motivic homotopy theory
 is the category of smooth $k$-schemes $\Sm_k$,
 equipped with the Nisnevich topology.\footnote{The restriction to \emph{smooth} $k$-schemes is primarily justified by point (3).}
 Below we recall for the comfort of the reader a few important aspects of this particular
 Grothendieck topology (see \cite[\textsection 3.1]{MV}).

\subsubsection{Excision property} \label{sec:excision}
A (Nisnevich/elementary) \emph{distinguished square} is a cartesian square of smooth $k$-schemes
$$
\xymatrix@=10pt{
W\ar[r]\ar[d]\ar@{}|\Delta[rd] & V\ar^p[d] \\
U\ar_j[r] & X
}
$$
such that $p$ is \'etale, $Z=X-U$ is seen as a reduced closed subscheme of $X$,
 and the induced morphism $T=p^{-1}(Z) \rightarrow Z$ is an isomorphism.

In fact, the family of morphisms of the form $(p,j)$ attached to a distinguished square
 as above generates the Nisnevich topology.
 In this situation, one deduces by taking quotient a morphism of pointed sheaves,\footnote{Indeed,
 this type of quotients always gives canonically pointed objects
 by the map $*=U/U \rightarrow X/U$ of sheaves of sets.}
$$
X/U=X/(X-Z) \rightarrow V/(V-T)=V/W.
$$
which can be seen to be an isomorphism.
 This is the so-called \emph{excision property}.
 
\begin{exem}
\begin{enumerate}
\item The obvious inclusion induces an isomorphism $\AA^1/\GG \rightarrow \PP^1/\AA^1$ of pointed sheaves.
\item Let $Z \rightarrow X$ be a closed immersion of smooth $k$-schemes.
 Assume there exists an étale morphism $p:X \rightarrow \AA^n_k$ such that $Z=p^{-1}(\AA^{c}_k)$.
 Then there exist an open subscheme $\Omega \subset (X  \times_{\AA^n_k} \AA^{c}_Z)$ containing $Z$ as a closed subscheme
 and such that the obvious projection map induces isomorphism of pointed sheaves:
$$
X/(X-Z) \xleftarrow\sim \Omega/(\Omega-Z) \xrightarrow \sim \AA^{c}_Z/(\AA^{c}_Z-Z).
$$
This property gives another concrete interpretation of point 3. in Section \ref{sec:advantage-Nis}.
\end{enumerate}
\end{exem}

\subsubsection{Points}\label{num:points}
A point of the Nisnevich site $\Sm_k$ is given by a pair $(X,x)$ where $X$ is a smooth $k$-scheme
 and $x \in X$ an element of the underlying set.
 The \emph{fiber} of a presheaf/sheaf $F$ on $\Sm_k$ at the point $(X,x)$ is defined as
$$
F_{X,x}=\ilim{V/X} F(V)
$$
where $V$ ranges over the Nisnevich neighborhoods of $x$ in $X$.\footnote{i.e.,
 the étale $X$-scheme with a point $v \in V$ over $x$
 such that the residual extension $\kappa(v)/\kappa(x)$ is trivial.}
 This defines a \emph{fiber functor} $F \mapsto F_{X,x}$, i.e., commutes with colimits and finite limits.
 The family of these fiber functors is \emph{conservative} on the category
 of sheaves of sets on $\Sm_k$: it preserves and detects isomorphisms (and monomorphisms).

Let $E/k$ be a finitely generated separable extension of $k$.
 Note that $E$ is the filtering colimit of its sub-$k$-algebra $A \subset E$
 with $A/k$ smooth (of finite type).
 Given any (pre)sheaf $F$ over $\Sm_k$, we put:
$$
F(E):=\ilim{A\subset E} F(\spec(A)).
$$
Obviously, given any smooth connected $k$-scheme $X$ with generic point $\eta$,
 and an isomorphism $E \simeq \kappa(X)=\cO_{X,\eta}^h$,
 one gets a canonical isomorphism $F(E) \simeq F_{X,\eta}$,
 so that $F \mapsto F(E)$ is a fiber functor of the Nisnevich site on $\Sm_k$.

\subsection{Tools from higher homotopy theory}\label{sec:tool-htp}

This section can be avoided on a first reading.
 Motivic homotopy relies on two theoretical tools in order
 to study localized categories --- an operation introduced in \cite{GZ} --- 
 along a set of morphisms: model categories and $\infty$-categories.
 Nowadays, model categories are often viewed as presentations
 of $\infty$-categories. The $\infty$-categorical perspective is therefore
 more synthetic and elegant. However, since most of the current literature in motivic homotopy
 theory still uses model categories, we have tried to provide the reader with sufficient tools
 to navigate between the two points of view.
 
\subsubsection{Model categories}
Model categories were introduced in \cite{Quillen},
 as a generalization of the homological algebra of Cartan and Eilenberg,
 and as a synthesis between the homotopy theories of topological spaces
 and simplicial sets (after Daniel Kan).
 They became the central tool for all generalizations of homotopy theory
 in other context than topology.
 This is particularly the case of motivic homotopy,
 which was developed using a particular model category on simplicial Nisnevich sheaves.

Model categories are designed to study categories obtained
 by formally inverting a collection $\mathscr W$ of morphisms, generically called \emph{weak equivalences},
 in a category whose objects are considered as \emph{models}.
 The main idea is to build two kinds of resolutions, called \emph{fibrant} (analogue of injective) and \emph{cofibrant} (analogue of projective),
 out of the so-called small object argument\footnote{which actually originally appeared in \cite{CE}
 to prove the existence of injective resolutions.}
 The central result is that morphisms in the localized category from a cofibrant object to a fibrant object can be computed
 by classes of morphisms in the original category up to an explicit equivalence relation.
 To achieve this, one requires the existence of two collections $\mathcal Fib$ and $\mathcal Cof$
 of morphisms called \emph{fibrations} and \emph{cofibrations}
 respectively. They are bound to satisfy a clear and simple axiomatic that draw inspiration from
 both the properties of injective and projective objects in homological algebra,
 and from the \emph{homotopy lifting property} that characterizes fibrations in topology.\footnote{In fact,
 according to the usual abuse of terminology, we call model category what Quillen defined
 as a \emph{closed model category} in \cite[I.5 Definition 1]{Quillen}.}
 An important tool of model categories is that, under appropriate assumptions, they can be localized,
 by adding more weak equivalences. We will review this procedure below.

\subsubsection{$\infty$-categories}\label{sec:infty-cat}
After developing his theory, Quillen had the major intuition that model categories
 concealed a deeper structure:
\begin{quote}[p. 0.4 of \cite{Quillen}]
Presumably there is a higher order structure ([8], [17]) on the homotopy
 category which forms the part of the homotopy theory of a model category,
 but we have not been able to find an inclusive general definition of this structure
 with the propery that this structure is preserved when there are adjoint functors which 
 establish an equivalence of homotopy theories.
\end{quote}
The said \emph{higher order structure} took many years to be uncovered.
 It is the theory of \textbf{$\infty$-categories}, due to many mathematicians 
 and which was presented at the Bourbaki seminar by \cite{Cisinski}.
 As in \emph{op. cit.}, we take \cite{LurieHTT} as a reference --- therefore,
 we use Joyal's theory of \emph{quasicategories} as a model for $\infty$-categories.
 Here are the important points to keep in mind:
\begin{itemize}
\item Any category admits an associated $\infty$-category via the nerve functor (see \textsection 2 in \cite{Cisinski}).
\item Many (if not all) classical constructions and concepts from category theory extends to $\infty$-categories:
 adjoint functors, equivalences, Kan extensions, limits/colimits, pro/ind-objects,
 (co)fibred category, etc.
\item Every $\infty$-category $\mathscr C$ admits a mapping space functor denoted by
 $\Map_{\mathscr C}(X,Y)$ for objects $X$ and $Y$. In fact, one can view any $\infty$-category
 as a simplicial category; see \textsection 12, and Theorem 12.4, in \cite{Cisinski}.
\end{itemize}

\subsubsection{Localizations}\label{sec:abstract-loc}
In both model categories and $\infty$-categories,
 the fundamental operation is that of localization, the procedure of adding
 invertible morphisms.

Under a suitable assumption, there is a context, seemingly first formulated by Bousfield,
 in which this localization procedure is more structured and better suited for computations.
 This corresponds to the property of being \emph{combinatorial} for model categories, and \emph{presentable}
 for $\infty$-categories.

Let us describe this particular type of localization for $\infty$-categories,
 called \emph{left localization} in \cite{Cisinski}. Let $\iC$ be a presentable
 $\infty$-category and $W$ be a set of morphisms. One introduces the following definitions:
\begin{itemize}
\item an object $X$ of $\iC$ is \emph{$W$-local} if for any $f_0 \in W$,
 the map $\Map_{\iC}(f_0,X)$ is a weak equivalence;
\item a morphism $f$ on $\iC$ is a \emph{$W$-local weak equivalences} (or simply \emph{weak $W$-equivalence})
 if for any $W$-local object $X$, the map $\Map_{\iC}(f,X)$ is a weak equivalence.
\end{itemize}
Letting $\bar W$ be the set of weak $W$-equivalences,
 there exists a functor $\pi:\iC \rightarrow \iC[\bar W^{-1}]$
 of $\infty$-categories\footnote{satisfying the usual universal property, and in particular unique in the $\infty$-categorical
 sense;} which admits a right adjoint $\nu$ which is \emph{fully faithful}
 and with essential image the $W$-local objects. We put $L_W=\nu \circ \pi$ and call it
 the \emph{$W$-localization functor}.

\begin{exem}\label{ex:infty-topos}
The main example for us comes from \emph{$\infty$-topoi} which are left localizations
 of the presentable $\infty$-category of presheaves on some Grothendieck site (see \cite{Cisinski}):
 we will work with the \textbf{$\infty$-topos of Nisnevich sheaves on the site $\Sm_k$ of smooth $k$-schemes}.
\end{exem}

\subsubsection{Monoidal structures}\label{sec:monoid-ifty}
Among the many advantages of $\infty$-categories, it is possible to transport
 all the definitions and constructions from the theory of monoidal categories.
 As $\infty$-categories must always encode all the higher structures in a precise way,
 the theory has its roots in the point of view of operads.
 We refer the reader to the excellent account of \cite[\textsection 3, 4]{GrothI}
 for the definitions of symmetric monoidal $\infty$-category, and commutative monoid
 objects in them.

\section{Unstable motivic homotopy}

\subsection{Homotopy theory of Nisnevich sheaves}

\subsubsection{$k$-Spaces}
Our basic ``homotopical objects'' will be the simplicial presheaves on the category $\Sm_k$:
$$
\cX:(\Sm_k)^{op} \rightarrow \sS,
$$
simply called \emph{$k$-spaces}. The corresponding category,
 with morphisms the natural transformations, is denoted by $\PSh(k,\sS)$.
 Note that it contains as a full subcategory the category of Nisnevich simplicial
 sheaves $\Sh(k,\sS)$. Both categories will serve as \emph{models} for motivic homotopy types.

\begin{exem}
\begin{enumerate}
\item Let $X$ be an arbitrary $k$-schemes.
 Then the (pre)sheaf it represents $X(-)=\Hom(-,X)$ can be seen as a discrete simplicial sheaf,
 and therefore as a (discrete) $k$-space.
 When $X$ is a smooth $k$-scheme, we will abusively denote it by $X$.\footnote{This is harmless because of the Yoneda lemma.
 Beware however that it is different  for singular schemes:
 as an example, if $X$ is a non reduced $k$-scheme, with reduction $X_{red}$,
 the nil-immersion $\nu:X_{red} \rightarrow X$ induces an isomorphism of $k$-spaces
 $\nu_*:X_{red}(-) \rightarrow X(-)$.} 
\item Given an arbitrary simplicial set $E_\bullet$, one can consider
 the constant simplicial presheaf $U/S \mapsto E_\bullet$ and view it as a $k$-space.
 We will abusively denote it by $E_\bullet$.
\end{enumerate}
\end{exem}

\subsubsection{Pointed $k$-spaces}\label{num:pointed&smash}
According to the notation of the previous example,
 the point $*$ seen as a $k$-space coincide with $\spec(k)$.
 It is the final object of the category of $k$-spaces.
 As in topology, a base point of $\cX$ is a map $x:* \rightarrow \cX$,
 and we say that $(\cX,x)$ is a pointed space. 
 Given a $k$-space $\cX$, we also denote by $\cX_+$ the pointed $k$-space
 with a free base point added.
As in topology, one defines a natural symmetric monoidal structure
 on pointed $k$-spaces whose product is the \emph{smash product}:
$$
\cX \wedge \cY=\cX \times \cY/\lbrack(* \times \cY) \cup (\cX \times *)\rbrack.
$$
Similarly, the coproduct in pointed $k$-spaces is denoted by $\vee$,
 and called the \emph{wedge product}.

\begin{exem} (See \cite[\textsection 4]{JT}, \cite[\textsection 4.1, p. 123, 128]{MV}))\label{ex:BG}
Let $G$ be a smooth group scheme over $k$. We define the pointed $k$-space $BG$ as the presheaf:
$$
U \mapsto B\big(G(U)\big)
$$
where $G(U)$ is seen as a group and $B(-)$ denotes the usual simplicial classifying space construction.
 The base point of $BG$ is given by the identity element $e \in G(k)=BG_0(k)$.
 Note that $BG$ is clearly a simplicial sheaf.

In fact, this construction of classifying spaces in a topos appears for the first time in \cite[\textsection 4]{JT},
 and makes sense for any sheaf of groups --- actually any sheaf of groupoids in \emph{loc. cit.}
 This pioneering work of Joyal and Tierney is part of the long story for the quest for $\infty$-stacks.
\end{exem}
 
\subsubsection{Weak equivalences}
As explained in the introduction, one can do homotopy theory using simplicial (pre)sheaves.
 This was originally introduced by \cite{Illusie}. The reader can find a thorough account in \cite{Jardine}.
 Here is the particular case of this theory that is relevant to us.
\begin{defi}\label{df:local-we}
A morphism of $k$-spaces $\cY \rightarrow \cX$ is a (Nisnevich-)local weak equivalence\footnote{Illusie
 says \emph{quasi-isomorphism} in \emph{op. cit.} This terminology has been replaced by the present one
 in the second period of study of simplicial sheaves, after Heller, Joyal and Jardine.}
 if for any Nisnevich point $(X,x)$ on $\Sm_k$, the induced morphisms
 $\cY_{X,x} \rightarrow \cX_{X,x}$
 on the fiber at $(X,x)$ is a weak equivalence of simplicial sets.
\end{defi}
\begin{rema}
A \emph{global} weak equivalence is a morphism of $k$-spaces $\cY \rightarrow \cX$
 such that for all smooth $k$-schemes $U$, the induced morphism $\cY(U) \rightarrow \cX(U)$
 on global sections over $U$ is a weak equivalence. This terminology is due to Jardine. 
\end{rema}

\subsubsection{Homotopy sheaves}\label{sec:local-htp-shv}
Let $\cX$ be a $k$-space (resp. $(\cX,x)$ be a pointed $k$-space). One defines its $0$-th (resp. $n$-th) homotopy sheaf
 as the (Nisnevich) sheaf associated with the presheaf of pointed sets (resp. groups if $n=1$, abelian groups if $n>1$)
$$
U \mapsto \pi_0\big(\cX(U)\big), \text{ resp. } \pi_n\big(\cX(U),x_U\big).
$$
We say that a map of pointed $k$-spaces $f:(\cY,y) \rightarrow (\cX,x)$ is a local weak equivalence
 if it is so after forgetting the base points.
 Then it induces an isomorphisms on homotopy sheaves for any $n \geq 0$
\begin{equation}\label{eq:functoriality_upi}
f_*:\upi_n(\cY,y) \rightarrow \upi_n(\cX,x).
\end{equation}

\begin{rema}
One should be careful, however, that the latter condition is not enough to guarantee that $f$ is a Nisnevich weak equivalence
 in the sense of Definition~\ref{df:local-we}.
 In fact, one must consider all possible choices of base points; see \cite[\textsection 2.2.1]{Illusie}.

Here is a particular case where this issue does not arise.
 We will say that a $k$-space $\cX$ is \emph{locally connected} if $\upi_0(\cX)$ is the constant sheaf.
 Then a pointed morphism $f:(\cY,y) \rightarrow (\cX,x)$ between locally connected $k$-spaces
 is a weak equivalence if and only if the maps \eqref{eq:functoriality_upi} are isomorphisms for all $n>0$.
\end{rema}

\begin{exem}
Let $G$ be a smooth group scheme over $k$. Then, as expected, one deduces from its construction that $BG$
 is locally connected. Moreover, one gets:
$$
\pi_n(BG,e)=\begin{cases}
G & n=1, \\
* & n \neq 1.
\end{cases}
$$
In other words, the $k$-space $BG$ is a ``$K(G,1)$''.
\end{exem}

\subsubsection{Associated homotopy category}
As in \cite[Def. 2.3.5]{Illusie}, one introduces\footnote{see also \emph{op. cit.} Theorem 2.3.6 for an equivalent construction using Gabriel-Zisman calculus of fractions}:
\begin{defi}
We define the \emph{Nisnevich homotopy category} $\htp^\nis(k)$ (resp. \emph{Nisnevich pointed homotopy category} $\htp_*^\nis(k)$)
 as the localization of the category of $k$-spaces (resp. pointed $k$-spaces)
 with respect to the local weak equivalences.
\end{defi}
Morphisms between two $k$-spaces $\cX$ and $\cY$ in $\htp^\nis(k)$ are simply called weak homotopy classes and denoted by
$$
[\cX,\cY]^\nis:=\Hom_{\htp_*^\nis(k)}(\cX,\cY).
$$
We similarly denote by $[(\cX,x),(\cY,y)]^\nis_\bullet$ in the pointed case.

\begin{exem}\label{ex:BG-comput}
\begin{enumerate}
\item Let $X$ and $Y$ be two smooth $k$-schemes. Then the canonical map
 $\Hom_k(X,Y) \rightarrow [X,Y]^\nis$ is a bijection,
 reflecting the fact that $X$ and $Y$ are discrete $k$-spaces. In particular,
 local weak equivalences are insensitive to the geometry of the underlying smooth $k$-schemes.
 For $n>0$, we further get:
$$
[S^n \wedge X_+,Y_+]^\nis=*.
$$
\item Let $G$ be a smooth group scheme over $k$.
 Then for any smooth $k$-scheme $X$ and any integer $n \geq 0$, one gets
 using the notation from \ref{ex:BG}:
$$
[S^n \wedge X_+,BG]^\nis=\begin{cases}
H^1_\nis(X,G) & n=0, \\
G(X) & n=1, \\
0 & n>1.
\end{cases} 
$$
On the first line, the right-hand side denotes the set of Nisnevich-local $G$-torsors on $X$.
 The second point can be reformulated 
 by saying that $\Omega BG=G$, where $\Omega$ is the loop space functor extended to $k$-spaces.
 We refer the reader to \cite[p. 120]{MV} for more details.
\end{enumerate}
\end{exem}

\begin{rema}
In the homotopy category $\htp^\nis(k)$, one can always identify a simplicial presheaf $\cX$ with its
 associated Nisnevich sheaf $a(\cX)$. Indeed, the canonical map
$$
\cX \rightarrow a(\cX)
$$
is readily seen to be a local weak equivalence. This justifies our choice to refer to a $k$-space as a simplicial presheaf,
 although the reader is free to work simplicial (Nisnevich) sheaves instead --- a choice that was
 in fact adopted in \cite{MV}.
\end{rema}

\subsubsection{Computational tools} As mentioned in Section \ref{sec:tool-htp},
 one can enhance the homotopy category $\htp^\nis(k)$ with several model category
 structures,\footnote{\label{fn:model-Nisnevich} One can retain two such model structures:
\begin{itemize}
\item The \emph{Joyal model structure}, equivalently \emph{injective Nisnevich-local model structure} used by Morel
 and Voevodsky: the base category is that of $k$-spaces that are Nisnevich sheaves,
 cofibrations are monomorphisms, weak equivalence are local weak equivalences, and fibrations
 are the maps with the RLP with respect to the acyclic cofibrations.
\item The \emph{Blander model structure}, equivalently \emph{Nisnevich-localized projective model structure}
 first considered in \cite{Blander} in this context:
 the base category is that of $k$-spaces, cofibrations are the maps with the LLP
 with respect to epimorphisms that are term-wise weak-equivalences, fibrations are the maps
 with the LLP with respect to the acyclic cofibrations. 
\end{itemize}
The advantage of the second model structure is that all representable $k$-spaces
 are cofibrant and a $k$-space is fibrant if and only if it is term-wise
 a Kan complex and is Nisnevich excisive in the sense of Definition \ref{df:BG-ppty}.}
 and a canonical $\infty$-categorical structure simply denoted here by $\ihtp^{\nis}(k)$,
 encompassing in the general notion of $\infty$-topos; see \ref{sec:infty-Nis-topos} below.
 One of the advantages of these structures is to
 allow the definition of a (Nisnevich-local) mapping space between $k$-spaces (resp. pointed $k$-spaces)
 $\cX$ and $\cY$:
$$
\Map^\nis(\cX,\cY) \text{ resp.} \Map^\nis_\bullet(\cX,\cY).
$$
It has the important property that:
$$
\pi_0\Map^\nis(\cX,\cY)=[\cX,\cY]^\nis
 \text{ resp. } \pi_i\Map^\nis_\bullet(\cX,\cY)=[S^i \wedge \cX,\cY]^\nis_\bullet.
$$

\begin{rema}
\begin{enumerate}
\item From the model category perspective, the mapping space is obtained by deriving
 an appropriate basic mapping space functor.\footnote{One says that the model category is
 \emph{simplicial}, see \cite[II.2, Definition 2]{Quillen}.}
 For example, with Morel and Voevodsky's
 notation, one puts for any $k$-spaces $\cX$ and $\cY$
$$
\Map^\nis(\cX,\cY)=S(\cX_c,\cY_f)
$$
where $\cX_c$ (resp. $\cY_f$) is a cofibrant (resp. fibrant) resolution of the sheaf associated with
 $\cX$ (resp. $\cY$) for the Nisnevich local injective model structure that they use.
\item From the $\infty$-categorical perspective, the mapping space comes readily
 out of the $\infty$-categorical structure (as explained in Section \ref{sec:infty-cat}).
 Moreover, we can view the $\infty$-category $\ihtp^\nis(k)$ as a simplicial
 category (as any $\infty$-category, see \cite[\textsection 12, Th. 12.4]{Cisinski}).
\end{enumerate}
\end{rema}

\begin{defi}\label{df:BG-ppty}
Let $\cX$ be a $k$-space. After Morel and Voevodsky, we say that $\cX$ is \emph{Nisnevich excisive}\footnote{Morel
 and Voevodsky originally used the term ``B.G.-property'', which appears to have been replaced in later literature
 by the terminology adopted here;}
 if $\cX(\varnothing)$ is contractible\footnote{This will be automatic if $\cX$ is a Nisnevich sheaf.}
 and for any distinguished square $\Delta$ as in Section \ref{sec:excision}, the resulting square of simplicial sets
$$
\xymatrix@=10pt{
\cX(X)\ar^{p^*}[r]\ar_{j^*}[d] & \cX(V)\ar[d] \\
\cX(U)\ar[r] & \cX(W)
}
$$
is homotopy cartesian (i.e., cartesian in the $\infty$-category $\ihtp$).
\end{defi}
This is a kind of ``homotopical excision property''.\footnote{For example,
 one can interpret it by saying that the induced map on the homotopy fibers of the vertical map with respect to any choice
 of base point in $\cX(W)$ is a weak equivalence. Indeed, these homotopy fibers are respectively the derived
 sections of $\cX$ at $X$ with support in $(X-U)$ and at $V$ with support in $(V-W)$.}
 It was first considered by Brown and Gersten, but for the Zariski topology (i.e., $V \rightarrow X$ is an open immersion).
 In practice, it is mostly used by applying the following proposition.
\begin{prop}
Let $\cX$ be a $k$-space which is Nisnevich excisive.
 Then for any smooth $k$-scheme $X$, the canonical map
$$
\cX(X) \rightarrow \Map^\nis(X,\cX)
$$
is a weak equivalence. In particular,
$$
[X,\cX]^\nis=\pi_0\big(\cX(X)\big)
$$
and when $\cX$ is pointed,
$$
[S^i \wedge X_+,\cX]^\nis_\bullet=\pi_i\big(\cX(X)\big).
$$
\end{prop}
\begin{proof}
This can be derived from the Blander model structure on $k$-spaces
 mentioned in footnote \ref{fn:model-Nisnevich} (use \cite{Blander}, Theorem 1.6, Lemmas 1.8 and 4.1).
For an $\infty$-categorical proof, we refer the reader to \cite[Th. 2.9]{LurieDAG11}.
\end{proof}

\begin{exem}
\begin{enumerate}
\item Let $X$ be an eventually singular $k$-scheme. The associated $k$-space $X(-)$
 is obviously Nisnevich excisive, as it is a Nisnevich sheaf.
\item Given a $k$-space $\cX$, the Godement resolution 
 provides a weak equivalence $\cX \rightarrow G(\cX)$,
 where the target is a Nisnevich excisive $k$-space
 (see \cite[Prop. 3.3]{Jardine} or \cite[\textsection 1, 1.66]{MV}).
 We call $G$ the \emph{Godement simplicial resolution functor}.
\end{enumerate}
\end{exem}

\subsubsection{The $\infty$-categorical description}\label{sec:infty-Nis-topos}
We can restate the previous constructions in light of the tools from higher category theory.
 The category $\htp^\nis(k)$ is the homotopy category associated with the $\infty$-topos
 $\ihtp^\nis(k)$ of Nisnevich sheaves over $\Sm_k$.

Let us recall the construction of the latter.
 We start from the $\infty$-category of presheaves on $\Sm_k$,
 $\iPSh(\Sm_k)=\Fun((\Sm_k)^{\op},\iS)$ (see \cite[\textsection 14]{Cisinski}).
 The objects of this category are genuinely $k$-spaces in
 the sense introduced earlier; the only difference from the ordinary category $\PSh(k)$
 is the presence of higher morphisms.\footnote{Recall that there is an $\infty$-functor
 $\mathrm N\PSh(k) \rightarrow \iPSh(k)$, allowing us to view (ordinary) morphisms of $k$-spaces
 as $1$-morphisms in $\iPSh(k)$.}
 The $\infty$-category $\ihtp^\nis(k)$ is obtained by left localization with respect to local weak equivalences
 as defined above. In particular, we have a notion of local objects --- these are simply called ($\infty$-)sheaves ---
 and a localization endofunctor $L_\nis$ of $\iPSh(k)$.
 In fact, a $k$-space $\cX$ is local if and only if it is Nisnevich excisive.
 An explicit model of the localization functor $L_\nis$ is given by the Godement simplicial functor
 described above.

Note finally that we can formulate the excision property in $\ihtp^\nis(k)$
 by saying that any distinguished square $\Delta$ induces a homotopy cartesian\footnote{This really
 means cartesian in the $\infty$-categorical sense. We use this expression to avoid possible confusions
 with the same notion in the underlying category of simplicial sheaves}
 square in $\ihtp^\nis(k)$.
 This amounts to say that the canonical map
$$
V//W \rightarrow X//U
$$
is a local weak equivalence, i.e., an isomorphism in the $\infty$-category $\ihtp^\nis(k)$.
 Here we have denoted by $X//U$ the homotopy cofiber\footnote{This really means the quotient/cokernel
 in the $\infty$-categorical sense. The same remark as in the previous footnote is in order: we use this expression
 to avoid possible confusions.}
 of the map $U \rightarrow X$.

\subsection{$\AA^1$-localization} 

Having prepared the ground in the preceding sections,
 we can now apply Voevodsky's main idea of using the affine line
 as an interval for doing homotopy with algebraic varieties over $k$.
\begin{defi}\label{df:A1-local}
A $k$-space $\cX$ will be called \emph{$\AA^1$-local} if for any smooth $k$-scheme $U$,
 the following map of spaces induced by the obvious projection is a weak equivalence:
$$
p^*:\Map^\nis(U,\cX) \rightarrow \Map^\nis(\AA^1_U,\cX).
$$
We say that a morphism of $k$-spaces $f:\cY \rightarrow \cX$ is
 a \emph{weak $\AA^1$-equivalence} if for any $\AA^1$-local $k$-space $\cZ$,
 the following map of spaces is a weak equivalence:
$$
f^*:\Map^\nis(\cX,\cZ) \rightarrow \Map^\nis(\cY,\cZ).
$$
We define the $\AA^1$-homotopy category $\htp(k)$ over $k$ as the localization
 of the (Nisnevich) homotopy category $\htp^\nis(k)$ of $k$-spaces
 with respect to weak $\AA^1$-equivalences.
 One also denotes by
$$
[\cX,\cY]^{\AA^1}
$$
the morphisms in $\htp(k)$, called weak $\AA^1$-homotopy classes.
 One defines similarly the corresponding pointed category,
 denoted by  $\htp_\bullet(k)$.
\end{defi}
This definition follows the classical pattern of left Bousfield localization.\footnote{This
 is presented in abstract terms in \ref{sec:abstract-loc}.}
 In particular, weak $\AA^1$-equivalences are
 exactly the morphisms of $k$-spaces that become isomorphisms in the
 motivic homotopy category $\htp(k)$.
 It is clear that naive $\AA^1$-equivalences (Definition~\ref{df:naive-A1})
 are a special case.
 We will see in Example \ref{ex:naive&weak-A1} that this inclusion is strict.

\begin{exem}\textbf{$\AA^1$-Local objects.}\label{ex:A1-local}
One considers the notation of Example \ref{ex:BG-comput}.
\begin{enumerate}
\item Let $X$ be a smooth $k$-scheme. One says that $X$ is \emph{$\AA^1$-rigid} if the associated $k$-space
 is $\AA^1$-invariant: for any smooth $k$-scheme $Y$, the application $$\Hom(Y,X) \rightarrow \Hom(\AA^1_Y,X)$$ is a bijection.
 Canonical examples are: $\GG$, abelian varieties, smooth proper curves.

If the $k$-scheme $X$ is $\AA^1$-rigid, one deduces that the $k$-space $X$ is $\AA^1$-local
 --- use Example~\ref{ex:BG-comput}(1).
 One further deduces that for any smooth $k$-scheme $Y$,
 the following canonical application is a bijection:
$$
\Hom(Y,X) \rightarrow [Y,X]^{\AA^1}.
$$
In particular, the $\AA^1$-rigid smooth $k$-schemes are \emph{discrete} from the point of view of motivic homotopy:
 $\AA^1$-weak equivalences between them are exactly the isomorphisms of $k$-schemes.
\item A more interesting example of an $\AA^1$-local $k$-space
 is provided by the classifying $k$-space $B\GG$ associated with the multiplicative group.
 This follows from the above definition and Example \ref{ex:BG-comput}(2), and yields the computation
$$
[S^n \wedge X_+,B\GG]^{\AA^1}=\begin{cases}
\Pic(X) & n=0 \\
\cO(X)^\times & n=1 \\
0 & n>1.
\end{cases}
$$
The attentive reader will recognize familiar motivic cohomology groups.
 In fact, $B\GG$ is the first example of a \emph{motivic Eilenberg-MacLane space}:
 $B\GG=K(\ZZ(1),1)$.
\end{enumerate}
\end{exem}

\subsubsection{Computational tools}
As explained in Section \ref{sec:tool-htp}, the motivic homotopy category
 admits the $\AA^1$-localized versions of the adequate model structures on $k$-spaces.\footnote{Let us be explicit
 and extend footnote \ref{fn:model-Nisnevich} page \pageref{fn:model-Nisnevich}.
 One defines two model structures whose homotopy category is $\htp(k)$:
\begin{itemize}
\item The \emph{$\AA^1$-localized Joyal model structure}:
 the base category is that of $k$-spaces that are Nisnevich sheaves,
 cofibrations are monomorphisms, weak equivalences are weak $\AA^1$-equivalences, and fibrations
 are the maps with the RLP with respect to the $\AA^1$-acyclic cofibrations.
\item The \emph{$\AA^1$-localized Blander model structure}:
 the base category is that of $k$-spaces, cofibrations are the maps with the LLP
 with respect to epimorphisms that are term-wise Nisnevich-local weak-equivalences,
 fibrations are the maps with the RLP with respect to the $\AA^1$-acyclic cofibrations. 
\end{itemize}
The first model structure is the one used in \cite{MV}, and in most of the works on motivic homotopy categories.
 The second model structure has the advantage that representable $k$-spaces are cofibrant objects,
 fibrant $k$-spaces are $k$-spaces (simplicial presheaves) $\cX$
 which are Nisnevich excisive, termwise Kan complexes, and such that $\cX(X) \rightarrow \cX(\AA^1_X)$
 are weak equivalences (of Kan complexes) for each smooth $k$-scheme $X$. Also, the two pairs of adjoint functors,
 $(f^*,f_*)$ and $(p_\sharp,p^*)$ for $p$ smooth, that appear on $S$-spaces when one works with a scheme instead of a field
 are Quillen adjunctions for the second model structure, but not for the first one.}
It also admits a canonical $\infty$-categorical structure denoted by $\ihtp(k)$:
 one considers the localization
 of the Nisnevich $\infty$-topos with respect to weak $\AA^1$-equivalences.
Using any of these enhancements, one deduces the definition of an $\AA^1$-local mapping space simply denoted by
 $\Map^{\AA^1}(\cX,\cY)$. One gets:
$$
[\cX,\cY]^{\AA^1}=\pi_0\Map^{\AA^1}(\cX,\cY).
$$
Moreover, the following principles hold:
\begin{itemize}
\item A morphism between $\AA^1$-local $k$-spaces is a weak $\AA^1$-equivalence
 if and only if it is a local weak equivalence (in the sense of Definition \ref{df:local-we}).
\item For any $k$-space $\cX$, there exists an $\AA^1$-local $k$-space
 $L_{\AA^1} \cX$ and a weak $\AA^1$-equivalence $\cX \rightarrow L_{\AA^1} \cX$
 (which can even be chosen functorially in the model category case).
\end{itemize}
Except for the last statement, these principles follow from the general procedure
 of localization in higher categories (see \ref{sec:abstract-loc}).

\begin{rema}
 In fact, Morel and Voevodsky provide a more concrete construction
 of the $\AA^1$-localization functor using the functor $\mathrm{Sing}_*^{\AA^1}$
 of ``Suslin's singular chains''
 (see \cite[Lem. 3.20]{MV}, with the addition that for Nisnevich sheaves on $\Sm_k$,
 one may take the countable cardinal).
\end{rema}

\subsubsection{Realizations}\label{sec:real-HA1}
We consider the notation of Section \ref{sec:real}.

We first consider a complex embedding $\sigma:k \rightarrow \CC$.
 Then the canonical functor $X \mapsto X^\sigma(\CC)$ with values in topological spaces,
 admits a canonical extension to $k$-spaces. It can be shown following either \cite[\textsection 5.1]{DIreal},
 or \cite[\textsection A.4]{PPRkth}
 that it sends weak $\AA^1$-equivalences
 to weak equivalences of topological spaces, therefore inducing an $\infty$-functor
 called the \emph{$\sigma$-realization functor}, or complex realization when $\sigma$ is clear:
$$
\rho_\sigma:\ihtp(k) \rightarrow \ihtp
$$
For $\sigma=\Id_\CC$, we simply write $\rho_\CC$.

Next we consider a real embedding $\sigma:k \rightarrow \RR$. 
 Given a smooth $k$-scheme, instead of looking at the real points $X^\sigma(\RR)$,
 it is more accurate to look at the complex points $X^{\sigma}(\CC)$ with the canonical
 action $\nu_X^\sigma$ of $\ZZ/2=\mathrm{Gal}(\CC/\RR)$.
 The canonical functor $X \mapsto (X^\sigma(\CC),\nu_X^\sigma)$ from smooth $k$-schemes to $\ZZ/2$-equivariant topological spaces
 also admits a canonical extension to $k$-spaces, and one shows this extension maps weak $\AA^1$-equivalences
 to weak equivalences of $\ZZ/2$-equivariant topological spaces (see \cite[Th. 5.5]{DIreal}).
 One deduces the \emph{equivariant $\sigma$-realization
 functor}:
$$
\rho^{\ZZ/2}_\sigma:\ihtp(k) \rightarrow \ihtp_{\ZZ/2}
$$
Taking $\ZZ/2$-homotopy fixed points then induces the $\sigma$-realization functor:
$$
\rho_\sigma:\ihtp(k) \rightarrow \ihtp.
$$
In particular, for a smooth $k$-scheme $X$, one gets: $\rho_\sigma(X)=X(\CC)^{h\ZZ/2}=X(\RR)$.
 When $\sigma$ is clear, one simply uses the terminology ``equivariant real'' and ``real realization''.

\begin{exem}
As expected, a smooth $k$-scheme $X$ is said to be \emph{$\AA^1$-contractible}
 if its structural morphism $p$ is a weak $\AA^1$-equivalence.

The study of this particular kind of varieties has been a rich question,
 started by \textcite{AD07} in their pioneering work.
 There are many interesting families of $\AA^1$-contractible smooth
 algebraic varieties. We refer the interested reader to \cite{AsokOst}.
 Let us mention a few examples:
\begin{enumerate}
\item \emph{Open varieties in quadrics}: Consider the affine quadric 
$$
Q_{2n}:\left\{\sum_i x_iy_i=z(1+z)\right\} \subset \AA^{2n+1}
$$
 and the closed subvariety 
$$
E_n:\{x_1=\hdots=x_n=0, z=-1\} \subset Q_{2n}.
$$
Then it was shown by \cite[\textsection 3.1]{ADF} that for all integers $n\geq 1$,
 $X_n=Q_{2n}-E_n$ is a quasi-affine smooth $k$-variety which is $\AA^1$-contractible,
 but not a unipotent quotient of an affine space for $n \geq 3$.\footnote{This works not only over an
 arbitrary field $k$, but even over an arbitrary base.} 
\item \emph{Koras-Russell 3-fold of the first kind}: they are smooth affine hypersurfaces
 defined by an equation of the form:
$$
K_{m,r,s}:\{x^mz=y^r+t^s+x\} \subset \AA^4_\CC
$$
where $m,r,s \geq 2$, and $r$ and $s$ are coprime.
 These complex algebraic varieties are known to be topologically contractible,
 non isomorphic to $\AA^3_\CC$. \cite{DF18} proved they are $\AA^1$-contractible, even working over
 any characteristic $0$ base field $k$.
\end{enumerate}
\end{exem}

Given the realization functors defined in \ref{sec:real-HA1},
 it is known that an $\AA^1$-contractible smooth $k$-variety
 is topologically contractible for all complex and real embeddings of the base field $k$.
 One may naturally ask whether the converse holds.
 This question was resolved in \cite[Th. 1.1]{CRpi0}.
\begin{theo}
Let $X$ be a smooth affine algebraic surface over a field $k$ of characteristic $0$.
 Then $X$ is $\AA^1$-contractible if and only if $X$ is isomorphic to $\AA^2_k$.
\end{theo}
It follows that there are smooth affine complex algebraic surfaces which are topologically contractible
 but not $\AA^1$-contractible: one can consider either a tom Dieck-Petri surface
 or the Ramanujam surface (see \emph{op. cit.} after Th. 1.2).

\begin{rema}
Note that this implies that, even when restricted to the subcategory of compact objects,
 the complex realization functor with source $\htp(\CC)$ is not conservative.
 This is in contrast with Beilinson's well-known conservativity conjecture for
 constructible rational mixed motives.
\end{rema}

\begin{exem}[Motivic spheres and Thom spaces]\label{ex:A1-spheres}
We will always assume that $\PP^1_k$ (resp. $\GG$, $\AA^n-\{0\}$) is pointed by $\infty$ (resp. $1$, $(0,\hdots,0,1)$).
 Given a vector bundle $V$ over a smooth $k$-scheme, one defines the \emph{Thom ($k$-)space of $V$}
 as the quotient\footnote{the base point is given by the canonical map $*=V^\times/V^\times \rightarrow \Th(V)$}
$$
\Th(V):=V/V^\times
$$
where $V^\times \rightarrow V$ is the open immersion of the complement of the zero section.

Then one obtains the following computations in the pointed motivic homotopy category $\htp_\bullet(k)$:
\begin{itemize}
\item $\PP^1 \simeq S^1 \wedge \GG$
\item $\AA^n-\{0\} \simeq S^{n-1} \wedge \GG^n$
\item $\Th(\AA^n_k) \simeq (\PP^1)^{\wedge,n}$
\item $Q_{2n-1}\simeq S^{n-1} \wedge \GG^{\wedge,n}$
\item $Q_{2n}\simeq S^{n} \wedge \GG^{\wedge,n}$
\end{itemize}
The first three isomorphisms are good exercises.\footnote{Hint: use the excision property and the $\AA^1$-contractibility of $\AA^n$.}
 The last two isomorphisms are due to \cite[Th. 2/2.2.5]{ADF}.
 The reader can now check, from the table in Section \ref{sec:real},
 that all these computations are compatible with the real and complex realizations as defined
 in \ref{sec:real-HA1}.
\end{exem}

\begin{exem}
Let us finish with a more advanced example. Assume our base field $k$ is algebraically closed field of characteristic $0$.
 Given an integer $n>0$ and a separable polynomial $P(z)$ of degree $d>0$,
 the associated \emph{Danielewski surface} is defined as the affine hypersurface
$$
D_{n,P}:x^ny=P(z) \subset \AA^3_k.
$$ 
It was shown by Danielewski and Fieseler
 (see \cite{Dub05}) that $D_{n,P}$ is a Zariski-local torsor under a line bundle over the affine line with a $d$-fold origin.
 This implies that the associated $k$-space has the motivic homotopy type of a wedge\footnote{see Section \ref{num:pointed&smash};}
 of spheres:
$$
D_{n,P} \simeq (\PP^1_k)^{\vee,d}.
$$
\end{exem}

\subsection{$\AA^1$-homotopy sheaves}

As in topology, one can encode weak $\AA^1$-equivalences via the appropriate
 notion of homotopy groups --- or rather homotopy sheaves, as in Section~\ref{sec:local-htp-shv}.
\begin{defi}
Let $\cX$ be a $k$-space. One defines its \emph{sheaf of $\AA^1$-connected components}
 $\piA 0(\cX)$ as the (Nisnevich) sheaf of sets on $\Sm_k$ associated with the presheaf:
$$
V \mapsto [V,\cX]^{\AA^1}.
$$
Assume that $\cX$ is pointed. Then one defines for any integer $n>0$
 the \emph{$n$-th $\AA^1$-homotopy sheaf} $\piA n(\cX)$ associated with $\cX$ as the sheaf 
 associated with the presheaf:
$$
V \mapsto [S^n \wedge V_+,\cX]_\bullet^{\AA^1}.
$$ 
\end{defi}
The properties of the $\AA^1$-localization functor imply the following important
 formula, for any $n \geq 0$:
\begin{equation}
\piA n(\cX)=\upi_n(L_{\AA^1}\cX)
\end{equation}
where the right-hand side sheaf was defined in Section \ref{sec:local-htp-shv}.

\begin{exem}\emph{Loop spaces}.--\label{ex:loop}
We recalled the definition of the smash product on pointed $k$-spaces;
 in Section \ref{num:pointed&smash}. This operation corresponds to a closed monoidal;
 structure on the homotopy category $\htp_\bullet(k)$. 

In particular, one gets an internal pointed hom functor, and one can define several loop $k$-spaces,
 associated with our different motivic spheres:
\begin{itemize}
\item simplicial: $\Omega_{S^1}=\uHom_\bullet(S^1,-)$;
\item $\GG$-loops: $\Omega_{\GG}=\uHom_\bullet(\GG,-)$;
\item $\PP^1$-loops: $\Omega_{\PP^1}=\uHom_\bullet(\PP^1_k,-)$.
\end{itemize}
We can iterate as usual these constructions. It follows formally that:
$$
\piA n(\cX) \simeq \piA 0(\Omega^n_{S^1}\cX).
$$
Then it can be shown as in topology that $\Omega_{S^1}\cX$ has an $h$-group structure:
 this gives another proof for the fact that $\piA 1(\cX)$ is a sheaf of groups.
 As is classical in topology,
 one also deduces that for $n>1$, the sheaf $\piA n(\cX)$ has two compatible group
 structures hence is abelian.

Let us also finally mention that one deduces from Example \ref{ex:A1-spheres}
 the identity $\Omega_{\PP^1}=\Omega_{S^1} \Omega_{\GG}$. In particular, any $\PP^1$-loop space
 has an $h$-group structure.
\end{exem}

\begin{rema}
As is typical in homotopy theory,
 it is crucial that the closed monoidal structure on $\htp_\bullet(k)$ corresponding
 to the smash product exists at the model or $\infty$-categorical level
 (see also Section \ref{sec:monoid-ifty}).
 First, this is the only way to get such a structure on the homotopy category.
 Second, it will be used to define the stable motivic homotopy category in Section
 \ref{sec:stable}. For the explicit construction in the $\infty$-categorical case,
 we refer the reader to \cite[\textsection 2.4.2]{Robalo}.
\end{rema}

\begin{defi}\label{df:A1-connected}
Let $n\geq 0$ be an integer
 and $\cX$ be a $k$-space, pointed if $n>0$.

One says that $\cX$ is \emph{$n$-$\AA^1$-connected} if for any $0 \leq i\leq n$, we have $\piA i(\cX)=*$.
 We say \emph{$\AA^1$-connected} for $0$-$\AA^1$-connected
 and \emph{simply $\AA^1$-connected} for $1$-$\AA^1$-connected.
\end{defi}

\begin{exem}\label{ex:unstable-A1htp-shv}
 One easily deduces from Example \ref{ex:A1-local} the following computations.
\begin{enumerate}
\item Let $X$ be a smooth $k$-scheme. Then the following conditions are equivalent:
\begin{enumerate}
\item[(i)] $X$ is $\AA^1$-rigid.
\item[(ii)] The canonical morphism of sheaves $X \rightarrow \upi_0(X)$ is an isomorphism.
\end{enumerate}
In this case, one also deduces that for any base point $x \in X(k)$ and any $n>0$,
 $\piA n(X,x)=*$.
\item $\piA n(B\GG)=\GG$ if $n=1$, and $*$ otherwise.
 This correctly reflects Example~\ref{ex:A1-local}(2).
\item Following Morel, one says that a sheaf of groups $\cG$ is \emph{strongly $\AA^1$-invariant}
 if for any smooth $k$-scheme $X$, both maps
$$
\cG(X) \rightarrow \cG(\AA^1_X) \text{ and } \ H^1(X,\cG) \rightarrow H^1(\AA^1_X,\cG)
$$
are isomorphisms. On the right-hand side, we considered pointed sets
 of isomorphisms classes of Nisnevich-local $\cG$-torsors on $X$.

One can check that the following properties are equivalent:
\begin{enumerate}
\item $\cG$ is strongly $\AA^1$-invariant.
\item The $k$-space $B\cG$ (see Example \ref{ex:BG}) is $\AA^1$-local.
\end{enumerate}
As in the case of $\GG$, when these properties hold, one gets:
 $\piA n(B\cG)=\cG$ if $n=1$, and $*$ otherwise.
\end{enumerate}
\end{exem}

\subsubsection{Unramified sheaves}\label{num:unramif}
 To state the next theorem we introduce the following terminology,
 based on classical considerations but due to Morel.
 Let $F$ be a Zariski sheaf of sets on $\Sm_k$. Given a smooth $k$-scheme $X$ and a point $x \in X$,
 we let 
\begin{equation}\label{eq:extension-shv-etf}
F(\cO_{X,x})=\varinjlim_{x \in U \subset X} F(U)
\end{equation}
be the Zariski fiber of $F$ over $X$ at the point $x$. When $x$ is a generic point with residue field $E$,
 we let $F(E)=F(\cO_{X,x})$.

Let us assume that for any dense open immersion $j:U \rightarrow X$ of smooth $k$-schemes,
 the induced morphism $F(X) \rightarrow F(U)$ is injective. Then when $X$ is connected
 with function field $E$, the set $F(U)$, and similarly $F(\cO_{X,x})$ for any $x \in X$, can be identified with
 a subset of $F(E)$. Using this identification, we say that $F$ is \emph{unramified} when in addition to the preceding
 condition we have for any connected smooth $k$-scheme $X$
$$
F(X)=\bigcap_{x \in X^{(1)}} F(\cO_{X,x}),
$$
where the intersection runs over the set $X^{(1)}$ of points of codimension $1$ in $X$.

Examples of unramified abelian sheaves come from Bloch-Ogus theory \cite[Th. 4.2]{BO},
 and also from Voevodsky's theory of homotopy invariant Zariski sheaves with transfers
 (see \cite[Th. 24.11]{VMW}). 
 The following fundamental structure theorem is due to Morel,
 inspired by the latter example.
\begin{theo}\label{thm:structure-piA}
Let $\cX$ be a pointed $k$-space.
\begin{enumerate}
\item The sheaf of groups $\piA 1(\cX)$ over $\Sm_k$ is unramified and strongly $\AA^1$-invariant.
\item Assume the field $k$ is perfect.
 Then for any $n>1$, the sheaf of abelian groups $\piA n(\cX)$ is unramified and \emph{strictly $\AA^1$-invariant}:
 it has $\AA^1$-invariant cohomology.
\end{enumerate}
\end{theo}
This theorem is due to \cite{MorelLNM}: 
 point (1) is Theorem 6.1, and point (2) would follow from Theorem 5.46.
 We warn the reader that the proof of the latter contains an as yet unproved claim.\footnote{This problem was raised by Niels Feld.
  Here is a detailed summary where, unless explicitly stated, references concern \cite{MorelLNM}:
\begin{enumerate}
\item The existence of well-defined transfers on strongly $\AA^1$-invariant sheaves $F$ is only proven on those of the form $F=M_{-2}$
 (Theorem 4.27, see \cite[Th. 5/6.1.5]{Feld3} for a more precise statement),
 but claimed to exist on sheaves of the form $F=M_{-1}$ in Remark 4.31;
\item Corollary 5.30 uses the existence of transfers on $M_{-1}$, via its reference to Theorem 5.26;
\item Theorem 5.31 uses Corollary 5.30 which cannot be applied in codimension $1$.
\end{enumerate}
By constrast, the proof of Tom Bachmann reduces Theorem 5.30 to the case of $\PP^1_S$ and then uses a well-defined pushforward
 for the projection $\PP^1_S \rightarrow S$.}
 However, two alternative proofs, heavily based on \emph{loc. cit.} but which avoid the problem,
 are given by \cite[Th. 1.6, Cor. 1.8]{BachShv} and by \cite{AyoubNotes}.

\subsubsection{Gersten resolutions}\label{sec:Gersten-resol}
Assume that the base field $k$ is perfect.
It is well-known from \cite{CTHK} that the strict $\AA^1$-invariance of $F=\piA n(\cX)$, $n>1$,
 implies that $F$ admits a Gersten resolution.\footnote{Note that the correct terminology should be
 ``Gersten property'' as it will be recalled in the next footnote that the Gersten resolution of $F$,
 if it exists, is uniquely determined up to unique is omorphism.}
 In particular, both Zariski and Nisnevich cohomologies of a smooth $k$-scheme $X$
 with coefficients in $F$  can be computed in terms of the associated Gersten complex.\footnote{Another way
 of stating this is that the restriction $F|_{X_\zar}$ to the small Zariski site of $X$ is Cohen-Macaulay
 in the sense of \cite[Def. p. 238]{HartRD}. Then the Gersten complex
 restricted to $X_\zar$ is the associated Cousin complex.
 In particular, this complex is unique up to unique isomorphism thanks to \cite[Proposition 2.3]{HartRD}.
 These considerations can be extended to the Nisnevich topology. See \cite[\textsection 4.3]{DFJ}, and also \cite{DKO}
 for further developments.}
 In the case of sheaves of the form $F=\piA n(\cX)$, the Gersten complex of a smooth $k$-scheme $X$
 with coefficients in $F$ even takes the following simpler form and is called the \emph{Rost-Schmid complex} after \cite[Chap. 5]{MorelLNM}:
\begin{equation}\label{eq:Rost-Schmid}
\bigoplus_{x \in X^{(0)}} F(\kappa(x)) \rightarrow \bigoplus_{x \in X^{(1)}} F_{-1}(\kappa(x))\{\nu_x\}
 \rightarrow \hdots \rightarrow \bigoplus_{x \in X^{(d)}} F_{-d}(\kappa(x))\{\nu_x\}
\end{equation}
with the following notation:
\begin{enumerate}
\item $d$ is the dimension of $X$;
\item $X^{(i)}$ denotes the set of points of codimension $i$ in $X$;
\item $F_{-1}$ is \emph{Voevodsky's (-1)-construction}, or \emph{simply contraction the contraction of $F$}:
 $F_{-1}(X)$ is the cokernel of the split monomorphism $F(\GG \times X) \xrightarrow{s_1^*} F(X)$,
 and $F_{-n}$ is the $n$-th iterated application of the $(-1)$-construction;
\item we have used notation \eqref{eq:extension-shv-etf} for the sheaf $F_{-n}$, with $\kappa(x)$ being the residue field of $x$ in $X$;
\item $\nu_x$ is the determinant of the conormal sheaf of the regular closed immersion $\{x\} \rightarrow X_{(x)}=\spec(\cO_{X,x})$;
\item given an invertible line bundle $\cL$ over $X$, $\cL^\times$ being the subsheaf complementary of the zero section, 
 we have adopted the following notation after Morel:
$$
F_{-n}(X)\{\cL\}=F_{-n}(X) \otimes_{\ZZ[\GG(X)]} \ZZ[\cL^\times(X)]
$$
using the natural action of $\GG$ on $F_{-n}$ and on $\cL^\times$.
\end{enumerate}
In fact, the core argument of the previous theorem is to show that
 a strongly $\AA^1$-invariant sheaf of \emph{abelian groups} $F$ admits a Gersten resolution.
 This in turn implies the $\AA^1$-invariance of all its cohomology sheaves.

For a systematic treatment of Rost-Schmid complexes, in the style of cycle modules defined by \cite{Rost},
 we refer the reader to \cite{Feld1}.

\begin{exem}
For the time of writing, the structure of the sheaf of $\AA^1$-connected components $\piA 0(\cX)$ remains mysterious.
 We know it is not $\AA^1$-invariant because of a counter-example due to \cite{AyoubNA1}.
 However, for an $h$-group $\cX$ (same definition as in topology),
 the sheaf $\piA 0(\cX)$ is unramified and $\AA^1$-invariant,
 as proved by \cite[Th. 4.18]{ChouH}.\footnote{See the proof p. 51 for the unramified property.}

Moreover, several fundamental computations in this setting
 were first established in the seminal work of \cite{AsokMorel}.
 Let $X$ be a smooth proper $k$-scheme over a perfect field $k$.
 Applying \emph{loc. cit.} Theorem 6.2.1 and Proposition 6.2.6 (taking into account Remark 2.2.3),
 one gets a canonical bijection for any function field $E/k$:
\begin{equation}\label{eq:R-quiv&piA_0}
X(E)/R \xrightarrow{\ \sim\ } \piA 0(X)(E)
\end{equation}
where $(-/R)$ denotes the quotient by Manin's $R$-equivalence relation.
 This is enough to deduce that if the left-hand side is trivial for any $E/k$,\footnote{in that case, we can say that $X$ is \emph{universally $R$-trivial}.
 Recall that this is implied by the property of being retract $k$-rational or even, universally $\CH_0$-trivial;}
 then $X$ is $\AA^1$-connected --- apply Proposition 2.2.7 of \emph{op. cit.}

Moreover, the isomorphism \eqref{eq:R-quiv&piA_0} was extended to the case where
 $E/k$ is an arbitrary function field of characteristic $0$,
 and $X=G$ is a semisimple, simply connected, isotropic, and absolutely almost simple algebraic group:
 see \cite[Th. 4.2]{BSgps}.\footnote{Recall that in this case,
 according to \cite[Th. 7.2]{GilleB}, the left hand side of \eqref{eq:R-quiv&piA_0} 
 is also isomorphic to the so called Whitehead group of $G$ over $k$.
 In particular, thanks to \emph{loc. cit.} Theorem 8.1, when $k$ is a global field, 
 $G$ is $\AA^1$-connected.}
\end{exem}

\begin{exem}
Let $f:\cX \rightarrow \cY$ be a morphism of $k$-spaces.
 Following \cite[Definition 7.1, Lemma 7.2]{MorelLNM}, one says that $f$ is an $\AA^1$-covering if it has the right lifting property
 with respect to weak $\AA^1$-equivalences.\footnote{Beware that
 if we want to give a meaningful $\infty$-categorical formulation,
 one really has to consider $f$ as a map in the Nisnevich $\infty$-category $\ihtp^\nis(k)$.}
 For a strongly $\AA^1$-invariant sheaf of groups $\cG$,
 any Nisnevich-local $\cG$-torsor is an $\AA^1$-covering: \emph{op. cit.} Lemma 7.5.
 An important theorem of Morel, \emph{op. cit.} Theorem 7.8,
 shows that any pointed $\AA^1$-connected $k$-space $\cX$ admits a \emph{universal $\AA^1$-covering}
$$
\tilde \cX \rightarrow \cX
$$
where $\tilde \cX$ is simply $\AA^1$-connected.
 As in topology, one formally deduces that $$\piA 1(\cX)=\underline{\Aut}_{\cX}(\tilde \cX)$$
 where the right-hand side is the sheaf of automorphisms of $\AA^1$-covers.

As an example, for an integer $n>1$, the canonical $\GG$-torsor
$$
\AA^{n+1}-\{0\} \rightarrow \PP^n
$$
is in fact the universal $\AA^1$-covering of $\PP^n$. One deduces (\emph{op. cit.} Theorem 7.13)
 that, for any integer $n>1$,
$$
\piA 1(\PP^n)=\GG.
$$

By contrast, the $\AA^1$-homotopy sheaf $\piA 1(\PP^1_k)$ is non-abelian.
 In fact, it is the \emph{free strictly $\AA^1$-invariant sheaf} of groups generated
 by the final sheaf of sets $*$ (see \emph{op. cit.} Lemma 7.23).
 We will give a more precise description in Proposition~\ref{prop:Api1-P1} below.
\end{exem}

\subsection{Morel quadratic degree}

\subsubsection{Milnor-Witt K-theory}\label{sec:KMW}
In the theory of motivic complexes, Milnor K-theory\footnote{Recall that this is the tensor $\ZZ$-graded algebra of the $\ZZ$-module $k^\times$
 modulo the \emph{Steinberg relation} $\{u,1-u\}=u \otimes (1-u)=0$ .}
 $\KM_*(k)$ of a field $k$ plays a central role,
 as the extension algebra between Tate twists $\un(n)$. In motivic homotopy,
 we have seen in Cazanave's Theorem \ref{thm:Caza} that inner products play a central role,
 at least visible in the naive $\AA^1$-endomorphism classes of the sphere $(\PP^1_k,\infty)$.

One of Morel's key insights in his analysis of motivic homotopy theory over a field
 was the introduction of a theory that encompasses both theories into what is now
 called the \emph{Milnor-Witt K-theory} of the field $k$;
 see \cite[\textsection 3.1]{MorelLNM}. Morel's definition of the corresponding $\ZZ$-graded
 ring, denoted by $\KMW_*(k)$, is by generators and relation:
\begin{itemize}
\item \underline{Generators}: symbols $[u]$ of degree $+1$ for a unit $u \in k^\times$, and the Hopf symbol $\eta$ of degree $-1$.

Let us write: 
 $[u_1,\hdots,u_n]=[u_1]\hdots[u_n]$, $\qd u=1+\eta.[u]$, $\epsilon=-\qd{-1}$, $h=1-\epsilon$.
\item \underline{Relations}:
\begin{itemize}
\item[(MW1)] $[u,1-u]=0$, $u \neq 0,1$;
\item[(MW2)] $[uv]=[u]+[v]+\eta[u,v]$;
\item[(MW3)] $\eta[u]=[u]\eta$;
\item[(MW4)] $\eta h=0$.
\end{itemize}
\end{itemize}
Morel was inspired by the Milnor conjecture, linking Milnor K-theory and the Witt ring. This definition indeed realizes the synthesis
 between these two theories, as seen by the following computations:
\begin{align*}
\KMW_*(k)\slash(\eta)&=\KM_*(k) \\
\KMW_*(k)[\eta^{-1}]&=\W(k)[\eta,\eta^{-1}].
\end{align*}
See \emph{loc. cit.}\footnote{The first computation is easy,
 and the second one follows from the presentation of $\W(k)$ by generators and relation given in \ref{num:present-MW}.}
 Moreover, one also deduces
$$
\KMW_0(k)=\GW(k)
$$
and in degree $0$, the projection modulo $\eta$ induces the rank map:
\begin{equation}\label{eq:mod-eta}
\GW(k)=\KMW_0(k) \rightarrow \KMW_0(k)/(\eta)=\KM_0(k)=\ZZ.
\end{equation}
The following computation is one of the major results of \cite[Cor. 6.43]{MorelLNM}.
 Recall that the group scheme $\GG$ is one of our motivic spheres, seen as a pointed $k$-scheme by
 the unit element.
\begin{theo}\label{thm:wA1-htp-classes}
Consider the above notation.
 Let $n,i,m,j \in \NN$ be integers such that $n\geq 2$. Then there are canonical isomorphisms:
$$
[S^m \wedge \GG^{\wedge j},S^n \wedge \GG^{\wedge i}]_\bullet^{\AA^1} \simeq \begin{cases}
0 & m<n \text{ or } (m=n, j>0, i=0), \\
\ZZ & m=n, j=i=0, \\
\KMW_{i-j}(k) & (m=n, i>0).
\end{cases}
$$
\end{theo}
The proof makes use of the \emph{motivic version of the Hurewicz theorem}
 (see Remark \ref{rem:Morel-Hurewicz}).

\begin{exem}\label{ex:symbols-HA1}
In fact, one can explicitly identify the generators of the Milnor-Witt ring
 with geometrically defined morphisms via the isomorphism appearing in the preceding theorem.
 First, a unit $u \in k^\times$ defines a morphism of $k$-schemes
 $\gamma_u:\spec(k) \rightarrow \GG$, whose weak $\AA^1$-homotopy class
 corresponds to the element $[u] \in \KMW_1(k)$.

More interestingly, the element $\eta \in \KMW_{-1}(k)$ is sent to the class of the obvious map
$$
\eta:\AA^2_k-\{0\} \rightarrow \PP^1_k.
$$
We call it the \emph{algebraic Hopf map} and denote it (abusively) by $\eta$.
 In fact, one can observe that when $k=\CC$, the topological realization of $\eta$
 is precisely the classical Hopf map $S^3 \rightarrow S^2$.

Another meaningful element for motivic homotopy is the class $\rho:=[-1] \in \KMW_1(k)$.
 We refer the reader to Theorem \ref{thm:Bachamn-real} for an illustration.
\end{exem}

\begin{rema}\label{rem:DM-modulo-eta}
Let us anticipate what follows by stating how this computation relates to motivic cohomology.
 Taking realization in the category of motivic complexes, one gets a canonical map:
$$
[S^m \wedge \GG^{\wedge j},S^n \wedge \GG^{\wedge i}]_\bullet^{\AA^1}\rightarrow \HMot^{n+i-m-j,i-j}(k)
$$
where the right-hand side is the motivic cohomology of the field $k$. In particular, when $n=m$,
 we get a map
$$
\KMW_{i-j}(k) \rightarrow \KM_{i-j}(k)
$$
which is in fact the projection modulo $\eta$, and in particular, the rank map \eqref{eq:mod-eta} when $i=j$.
 In other words, motivic complexes do not see the quadratic phenomena of motivic homotopy.

Note also that Morel's computations do not say anything about the range $m>n$, $i=j$,
 which in fact corresponds to the range of the Beilinson-Soulé vanishing conjecture in
 motivic cohomology.
\end{rema}

One of the main applications indicated by Morel is
 a notion of degree in motivic homotopy analogous to the Brouwer
 degree. We will call it the \emph{Morel degree}.
\begin{coro}\label{cor:un-end-spheres}
Consider an integer $n \geq 2$.
 Then there exists isomorphisms of abelian groups:
$$
[(\PP^1_k)^{\wedge n},(\PP^1_k)^{\wedge n}]_\bullet^{\AA^1} \simeq [\AA^n_k-\{0\},\AA^n_k-\{0\}]_\bullet^{\AA^1}
 \simeq \GW(k).
$$
Indeed, the abelian group structure on the first two sets comes from Example~\ref{ex:A1-spheres}:
 it is proved there that both $k$-spaces $(\PP_k^1)^{\wedge n}$ and $\AA^n-\{0\}$
 are at least a $2$-simplicial suspension of another $k$-space, except for
 $\AA^2-\{0\}$ which is an $h$-space.
\end{coro}

\begin{exem}
In particular, each pointed endomorphism $f$ of the $k$-space $\AA^n-\{0\}$ admits a quadratic degree
 $\tdeg(f) \in \GW(k)$.
 As an example, consider a unit $u \in k^\times$ and the associated pointed endomorphism
 $f_u:(t_1,...,t_n) \mapsto (ut_1,t_2,\hdots,t_n)$ of $\AA^n_k-\{0\}$.
 Then $\tdeg(f_u)=\qd{\bar u}$ with the notation of Section~\ref{num:present-MW}.

Let us consider a complex embedding $\sigma:k \rightarrow \CC$.
 Then a general pointed endomorphism $f$ as above induces a pointed continuous map
$$
\rho_\sigma(f):S^{2n-1} \rightarrow S^{2n-1}
$$
to which is associated a Brouwer degree. One gets:
$$
\mathrm{rk}(\tdeg(f))=\deg(\rho_\sigma(f)).
$$
\end{exem}

The case of the projective line can also be considered,
 but it requires more care.
 This is the following major theorem, obtained by the joint efforts of Morel and Cazanave.
 We refer the reader to \cite[Theorem 7.36]{MorelLNM} and \cite[Theorems 3.22, 4.6]{Caza2}
 for the proofs. In addition, it allows to compute both the naive pointed $\AA^1$-homotopy classes
 (Definition \ref{df:pointed-naive}) and the pointed $\AA^1$-homotopy classes
 (Definition \ref{df:A1-local}) in the particular case of endomorphisms of the projective line.
\begin{theo}\label{thm:MorelP1}
The map \eqref{eq:Caza} (page \pageref{eq:Caza}) induces an isomorphism of rings:
$$
[\PP^1_k,\PP^1_k]^{\AA^1}_\bullet \xrightarrow \sim \GW(k) \times_{Q(k)} k^\times
$$
where, on the left-hand side, the addition comes from the fact $\PP^1_k=S^1 \wedge \GG$,
 and the multiplication is induced by the composition of pointed maps.

In particular, the canonical map
$$
[\PP^1_k,\PP^1_k]^N_\bullet \rightarrow [\PP^1_k,\PP^1_k]^{\AA^1}_\bullet
$$
can be interpreted as the canonical morphism to the group completion of the left-hand side, with its canonical additive monoid structure.
\end{theo}

\begin{exem}\label{ex:naive&weak-A1}
When $k=\CC$, the rank map gives the following computation:
\begin{align*}
[\PP^1_k,\PP^1_k]^N_\bullet&=\NN \times \CC^\times; \\
[\PP^1_k,\PP^1_k]^{\AA^1}_\bullet&= \ZZ \times \CC^\times.
\end{align*}

When $k=\RR$, the signature map gives:

\begin{align*}
[\PP^1_k,\PP^1_k]^N_\bullet&=(\NN \times \NN) \times \RR^\times; \\
[\PP^1_k,\PP^1_k]^{\AA^1}_\bullet&=(\ZZ \times \ZZ) \times \RR^\times.
\end{align*}
In fact, whatever the base field $k$ is, 
 naive $\AA^1$-homotopy classes and weak $\AA^1$-homotopy classes of endomorphisms of the pointed $k$-space $\PP^1_k$
 \emph{do not coincide}.
\end{exem}

\subsubsection{Unramified Milnor-Witt K-theory and $\AA^1$-homotopy sheaves}\label{sec:unramifKM-W}
Recall that the \emph{$n$-th unramified Milnor K-theory} of a smooth connected $k$-scheme $X$ with function field $E$ is defined as the kernel:
$$
\uKM_n(X)=\Ker\big(\KM_n(E) \xrightarrow d \oplus_{x \in X^{(1)}} \KM_{n-1}(\kappa(x))\big)
$$
where $d$ is the sum of the residue maps associated with the discrete valuation $v_x$ on $E$ of a codimension $1$ point $x \in X$.
 This defines an unramified and strictly $\AA^1$-invariant (see e. g. Section~\ref{num:unramif}) abelian Nisnevich sheaf $\uKM_n$ on $\Sm_k$
 (see \cite[Prop. 6.9]{Deg4} applied to the cycle module $\KM_*$).

Similarly, one can define following Morel the \emph{$n$-th unramified Milnor-Witt K-theory} of a smooth $k$-scheme $X$ as the kernel:
$$
\uKMW_n(X)=\Ker\big(\KMW_n(E) \xrightarrow d \oplus_{x \in X^{(1)}} \KMW_{n-1}(\kappa(x))\{\nu_x\}\big)
$$
where $d$ is again the sum of some residue maps --- we have used the notation of Section \ref{sec:Gersten-resol}.
 Again $\uKMW_n$ is an abelian unramified and strictly $\AA^1$-invariant Nisnevich sheaf on $\Sm_k$
 (see \cite[\textsection 3.2]{MorelLNM}, or apply \cite{Feld2}, Theorem 4.1.7 to the Milnor-Witt cycle module
 $\KMW_*$). These considerations allow one to state the previous computations in terms
 of homotopy sheaves. Here is a fundamental example.

\begin{exem}\label{ex:piA-spheres}
For $n \geq 2$, we can now give more meaning to the assertion that $\AA^n-\{0\}$ is a motivic sphere,
 suggested in Section \ref{sec:real}.
 In fact, Theorem \ref{thm:wA1-htp-classes} implies it is $(n-1)$-$\AA^1$-connected
 in the sense of Definition \ref{df:A1-connected}, which is reflected by its complex realization. 
 Moreover, one can deduce from Corollary \ref{cor:un-end-spheres} (or see directly \cite[Rem. 6.42]{MorelLNM})
 that its first non-trivial homotopy sheaf is
$$
\piA {n-1}(\AA^{n}-\{0\})=\uKMW_{n}.
$$
The same result hold for $(\PP^1)^{\wedge,n}$.
\end{exem}

We end up this section with a beautiful computation, again due to Morel,
 that we state separately.
\begin{prop}\label{prop:Api1-P1}
There exists a weak $\AA^1$-equivalence of $k$-spaces:
$$
B\GG \simeq \PP^\infty_k=\varinjlim_{n \in \NN} \PP^n_k.
$$
Moreover, the obvious sequence of $k$-spaces:
$$
(\AA^2-\{0\}) \rightarrow \PP^1_k \xrightarrow p \PP^\infty_k
$$
is an $\AA^1$-fiber sequence, in the sense that the first $k$-space is the homotopy fiber of the map $p$.
 Applying the functor $\piA 1$, one deduces a short exact sequence of strongly $\AA^1$-invariant sheaves of groups:
$$
0 \rightarrow \uKMW_2 \rightarrow \piA 1(\PP^1_k) \rightarrow \GG \rightarrow 0.
$$
This is in fact a central extension.
\end{prop}
We have applied the previous example for the first sheaf 
 and Example \ref{ex:unstable-A1htp-shv} for the third one.
 For the rest, see \cite{MorelLNM}, Theorem 7.29 taking into account Lemma 7.23
 and the remark that follows the proof.

\subsection{K-theory and vector bundles}

We have already seen with the case of the Picard group how classifying spaces can be used 
 in motivic homotopy. In the case of vector bundles of higher rank, the situation is more complex.
 Essentially by definition, it is clear that the (Nisnevich) classifying space $\BGL$ of the infinite
 general linear group classifies vector bundles in the simplicial homotopy category, as in topology.
 This basic fact can be extended to motivic homotopy in several directions.
 We will start by looking at the higher K-theory groups as defined by Quillen.
 Here the situation is exceptionally nice, and we have the following fundamental
 result of Morel and Voevodsky. See \cite[\textsection 4, Th. 3.13]{MV}.\footnote{Working relatively, that result is actually 
 proved for a regular scheme $X$. But beware that
 there is a gap in the proof of Morel and Voevodsky that was found by
 \cite[Remark 8.5]{ST}. It can be fixed by applying Theorem 8.2 of
 \emph{op. cit.}}
\begin{theo}\label{thm:unstable-rep-Kth}
For any smooth $k$-scheme $X$ and any pair of non-negative integers $(n,i)$,
 there exists an isomorphism, natural in $X$ with respect to pullbacks:
$$
[S^n \wedge (\PP^1_k)^{\wedge i} \wedge X_+,\ZZ \times \BGL]_\bullet^{\AA^1} \simeq \K_{n}(X)
$$
where $\ZZ \times \BGL$ is the product of the discrete $k$-space $\ZZ$ with the classifying space
 of the infinite general linear group, with base point given by the canonical base point of $\{0\} \times \BGL$.
\end{theo}
This result uses many of the good properties of Quillen's higher K-theory:
 its presentation via the Q-construction,
 its $\AA^1$-invariance over regular schemes also due to Quillen, Thomason-Trobaugh's Nisnevich descent theorem
 and lastly its $\PP^1$-periodicity property (see loc. cit. for details).

On the other hand, the output is truly remarkable as K-theory is representable by a
 very simple $k$-space, $\ZZ\times \BGL$, which is in some sense the\emph{ group completion of the h-monoid
 $\sqcup_n \BGL_n$} (see loc. cit. for more details).

Moreover, one directly reads off the above computation the following form of the classical
 (Bott) $\PP^1$-periodicity property of algebraic K-theory:
$$
\beta:\Omega_{\PP^1_k} (\ZZ \times \BGL) \simeq \ZZ \times \BGL
$$
where the $\PP^1$-loop space was defined in Example \ref{ex:loop}.
 In particular, $\ZZ \times \BGL$ has a structure of infinite $\PP^1$-loop space,
 which also gives a structure of infinite simplicial loop space.
 This gives the commutative $h$-group structure on the motivic pointed $k$-space
 $\ZZ \times \BGL$, which is compatible with the abelian group structure
 on K-theory via the above isomorphism.

\begin{exem}
In his PhD thesis, \cite{RiouOper} used the above representability result
 to deduce that all operations
 ($\lambda$-operations, Adams operations, products)
 on the functor $\K_0$ lift to operations on the $k$-space $\ZZ\times \BGL$
 seen in the $\AA^1$-homotopy category $\htp(k)$. 
\end{exem}

\begin{rema}\label{rem:vb&BGL}
Building on Example~\ref{ex:A1-local}(2),
 the above theorem was later extended to the so called \emph{unstable case}.
 Let us fix a smooth affine $k$-scheme $X$ and an integer $n>0$.
 Denote by $\Phi_n(X)$ the set of isomorphism classes of rank-$n$
 vector bundles over $X$. Recall that one has  a canonical isomorphism:
$$
\Phi_n(X)\simeq H^1(X,\GL_n),
$$
where the right-hand side is the set of Zariski (or equivalently, Nisnevich) 
$\GL_n$-torsors over $X$.
 
Then, one gets an isomorphism:
\begin{equation}\label{eq:vb&BGL}
[X,\BGL_n]^{\AA^1} \rightarrow \Phi_n(X), f \mapsto f^*(\gamma_n)
\end{equation}
where $\gamma_n$ is the universal $\GL_n$-torsor, and $f^*(\gamma_n)$
 is a well-defined pullback along a map $f:X \rightarrow \BGL_n$
 in the $\AA^1$-homotopy category.
 This result was first proved by \cite[Th. 1.29]{MorelLNM}, excluding the case $n=2$,
 and later extended by \cite[Th. 1]{AHW}.\footnote{To connect with the latter reference one uses
 the canonical $\AA^1$-weak $\BGL_n \simeq \mathrm{Gr}_n$,
 where the right-hand side is the infinite Grassmanian of $n$-dimensional subvector spaces. 
 In fact, the result is proved over an arbitrary smooth scheme over a Dedekind domain
 with perfect residue fields.}
\end{rema}

\subsubsection{Symmetric vector bundles and classifying spaces}
From what we have already seen, inner products and quadratic invariants
 are central in motivic homotopy theory, in particular via the Grothendieck-Witt group
 of a field. In fact, it is classical that symmetric vector bundles over a field,
 or more generally a scheme, can be viewed as torsors under the infinite orthogonal group
 $\Orth$. However this interpretation is only true \'etale locally --- indeed Nisnevich torsors
 over a field are trivial. Thus the (Nisnevich) classifying space $\mathrm{BO}$ that was introduced
 in Example \ref{ex:BG} is not adapted for this study. However, it is possible to define another $k$-space
$$
\BetO=L_{\et} \mathrm{BO}
$$
that represents \'etale local torsors in the simplicial Nisnevich category:
 it is the localization of $\mathrm{BO}$ with respect to \'etale hypercovers.\footnote{The associated
 hypercomplete \'etale sheaf, viewed in the Nisnevich $\infty$-topos, in the terminology of \cite{LurieHTT}.}
 By construction, we have for any smooth $k$-scheme $X$ an isomorphism in the homotopy category $\htp^\nis(k)$:
$$
[X,\BetO]^\nis=H^1_\et(X,\Orth).
$$
Note that $\GW(k)$ is the group completion of the latter pointed set in the particular case $X=\spec(k)$,
 with its natural monoid structure.

\subsubsection{Hermitian K-theory}\label{sec:hermit-K-th}
In fact, one can extend the Grothendieck-Witt groups of a field along lines similar to those used in K-theory.
 This is a very rich topic, which was started independently by Karoubi (hermitian K-theory) and Ranicki (L-theory).
 In the next statement, we will use the construction that was finally done by Hornbostel and Schlichting,
 at the price of assuming that $\car(k) \neq 2$.\footnote{In fact, there is now a general construction valid
 even in characteristic $2$, \cite{CHN}. See also Example \ref{ex:mot-spectra}(4).}
 According to \cite{SchHK}, Definition 9.1 (see also Proposition 9.3),
 given a line bundle $L/X$,  one defines bigraded higher Grothendieck-Witt groups $\GW_n^{[i]}(X,L)$
 which are contravariantly functorial --- when $L=\cO_X$, we just omit it from the notation.
 They satisfy analogous properties to Quillen's K-theory (contravariant, Nisnevich descent, an appropriate
 localization property). Here are some important distinctive features:
\begin{itemize}
\item \textit{Periodicity}: there exist isomorphisms: $\GW_n^{[i+4]}(X) \simeq \GW_n^{[i]}(X)$.
\item When $X$ is affine, the abelian group $\GW_0^{[0]}(X)$ (resp. $\GW_0^{[2]}(X)$) is
 the Grothendieck group of vector bundles\footnote{equivalently, finite rank locally free $\cO_X$-module}
 with a non-degenerate symmetric (resp. symplectic)
 bilinear form.\footnote{When $X$ is not affine, one has to take care about the so-called
 metabolic forms. In general, $\GW_0^{[0]}(X)$ (resp. $\GW_0^{[2]}(X)$) does coincide
 with the definition of \cite{Kneb}.}
\end{itemize}
In particular, $\GW_0^{[0]}(k)$ really is the Grothendieck-Witt group that we have introduced
 earlier. Given the two previous definitions, one can state the quadratic analogue of the previous theorem
 which is due to \cite{ST}.
\begin{theo}\label{thm:unstable-rep-GW}
Let $k$ be a field of characteristic different from $2$.
 Then for any smooth $k$-scheme $X$ and any pair of non-negative integers $(n,i)$,
 there exists an isomorphism, natural in $X$ with respect to pullbacks:
$$
[S^n \wedge (\PP^1_k)^{\wedge i} \wedge X_+,\ZZ \times \BetO]_\bullet^{\AA^1} \simeq \GW_{n}^{[-i]}(X)
$$
where $\ZZ \times \BetO$ is defined as previously.
\end{theo}
The proof is similar to that of the preceding theorem.
 Note that in this setting, the periodicity theorem now takes the form:
$$
\Omega^4_{\PP^1_k} (\ZZ \times \BetO) \simeq \ZZ \times \BetO
$$
which aligns well with the classical topological case.

\begin{rema}
\begin{enumerate}
\item In addition, to both these representability theorems,
 one also has a geometric presentation of the $k$-space
 $\ZZ \times \BGL$ (resp. $\ZZ\times \BetO$) by the infinite grassmannian
 (resp. the infinite orthogonal grassmannian). See the references already mentioned in both cases.
\item Though we now have a definition of hermitian K-theory which has all the expected properties
 in characteristic $2$, thanks to \cite{CHN}, the extension of the previous theorem
 is not clear at the moment.
\end{enumerate}
\end{rema}

\subsubsection{Motivic obstruction theory}\label{sec:obstruction}
Based on Remark~\ref{rem:vb&BGL}, \cite{MorelLNM} started a motivic homotopical study of algebraic vector bundles,
 modeled on the topological case. He sets up a general obstruction theory,
 based on an appropriate notion of Postnikov tower on the motivic homotopy category (\emph{loc. cit.} Appendix B),
 and used it (in the case of $\SL_n$) to define an Euler class for vector bundles with trivial determinant
 which refines the usual top Chern class; \emph{loc. cit.} Theorem~8.14.

As proposed by Morel, this class was later compared by Asok and Fasel to a previous definition due to \cite{BM-CHW}
 and related to the so-called Chow-Witt groups, thoroughly studied in \cite{FaselCHW} and subsequent works.
 This result started a series of works based on Morel's obstruction theory,
 culminating in the proof of a conjecture of Murthy in dimension~$4$ by \cite[Theorem~2]{AF-splitting}. 
\begin{theo}
Over an algebraically closed field of characteristic not $2$,
 a rank $3$ vector bundle over an algebraic smooth affine $4$-fold splits off a trivial line bundle
 if and only if its third Chern class vanishes.
\end{theo}
Let us explain the strategy of the proof. The main tool is the isomorphism \eqref{eq:vb&BGL},
 in the case $n=4$. The authors uses this in conjunction with the following homotopy fiber sequence\footnote{that is,
 the first term of the sequence is the homotopy fiber of the pointed map $\iota_n$ computed
 in $\HA(k)$, or equivalently, the kernel in the $\infty$-categorical sense;}
  in $\HA(k)$:
\begin{equation}\label{eq:Morel-BGL-fiber-seq}
(\AA^n-\{0\}) \rightarrow \BGL_{n-1} \xrightarrow{\iota_{n}} \BGL_{n}
\end{equation}
where $\iota_n$ is the canonical map.\footnote{One deduces this fact from \cite[Prop.~8.11]{MorelLNM}
 and the weak $\AA^1$-equivalence:
 $$\AA^n-\{0\} \simeq \GL_n/\GL_{n-1}.$$}
 Then obstruction theory gives a way to control liftings 
 of weak $\AA^1$-homotopy classes along $\iota_{n}$ in terms of the $\AA^1$-homotopy sheaves of its
 $\AA^1$-homotopy fiber $\AA^n-\{0\}$. Using the isomorphism \eqref{eq:vb&BGL}, $\iota_{n*}$ can be identified with the map
$$
\Phi_{n-1}(X)=[X,\BGL_{n-1}]^{\AA^1} \xrightarrow{\ \iota_{n*}\ } [X,\BGL_{n}]^{\AA^1}=\Phi_{n}(X), [V] \mapsto [V \oplus \AA^1]
$$
so that a rank-$n$  vector bundles $V$ splits off a trivial direct summand if and only if
 its isomorphism class $[V]$ lies in the image of $\iota_{n*}$.
 The main point in the proof of the above theorem is to show that for $n=4$,
 and under the above assumptions, the obstruction to this lifting deduced
 from \eqref{eq:Morel-BGL-fiber-seq} always vanishes.

The method used to prove this vanishing, together with the desire to extend the above result to higher
 dimensions and an initial conjecture of Morel, led the authors
 to the following conjecture in motivic homotopy theory, that stimulated many developments
 in the field.
\begin{conj}\label{conj:AF}
Let $n \geq 4$ be an integer. Then there exists an exact sequence of unramified Nisnevich sheaves (see \ref{num:unramif}):
$$
0 \longrightarrow \uKM_{n+2}/24 \longrightarrow \upi_n^{\AA^1}\!(\AA^n-\{0\}) \longrightarrow \underline{\GW}_{n+1}^{[n]}
$$
where the right-hand side denotes the Nisnevich sheaf associated to the presheaf 
 of higher Hermitian K-theory $X \mapsto \GW_{n+1}^{[n]}(X)$,
 as recalled in \ref{sec:hermit-K-th}.
\end{conj}
This statement complements the computation of Example~\ref{ex:piA-spheres}.
 In the case $n=3$, a similar exact sequence is established by \cite[Theorem 3]{AF-splitting}.

\subsubsection{Towards $\PP^1$-stabilization.}\label{sec:P1-freudenthal}
In the next section, we introduce the important procedure of \emph{$\PP^1$-stabilization}.
 In particular, we will states a know, \emph{$\PP^1$-stable}, version of this conjecture in Theorem~\ref{thm:first-second-smstem}.
 Building on this result, Conjecture~\ref{conj:AF} has been recently settled in characteristic zero
 by \cite[Theorem 7.2.1]{ABH}. The main ingredient of the proof is the so-called
 \emph{$\PP^1$-stabilization theorem}, \emph{loc.~cit.} Theorem~6.2.1, which is the analog
 of the Freudenthal suspension theorem --- see \eqref{eq:stem} for reference --- in motivic homotopy theory with respect to the
 sphere $\PP^1$. Along the lines of \cite{AF-splitting}, the authors also deduce Murthy's conjecture
 in arbitrary dimension over a field of characteristic $0$. We refer the reader
 to \cite{ABH} for further details.

\section{Stable motivic homotopy}\label{sec:stable}

\subsection{The Dold-Kan correspondance}\label{sec:Dold-Kan}

In classical topology, the relations between homotopy and homology
 are governed by two main results:
\begin{itemize}
\item The \emph{Dold-Kan correspondence}: it provides an explicit equivalence of categories between that of simplicial abelian groups
 and homologically non-negative complexes of abelian groups. This allows us to define an adjunction of homotopy/$\infty$ categories:
$$
\ZZ:\ihtp \leftrightarrows \iDer_{\geq 0}(\Ab):K
$$
where the right-hand side category is the subcategory of the derived $\infty$-category
 of abelian groups made by complexes whose homology is non-negative.
 The left-adjoint functor, that we denoted by $\ZZ$ as the \emph{(derived) free abelianisation functor},
 is really the functor which to a space $X$ associates the complex of singular chains $C_*(X,\ZZ)$.
 Its right adjoint $K$ has the classical property that for any abelian group $A$,
 and any integer $n \geq 0$,
$$
K(A,n)=K(A[n])
$$
where the left-hand side is the \emph{$n$-th Eilenberg-MacLane space associated with $A$},
 and, on the right-hand side, $A[n]$ denotes the complex with only one non-trivial term equal to $A$
 in homological degree $n$.\footnote{Note that $K(A,0)$ is simply equal to $A$ with the discrete topology.}
\item The \emph{Hurewicz theorem}: for a simply connected pointed space $X$, it says that $X$
 is $n$-connected if and only if the complex $\ZZ(X)=C_*(X,\ZZ)$ is concentrated in homological degree $>n$.
 That is,
$$
(\forall i \leq n,  \pi_i(X)=0) \Leftrightarrow (\forall i \leq n,  H_i(X,\ZZ)=0)
$$
\end{itemize}
This picture is beautiful, but incomplete.
 First, the homotopy category of the right-hand side is not triangulated; equivalently,
 the underlying $\infty$-category is not stable.
 Each Eilenberg-MacLane spaces $K(\ZZ,n)$ only represents a single homology group:
\begin{equation}\label{eq:KZn-represents}
[X_+,K(\ZZ,n)] \simeq H^n(X,\ZZ).
\end{equation}
Nevertheless, one can assemble the Eilenberg-MacLane spaces into a \emph{spectrum},
 using the natural suspension maps:\footnote{For example, they come out of the universal property
 given by the functorial isomorphism \eqref{eq:KZn-represents}, which gives a unique
 weak equivalence (isomorphism in $\infty$-categorical terms): $\Omega K(\ZZ,n+1)=K(\ZZ,n)$.}
$$
\sigma_n:S^1 \wedge K(\ZZ,n) \rightarrow K(\ZZ,n+1).
$$
This is the so-called Eilenberg-MacLane spectrum $H\ZZ=(\K(\ZZ,n),\sigma_n)_{n \geq 0}$.
Spectra of this kind are the objects of an $\infty$-category $\iSH$,
 called the \emph{stable homotopy category},\footnote{We refer the reader (for example) to \cite[\textsection 1.4]{LurieHA}
 for the construction of this $\infty$-category.}
 which allows us to complete the picture drawn by the Dold-Kan correspondence as follows:
\begin{equation}\label{eq:classical-DK}
\xymatrix@R=28pt@C=35pt{
\ihtp\ar@<2pt>^-{\ZZ}[r]\ar@<2pt>^-{\Sus}[d] & \iDer_{\geq 0}(\Ab)\ar@<2pt>^\iota[d]\ar@<2pt>^-{K}[l] \\
\iSH\ar@<2pt>^-{\ZZ}[r]\ar@<2pt>^-{\Lop}[u] & \iDer(\Ab).\ar@<2pt>^{\tau_{\geq 0}}[u]\ar@<2pt>^-{H}[l]
}
\end{equation}
The adjoint pair $(\Sus,\Lop)$ is made of the infinite suspension spectrum and infinite loop space functors.
 On the right-hand side, $\tau_{\geq0}$ is the truncation functor in \emph{homological notations},
 and $\iota$ the obvious inclusion.

One of the fundamental ideas of motivic homotopy,
 which is certainly the main driving insight of Voevodsky,
 is that this picture should exist in algebraic geometry as well.
 Moreover, the derived $\infty$-category of abelian groups would be replaced by the derived $\infty$-category of motivic complexes
 whose existence was conjectured by Beilinson. Finally,
 motivic cohomology would assume the universal role of singular cohomology.

\subsection{$\PP^1$-stabilization}

\subsubsection{Stabilization}
The classical stable homotopy category
 can be obtained along classical lines with an explicit model category.\footnote{In fact,
 there are two concurrent model categories in order to get a symmetric monoidal model category:
 $S$-modules after Elmendorf, Kriz and May and symmetric spectra after Hovey, Smith and Shipley.}
 Here the $\infty$-categorical framework is much more efficient.
 The $\infty$-category $\iSH$ actually satisfies two universal properties:
\begin{enumerate}
\item It is the universal symmetric monoidal $\infty$-category
 which is the target of a symmetric monoidal functor with source $\ihtp$
 which sends the sphere $S^1$ to a $\otimes$-invertible object.
\item It is the universal $\infty$-category with an exact functor 
 with source $\ihtp_\bullet$ and whose target is stable as an $\infty$-category.
\end{enumerate}
The notion of stability in $\infty$-categories is one of the great advantages
 of this formalism. It replaces the more classical notion of triangulated categories in homological algebra.
 Strikingly, it is actually a property and not an additional structure. Explicitly,
 an $\infty$-category $\iC$ is \emph{stable} if it admits all finite limits and colimits.
 In addition, a square is cartesian if and only if it is cocartesian.
 These two properties suffice to ensure that the homotopy category $\Ho\iC$ admits
 a canonical structure of triangulated category. We refer the reader to the beautiful
 and efficient presentation of \cite[\textsection 1.1]{LurieHA}.\footnote{See
 Proposition 1.1.3.4 of \emph{op. cit.} for the definition we have adopted in this paragraph.}

Let us quickly comment on those two constructions.
 The first one comes from the classical construction of spectra, recall at the beginning.
 The general construction, which consists in adding a $\otimes$-inverse of an object in a presentable monoidal
 $\infty$-category has been first written down in \cite[\textsection 2.1]{Robalo}, to which we refer.
 The second one has been obtained in \cite[\textsection 1.4.2]{LurieHA}:
 and \emph{pointed $\infty$-category with finite limits} admits a universal stabilization
 (see \emph{loc. cit.} for the precise formulation).

As announced earlier, we will follow the same path in motivic homotopy theory.
 The main difference is that we have several possible spheres (see Section \ref{sec:real})
 to choose from when considering the analogue of construction (1).
 The chosen sphere, guided by Beilinson's conjectures on motivic complexes, is the
 projective line $\PP^1_k$.
\begin{defi}\label{df:P1-stabilization}
The motivic stable homotopy $\infty$-category $\iSH(k)$ over $k$
 is the universal $\infty$-category obtained from $\ihtp(k)$ by tensor-inverting
 the $k$-space $\PP^1_k$. This construction is called the \emph{$\PP^1$-stabilization}.
 The objects of $\iSH(k)$ are called \emph{motivic spectra}.

In particular, we have a pair of adjoint functors
$$
\Sus:\ihtp_\bullet(k) \leftrightarrows \iSH(k):\Lop
$$
where $\Sus$ is a monoidal functor such that $\Sus \PP^1$ is $\otimes$-invertible.
\end{defi}
According to the decomposition of Example \ref{ex:A1-spheres},
 both spheres $S^1$ and $\GG$ become $\otimes$-invertible in $\iSH(k)$.
 In particular, the $\infty$-category $\iSH(k)$ is stable.
 We use two conventions for the various powers of spheres:
\begin{equation}
S^{n,i}=\un(i)[n]=(\Sus S^1)^{\otimes n-i} \otimes (\Sus \GG)^{\otimes i}.
\end{equation}

According to the construction of $\iSH(k)$ as the countable limit over 
 the $\PP^1$-loop space functor $\Omega_{\PP^1}$ of Example \ref{ex:loop},
 one deduces for any $k$-spaces $\cX$, $\cY$, the following computation:
\begin{equation}\label{eq:Hom-SH}
[\Sus \cX_+,\Sus \cY_+]_{\SH(k)}:=\varinjlim_n \ [(\PP^1_k)^{\wedge n} \wedge \cX_+,(\PP^1_k)^{\wedge n} \wedge\Sus \cY_+]_\bullet^{\AA^1}
\end{equation}
where the left-hand side stands for morphisms in the homotopy category $\SH(k)=\Ho\iSH(k)$.
 We simply drop the subscript when the context is clear.
 A morphism of pointed $k$-spaces $f:\cX \rightarrow \cY$
 will be called a \emph{stable weak $\AA^1$-equivalence} if it becomes an isomorphism after application of the functor $\Sus$.
 According to the previous computation, this is equivalent  to ask that there exists an integer $n>0$
 such that $(\PP^1_k)^{\wedge n} \wedge f$ is a weak $\AA^1$-equivalence.

\begin{exem}
\begin{enumerate}
\item \cite{WickelgrenDesus} showed that, over the base field $\QQ$,
 $\GG \vee \GG$ and $\PP^1_{\QQ}-\{0,1,\infty\}$ are stably weakly $\AA^1$-equivalent
 but not weakly $\AA^1$-equivalent.
\item A pointed smooth $k$-scheme $X$ is \emph{stably $\AA^1$-contractible}
 if the structural map $X \rightarrow \spec(k)$ is a stable weak $\AA^1$-equivalence.
 \cite{HKO} proved that Koras-Russel threefolds of the second kind,
 for integer positive integers $m,n,\alpha_1,\alpha_2$, such that $\alpha_1, \alpha_2 \geq 2$,
 $n\alpha_1$ and $\alpha_2$ are coprime, and a unit $a \in k^\times$:
$$
K_{m,n,\alpha_i,a}:=\{(x^n+y^{\alpha_1})^mz=t^{\alpha_2}+ax=0\} \subset \AA^4_\CC
$$
are stably weakly $\AA^1$-contractible.
 It is not known whether they are weakly $\AA^1$-contractible or not.
\end{enumerate}
\end{exem}
 
According to the decomposition $\PP^1_k=S^1 \wedge \GG$, both the spheres $S^1$ and $\GG$ become invertible in $\iSH(k)$.
 This implies that the $\infty$-category $\iSH(k)$ is stable, and one can define 
 on a motivic spectrum $\E$ shifts and twists
 by integers $n$ and $i$:
$$
\E(i)[n]:=\E\otimes (\Sus S^1)^{\otimes n-i} \otimes (\Sus \GG)^{\otimes i}.
$$
\begin{defi}
Let $\E$ be a motivic spectrum over $k$.
 One defines the cohomology of a smooth $k$-scheme $X$ in bidegree $(n,i)$ represented by $\E$
 by the formula: 
$$
\E^{n,i}(X)=\left[\Sus X_+,\E(i)[n]\right].
$$
We say that $\E$ is a \emph{ring spectrum} (resp. \emph{$E_\infty$-ring spectrum})
 if $\E$ admits a commutative monoid structure in the homotopy category $\Ho\iSH(k)$
 (resp. is a commutative monoid in the monoidal $\infty$-category $\iSH(k)$).
 In that case, one defines a product on the cohomology theory represented
 by $\E$ as usual.
\end{defi}
All ring spectra appearing in the sequel have in fact a structure of $E_\infty$-ring spectra.
 In particular, we will say abusively ring spectra for $E_\infty$-ring spectra.

\begin{exem}\label{ex:mot-spectra}
\begin{enumerate}
\item Classical cohomology theories in algebraic geometry are all representable by
 ring spectra in $\iSH(k)$:
 in characteristic $0$, integral Betti cohomology (see also below), algebraic de Rham cohomology
 in positive characteristic $p>0$, rigid cohomology
 and in all characteristics integral $\ell$-adic \'etale cohomology ($\ell \in k^\times$).
 With rational coefficients, all these theories are in fact instances of a mixed Weil cohomology
 as defined by \cite{CD2}. Their representability is proved in \emph{loc. cit.} Prop. 2.1.6, Using
 the $\AA^1$-derived category of \ref{num:A1-derived}.
\item Motivic cohomology of smooth $k$-schemes with integral coefficients is represented
 by the so-called \emph{Eilenberg-MacLane motivic spectrum} $\HMot\ZZ$ 
 defined by \cite[\textsection 6.1]{VoeICM}.
 In particular, one gets a computation in terms of Bloch's higher Chow groups:
$$
\HMot^{n,i}(X)=\CH^i(X,2i-n).
$$
 The construction of Beilinson's conjectural motivic cohomology theory
 was in fact Voevodsky's main motivation for introducing motivic homotopy theory.\footnote{Given an arbitrary ring $R$,
 one simply gets the $R$-linear Eilenberg-MacLane motivic spectrum $\HMot R$
 by applying Voevodsky's construction with $R$-linear finite correspondences,
 obtained by naively tensoring with $R$ over $\ZZ$ --- using that the groups of finite correspondences
 are free abelian groups.}
\item Algebraic K-theory (resp. higher Grothendieck-Witt groups)
 is representable by a ring spectrum $\KGL$ (resp. $\sGW$) according to \cite[Th. 15.22]{BHnorms}
 (resp. \cite{CHN}).\footnote{The first proof of this representability result
 (regardless of the $E_\infty$-ring structure)
 was given by \cite[\textsection 5.2]{RiouOper}, based on Theorem \ref{thm:unstable-rep-Kth}.
 For hermitian K-theory, the representability result was first proved by \cite{Hornb},
 in characteristic not $2$ and without the ring structure.}
 According to the periodicity isomorphisms, one has the following canonical identifications:
\begin{align*}
\KGL^{n,i}(X)&=\K_{2i-n}(X) \\
\sGW^{n,i}(X)&=\GW_{2i-n}^{[i]}(X).
\end{align*}

\item Algebraic cobordism is representable by the Thom spectrum $\MGL$,
 as first defined by \cite{VoeICM} (see \cite[Th. 16.19]{BHnorms} for the $E_\infty$-structure).
 \textcite{MLcob} provided a more concrete notion of algebraic cobordism
 which was presented in this seminar by \textcite{LoeserB}.
 In characteristic $0$, it was shown by Levine that the latter theory
 is isomorphic to the $\ZZ$-graded part of Voevodsky's theory: $\MGL^{2*,*}(-)$.
 Levine's argument was conditional on the Hopkins-Morel theorem, which was proved by \cite{HoyoisHM}
 (see also \emph{loc. cit.} Cor. 8.15).
\end{enumerate}
\end{exem}

\begin{rema}\label{rem:GW=KQ}
One also find the name \emph{hermitian K-theory} and the notation $\KQ$
 for the ring spectrum $\GW$ (see Section \ref{sec:veff-slice}).
 Both notations have their advantages.
\end{rema}

\subsubsection{Topological realizations}\label{sec:stable-real}
We consider again the notation of Section \ref{sec:real-HA1}.
 Then one deduces the following realization:
\begin{itemize}
\item Given a complex embedding $\sigma:k \rightarrow \CC$,
 one obtains the $\sigma$-Betti realization
 $\rho_\sigma:\iSH(k) \rightarrow \iSH$, which sends $\Sus X_+$ to $\Sus X(\CC)_+$.
 This is obvious according to \emph{loc. cit.}, as the (unstable) realization
 functor maps the motivic sphere $\PP^1_k$ to $S^2$.
 We refer the interested reader to \cite{AyoubAn} for Betti realizations over a base complex scheme.
\item Let us recall that one defines the genuine stable homotopy category
 of $\ZZ/2$-equivariant homotopy category by $\otimes$-inverting both the simplicial sphere $S^1$
 and the sphere $S^1$ with the antipodal action of $\ZZ/2$, usually denoted by $S^\sigma$.\footnote{The adjective genuine 
 here refers to the fact one has $\otimes$-inverted both the spheres $S^1$ and $S^\sigma$.}
 
Given a real embedding $\sigma:k \rightarrow \RR$, from the unstable realization attached to $\sigma$,
 one deduces the $\sigma$-Betti equivariant realization 
 $\rho^{\ZZ/2}_\sigma:\iSH(k) \rightarrow \iSH_{\ZZ/2}$, which sends $\Sus X_+$ to $\Sus X(\CC)_+$
 with the action of $\ZZ/2$ given by the complex conjugation.
 We refer the reader to \cite[\textsection 4.4]{HeOr} for the explicit construction.

After taking homotopy fixed points (with respect to $\ZZ/2$), one deduces
 the $\sigma$-Betti realization $\rho_\sigma:\iSH(k) \rightarrow \iSH$,
 which sends $\Sus X_+$ to $\Sus X(\RR)_+$. 
\end{itemize}

\begin{exem}\label{ex:stable-real}
In both the complex and real cases, $\rho_\sigma$ is monoidal and admits a right adjoint that we will denote by $\rho_{\sigma*}$.
When $\sigma$ is a complex embedding, and fixing an arbitrary (commutative) ring $R$,
 one deduces a motivic ring spectrum:\footnote{the $E_\infty$-structure comes for free
 from the fact the ($\infty$-)functor $\rho_{\sigma*}$ is weakly monoidal, as a right adjoint of a monoidal ($\infty$-)functor;}
$$
\HB^\sigma R:=\rho_{\sigma*}(\HH R)
$$
which, by the adjunction property, represents the the $R$-linear Betti cohomology over $k$ attached to the embedding $\sigma$.
 When $\sigma=\Id_\CC$, we simply denote by $\HB R$ this ring spectrum.
\end{exem}

\subsubsection{The $\AA^1$-derived category}\label{num:A1-derived}
The Dold-Kan correspondence recalled in Section \ref{sec:Dold-Kan} easily extends to the motivic setting.
 One first considers $\iDer(Sh(k,\ZZ))$,
 the derived $\infty$-category of abelian Nisnevich sheaves on $\Sm_k$.\footnote{This can be simply obtained
 as the localization of the nerve of the category of complexes of such sheaves with respect
 to quasi-isomorphisms, but then the monoidal structure is not obvious.
 A classical way of constructing this monoidal $\infty$-category is by using
 model structures (see e.g. \cite{CD1}). A more elegant procedure
 is to identify $\iDer(\Sh(k,\ZZ))$ with the commutative monoid objects
 of the Nisnevich $\infty$-topos associated the smooth site $\Sm_k$, equipped with its cartesian
 monoidal structure. Then the desired monoidal structure comes from \cite[Ex. 3.2.4.4]{LurieHA}.}
 Then one follows the recipe used to construct the stable motivic homotopy category,
 replacing $X_+$ by the abelian sheaf $\ZZ(X)$ represented by a smooth $k$-scheme $X$:
 one considers its $\AA^1$-localization as in Definition \ref{df:A1-local},
 and then its $\PP^1$-stabilization as in the preceding definition.
 The resulting category is denoted by $\iDA(k)$ and called, after Morel,
 the ($\PP^1$-stable) $\AA^1$-derived category over $k$. The Dold-Kan correspondence
 formally extends as an adjunction:
$$
\ZZ_{\AA^1}:\iSH(k) \leftrightarrows \iDA(k):\uH_{\AA^1}.
$$

\begin{rema}\label{rem:Morel-Hurewicz}
One of the main theorems of \cite{MorelLNM} that we have not yet discussed,
 is the motivic analogue of the Hurewicz theorem, which discusses the properties
 of the composite functor $\ihtp_\bullet(k) \xrightarrow{\Sus} \iSH(k) \xrightarrow{\ZZ_{\AA^1}} \iDA(k)$.
 To avoid inflating indefinitely this presentation, we refer the reader to \emph{loc. cit.}
 Section 6.3. 
\end{rema}

\subsubsection{Motivic complexes}\label{sec:motivic-real}
Let $k$ be a perfect field.
 Voevodsky's theory of motivic complexes was exposed in the Bourbaki seminar by \cite{FriedB}.
 For our needs, we need to consider a slightly bigger category,
 the ``big'' derived $\infty$-category $\iDM(k)$ of mixed motives over $k$.
 As above, it is obtained by applying the $\AA^1$-localization and $\PP^1$-stabilization
 procedure to the the monoidal derived $\infty$-category of the abelian category of sheaves with transfers over $k$
 (\emph{loc. cit.} Section 2).\footnote{This procedure can be realized via
 an explicit monoidal model structure, according to \cite[Th. 11]{ROMZ}
 or \cite[Def. 11.1.1]{CD1}.}

Then one can refine the previous Dold-Kan correspondence
 and introduce the following \emph{motivic realization functor}:
$$
M:\iSH(k) \xrightarrow{\ZZ_{\AA^1}} \iDA(k) \rightarrow \iDM(k).
$$
The last map is derived from the functor which ``adds transfers'' to an
 abelian Nisnevich sheaf.\footnote{see \cite[\textsection 2.2.1]{ROMZ} or \cite[\textsection 11.2.16]{CD3}.}
 The functor $M$ is monoidal, and sends the infinite suspension $\PP^1$-spectrum
 $\Sus X_+$ to the motivic complex $M(X)$.

\bigskip

The following beautiful theorem, due to \cite[Th. 25]{BachMot} is an avatar of the Hurewicz theorem,
 which in some sense strengthen the results of Morel recalled in the above remark.
\begin{theo}
Let $k$ be a perfect field with finite $2$-cohomological dimension.
 Then the motivic realization functor $M:\iSH(k) \rightarrow \iDM(k)$ is conservative
 when restricted to compact motivic spectra, i.e., it detects isomorphisms between such objects.
\end{theo}
Classically in this context, a motivic spectrum $\E$ is compact if the functor $[\E,-]$ commutes with coproducts.
 A nice feature of the Nisnevich topology is that this condition is equivalent to ask that
 $\E$ is in the subcategory of $\iSH(k)$ generated by extensions and finite sums of motivic spectra of the
 form $\Sus X_+(i)[n]$. One says that $\E$ is \emph{constructible} or \emph{geometric}.

\begin{rema}\label{rem:mot-modules}
Formally, the monoidal functor $M$ admits a right adjoint $\uH_M:\iDM(k) \rightarrow \iSH(k)$.
 It follows from the adjunction property that $\uH_M(\un)$ is a motivic spectrum which represents
 motivic cohomology. For an arbitrary ring $R$, one gets the $R$-linear version $\HMot R$ by using 
 the $\infty$-category of $R$-linear motivic complexes.
 In fact, coming back to the underlying model categories, it is exactly the motivic
 spectrum described by \textcite[\textsection 6.1]{VoeICM}.
 The advantage of this presentation is that, as a right adjoint of a monoidal functor,
 $\uH_M$ is weakly monoidal. This immediately implies that $\HMot\ZZ=\uH_M(\un)$ is an $E_\infty$-ring spectrum.

Moreover, this allowed \textcite[see Th. 58, Th. 68]{ROMZ} to build a monoidal functor
 $\HMot\ZZ\!-\!\mathrm{mod} \rightarrow \iDM(k)$ with, as source, the category of modules over the motivic ring spectrum,
 and to deduce that it is an equivalence
 of $\infty$-categories when $k$ is of characteristic $0$, or after tensoring with $\QQ$.
\end{rema}

\subsection{Motivic stable stems and Morel degree}

\subsubsection{Stable stems}\label{sec:stab-stem}
According to the classical terminology, due to \textcite{Freudenthal}
 who used the German word \emph{$n$-Stamm},\footnote{see the first paragraph of \emph{loc. cit.}}
 one defines for an integer $n \geq0$ the \emph{$n$-stem} as the
 $n$-th extension group of the sphere spectrum:
$$
\piS n:=[\Sus S^n,\Sus S^0]_{\SH}=\varinjlim_r [S^{n+r},S^r]_{\htp_\bullet}=\pi_{n+r}(S^r) \text{ for } r\geq n+2.
$$
The last isomorphism follows from Freudenthal suspension theorem, \emph{loc. cit.}
 According to well-known facts from algebraic topology,
 $\piS *$ is a non-negatively graded ring and $\piS 0=\ZZ$.
 Moreover, Serre's finiteness theorem implies that, for $n>0$, $\piS n$ is finite.
\begin{defi}
For any pair of integers $(n,i) \in \ZZ^2$, and any field $k$, 
 the \emph{motivic $(n,i)$-stable stem} over $k$ is
$$
\pi_{n,i}^k:=[S^{n,i},S^{0,0}]_{\SH(k)}.
$$
\end{defi}
The motivic stable stem is a rich invariant of fields. It can be connected
 to other important invariants.
 First, the constant simplicial sheaf functor induces a morphism of $\ZZ$-graded ring
\begin{equation}\label{eq:cl2mot-stem}
\piS * \rightarrow \pi_{*,0}^k.
\end{equation}
Moreover, the motivic realization functor induces a morphism of bigraded rings:
\begin{equation}\label{eq:mot2Kth-stem}
\pi_{*,*}^k \rightarrow \HMot^{-*,-*}(k)
\end{equation}
inverting all indices on the right,
 as we have used homological (cohomological) conventions for the left-hand (resp. right-hand) side.

The determination of the motivic stable stem can be reduced to computations in the (unstable) motivic homotopy
 category according to formula \eqref{eq:Hom-SH}. Therefore, as a corollary of Morel's fundamental
 computation, one knows at least some portion of these groups.
\begin{theo}\label{thm:Morel-stable}
For any field $k$ and all integers $n \geq i$, one gets:
$$
\pi_{n,i}^k=\begin{cases}
\KMW_{-n}(k) & \text{if } n=i, \\
0 & \text{if } n<i.
\end{cases}
$$
\end{theo}
This result should be considered as the analogue of the description
 of the negatively graded part of the stable stem.
 In degree $0$, one gets the fundamental identification: $\pi^k_{0,0}=\GW(k)$.

The morphism \eqref{eq:cl2mot-stem} in degree $0$ is the obvious
 canonical map $\ZZ \rightarrow \GW(k)$. In particular, it is not an isomorphism
 for non-quadratically closed fields.

In fact, one should be careful that the constant sheaf functor,
 left adjoint to the evaluation at the base field $k$,
 which is a fiber functor of the Nisnevich site $\Sm_k$,
 does induce a fully faithful functor $\ihtp \rightarrow \ihtp(k)$.\footnote{We leave this as an exercise
 to the reader.} But the $\PP^1$-stabilization procedure has introduced a non-trivial phenomena
 and the induced functor $c:\iSH \rightarrow \iSH(k)$
 is not fully faithful in general, as seen by the previous example.
 However, one has the following remarkable result of \cite{LevRig}.
\begin{theo}\label{thm:LevinePi}
When $k$ is an algebraically closed field of characteristic $0$,
 the map $\piS * \rightarrow \pi_{*,0}^k$ of \eqref{eq:cl2mot-stem}
 is an isomorphism.
\end{theo}
One deduces that under the above assumption on $k$,
 the above functor $c$ is in fact fully faithful:
 this easily follows as $\iSH$ is generated by the sphere spectrum.
 This theorem is analogous to Suslin's rigidity theorem in algebraic K-theory,
 and in fact, it uses some version of rigidity for Suslin homology, due to Suslin
 and Voevodsky.\footnote{See Section \ref{sec:hmot-torsion} for more information on this point.}
 The core argument of the proof consists in analyzing two slice spectral
 sequences, and most notably, proving their strong convergence.

\subsubsection{Morel's plus-minus decomposition}\label{num:pm}
Identifying $\pi_{0,0}^k$ with $\GW(k)$ and using notation from Sections \ref{num:GW} and \ref{sec:KMW},
 we consider $\epsilon=-\qd{-1}$ as an endomorphism of the unit $\un$ of $\iSH(k)$.
 Then it follows that $\epsilon$ is an idempotent automorphism of $\un$.
 After inverting $2$ in $\iSH(k)$, one deduces, 
 two complementary orthogonal projectors of the object $\un[1/2]$:
$$
p_+=\frac{1-\epsilon}2, \ p_-=\frac{1+\epsilon}2.
$$
This implies that any motivic spectrum $\E$ on which $2$ is invertible decomposes as
\begin{equation}\label{eq:pm}
\E=\E^+ \oplus \E^-,
\end{equation}
where $\E^+$ (resp. $\E^-$) is the image of $p_+$ (resp. $p_-$)
 and called the plus-part (resp. minus-part) of $\E$.
 By construction, the action of $\epsilon$ on $\E^+$
 (resp. $\E^-$) becomes $(+1)$ (resp. $(-1)$).
The following result gives a beautiful interpretation of the rational
 motivic stable stem.
\begin{theo} For an arbitrary field $k$ and all integers $(n,i) \in \ZZ^2$,
 the decomposition \eqref{eq:pm} induces the following identifications:
\begin{align*}
\pi_{n,i}^{k}\otimes \QQ&=\left(\pi_{n,i}^{k+}\otimes \QQ\right),
 \oplus \left(\pi_{n,i}^{k-}\otimes \QQ\right), \\
\pi_{n,i}^{k+}\otimes \QQ& \simeq \HMot^{-n,-i}(k)_\QQ=K_{n-2i}^{(-i)}(k)_\QQ \\
\pi_{n,i}^{k-}\otimes \QQ& \simeq \begin{cases}
W(k)_\QQ & \text{if } n=i, \\
0 & \text{if } n \neq i.
\end{cases}
\end{align*}
The first isomorphism is induced by the map \eqref{eq:mot2Kth-stem},
 while the second one is induced by that of the Theorem \ref{thm:Morel-stable}.
 One deduces that $\pi_{n,i}^{k-}$ is torsion whenever $n \neq i$,
 or in all cases if $(-1)$ is a sum of squares in $k$.\footnote{e.g. if $k$
 is quadratically closed or a field of positive characteristic.}
\end{theo}
For the plus part, we refer the reader to \cite[Theorem 16.2.13]{CD3}. This result was 
 first announced by Morel, and indeed, the proof of \emph{loc. cit.}
 critically rests on Morel's previous theorem, as well as on the homotopy
 $t$-structure (see Section \ref{sec:stab-htp-t}). It also uses the rational motivic spectrum $\KGL$
 and its decomposition under the Adams-operation, as was obtained by \cite{RiouOper}.
 The final argument is a description of the projector on the plus-part
 as taking the homotopy cofiber with respect to $\eta$.

The computation of the minus-part was obtained by \cite[Theorem 5]{ALP}.
 It also uses the homotopy $t$-structure, and the description on the 
 projector on the minus-part by the so-called operation of inverting $\eta$.

\subsubsection{The action of the algebraic Hopf map}
Motivated by the preceding sketch of proof, 
 we show how Morel's plus-minus decomposition
 is related to the algebraic Hopf map.
 Consider a motivic spectrum $\E=\E[1/2]$ as in Section \ref{num:pm}.
 Let us consider the algebraic Hopf map $\eta$, geometrically defined in Section \ref{ex:symbols-HA1},
 as a morphism $\eta:S^{1,1}=\un(1)[1] \rightarrow S^{0,0}=\un$ of motivic spectra.
 Let us formulate some relations which hold in the Milnor-Witt ring of $k$:
\begin{align*}
\epsilon.\eta&=\eta, \\
\rho.\eta&=-1-\epsilon,\text{ where } \rho=[-1].
\end{align*}
 One deduces that the action of $\eta$ on $\E^+$ (resp. $E^-$) becomes trivial (resp. invertible).
 Moreover, one gets the following (homotopy) exact sequence in $\iSH(k)$
 (obtained by tensoring the analogous sequence for $\E=\un[1/2]$):
$$
\E(1)[1] \xrightarrow \eta \E \rightarrow \E^+.
$$
The minus part is obtained by formal inversion of $\eta$
$$
\E^-=\E[\eta^{-1}],
$$
where one defines $\E[\eta^{-1}]$ as the homotopy colimit of the following tower:
$$
\E \xrightarrow{\eta} \E(-1)[-1] \xrightarrow{\eta} \E(-2)[-2] \hdots
$$

\begin{rema}
In fact, the previous theorem can be extended as a computation of the whole rational motivic stable
 $\infty$-category $\iSH(k) \otimes \QQ$. Indeed, the plus-minus decomposition extends at the $\infty$-categorical level,
 and one can identify its plus-part with the $\infty$-category of rational mixed motivic complexes $\iDM(k) \otimes \QQ$
 and its minus-part as the modules over the \emph{rational unramified Witt sheaf}.
 These computations are motivic extensions of the fact that the stable Dold-Kan adjunction $(\ZZ,H)$
 of \eqref{eq:classical-DK} induces an equivalence after rationalization.

We refer the interested reader to \cite[Theorem  16.2.13]{CD3} and \cite[Theorem 7]{ALP}.
 Let us further illustrate these techniques with the following theorem of 
 \cite[Theorem 35 and Proposition 36]{BachRho}.
\end{rema}
\begin{theo}\label{thm:Bachamn-real}
Consider the endomorphism $\rho:\un \rightarrow \un(1)[1]$ as defined above,
 associated with the unit $(-1) \in k^\times$.

The localization procedure with respect to $\rho$, described above, allows to define
 the stable motivic $\infty$-category $\SH(k)[\rho^{-1}]$ of $\rho$-inverted objects.
 Then the real realization functor induces an equivalence of monoidal $\infty$-categories:
$$
\iSH(\RR)[\rho^{-1}] \rightarrow \iSH.
$$
Further, for an arbitrary field $k$, one obtains an equivalence
$$
\iSH(k)[\rho^{-1}] \rightarrow \iSH(k_{r\et}),
$$
where the right-hand side is the $S^1$-stable $\infty$-category associated with Scheiderer real-\'etale site of $k$.
\end{theo}
As a corollary (\emph{loc. cit.} Cor. 42),
 one obtains the following description of the $\rho$-inverted
 motivic stable stem of a real closed field $k$, as a bigraded ring:
$$
\pi_{**}^k[\rho^{-1}] \simeq \piS *\,[\rho^{-1}]\,.
$$
Here the letter $\rho$ on the right-hand side stands for a variable of bidegree $(-1,-1)$.

\subsection{Stable motivic homotopy sheaves}\label{sec:stab-htp-t}
 
Following Morel, as in the unstable case, we adopt the following definition.
\begin{defi}
Let $\E$ be a motivic spectrum over $k$.
 For any integer $n \in \ZZ$, the \emph{$n$-th motivic stable homotopy sheaf}
 of $\E$ is the $\ZZ$-graded (Nisnevich) sheaf $\piAS n(\E)$ on $\Sm_k$ associated
 with the following presheaf
$$
V \mapsto \big[\Sus V_+,\E(r)[r-n]\big], r \in \ZZ.
$$
\end{defi}
If we want to refer to the $\ZZ$-grading of this kind of sheaves, one uses the notation
 $\upi_n^{\AA^1}(\E)_r$. This numbering can seem awkward but it will be justified
 here below.

\begin{exem}
\begin{enumerate}
\item Recall from Section \ref{num:points} that a separated function field $F/k$
 defines a fiber functor of the smooth site on $\Sm_k$.
 By a base change argument, one obtains the following computation:
$$
\piAS n(\un)_r(F)=\pi_{n-r,-r}^F.
$$
In particular, one gets $\piAS 0(\un)_r(F)=\KMW_r(F)$, according to Theorem \ref{thm:Morel-stable}.
\item Let $\E$ be a motivic spectrum, and $\E^{**}$ be the associated bigraded cohomology theory.
 Then, the motivic sheaves $\piAS n(\E)$ have already famously been considered by \cite{BO}, in their extension
 of the Gersten conjecture. More precisely, one has the relation
$$
\piAS n(\E)_r=\underline \E^{r-n,r}
$$
where we have denoted by $\underline \E^{n-r,r}$ what is usually called the \emph{unramified cohomology}
 associated with $\E$ (denoted with curly letters in the first page \emph{loc. cit.}).
 As an example, one deduces that, over any perfect field $k$
$$
\piAS 0(\HMot\ZZ)=\uKM_*,
$$
where the right-hand side is the unramified Milnor K-theory sheaf (see Paragraph \ref{sec:unramifKM-W}).
\end{enumerate}
\end{exem}

\begin{rema}
Consider the notations of second point in the preceding example.
 In general, with the help of the six functors formalism,
 $\E^{**}$ can indeed be extended to a Poincaré duality theory with support in the sense
 of \cite{BO}, except that one has to modify properties (1.3.3) (Fundamental class) and (1.3.4) (Poincaré duality), 
 in order to take into account twists by Thom spaces (defined in Example \ref{ex:A1-spheres}).
 We refer the reader to \cite{DJK} for this extension --- summarized in the notion of (twisted)
 bivariant theory.
 We have already seen the need for Thom spaces, when considering particular forms of the Gersten resolution
 of an unstable motivic homotopy sheaf in Section \ref{sec:Gersten-resol}.
\end{rema}

The following result, analogous to Theorem \ref{thm:structure-piA}, is also due to \cite[Corollary 6.2.9]{MorelStab},
 using \cite{HKGab} when the base field $k$ is finite.
\begin{theo}
Let $k$ be a perfect field.
 Then for any motivic spectrum $\E$ over $k$ and any integer $n \in \ZZ$,
 the $\ZZ$-graded sheaf $\piAS n(\E)$ is unramified and strictly $\AA^1$-invariant
 (as in Theorem \ref{thm:structure-piA}).
\end{theo}
The key point of the proof is the so called \emph{stable $\AA^1$-connectivity theorem},
 which states that the $\AA^1$-localization functor on the $\infty$-category
 of $S^1$-spectra associated with the $\infty$-category of (Nisnevich) sheaves on $\Sm_k$
 (\emph{aka} the stabilization of the Nisnevich $\infty$-topos on $\Sm_k$)
 respects connectivity in the sense of the canonical (Postnikov) $t$-structure.
 See \cite{MorelStab}, Theorem 6.1.8.

\subsubsection{Homotopy modules}
The $\ZZ$-graduation of a stable motivic homotopy sheaf $F_*=\piAS n(\E)_*$ is not arbitrary.
 First, note that one gets a tautological split homotopy exact sequence of motivic spectra:
$$
\Sus \spec(k)_+ \xrightarrow{s_{1*}} \Sus (\GG)_+ \rightarrow \Sus \GG=\un(1)[1].
$$
By definition of Voevodsky's $(-1)$-construction, Section \ref{sec:Gersten-resol}(2),
 this yields a canonical isomorphism for any integer $n \in \ZZ$:
\begin{equation}
\epsilon_n:(F_n)_{-1} \rightarrow F_{n-1}.
\end{equation}
\begin{defi}\label{df-htp-mod}
A \emph{homotopy module} over $k$ is a $\ZZ$-graded Nisnevich sheaf $F_*$ on $\Sm_k$
 which is strictly $\AA^1$-invariant and equipped with an isomorphism $\epsilon_*$
 of the above form.
 Morphisms of homotopy modules are natural transformations of $\ZZ$-graded sheaves,
 homogeneous of degree $0$, compatible with the structural isomorphisms $\epsilon_*$.

We let $\HM(k)$ be the category of homotopy modules over $k$.
 Given a homotopy module $F_*$ and an integer $i \in \ZZ$,
 one defines the $i$-twisted homotopy modules $F_*\tw i$ such that $(F_*\tw i)_r=F_{r+i}$.
\end{defi}

\begin{exem}\label{ex:MW-mod}
Over a perfect field $k$, the fundamental result of \cite{Feld2} defines an equivalence of categories
 which allows us to describe homotopy modules as certain $\ZZ$-graded functors $M_*$ on function fields over $k$ called \emph{Milnor-Witt cycle modules},
 inspired by the theory cycle modules due to \cite{Rost}.\footnote{And later developments due to his PhD student Manfred Schmid
 on the so-called Rost-Schmid complex.}

To such a functor $M_*$ and any smooth $k$-scheme $X$, Feld attaches an explicit complex $C^*(X,M)$
 of $\ZZ$-graded abelian concentrated in non-negative cohomological degrees.
 It has the form described by \eqref{eq:Rost-Schmid}.
 He then shows that the zero cohomology of this complex, which for a connected $X$ can be described as:
$$
\underline M_*(X)=\Ker\big(M_*(\kappa(X)) \rightarrow \oplus_{x \in X^{(1)}} M_*(\kappa(x))\{\nu_x\}\big),
$$
does define a homotopy module over $k$.
 Examples are given by the Milnor-Witt functor $\KM_*$ defined in Section \ref{sec:KMW}, as well as Milnor K-theory $\KM_*$,
 extending the considerations of Paragraph \ref{sec:unramifKM-W}.
\end{exem}

When $k$ is a perfect field,
 stable motivic sheaves are particular instances of homotopy modules.
 In fact, one easily deduces from the preceding theorem and the stable $\AA^1$-connectivity
 theorem the following result of \cite[\textsection 5.2]{MorelIntro}.
\begin{theo}\label{thm:htp-t-struct}
Assume $k$ is perfect.

There exists a unique $t$-structure on $\SH(k)$, called the \emph{homotopy $t$-structure},
 whose homologically non-negative (resp. negative)
 objects are the motivic spectra $\E$ over $k$ such that
$$
\piAS n(\E)=0 \text{ if } n<0 \quad (\text{resp. } n\geq 0).
$$
This $t$-structure is compatible with the monoidal structure on
 $\SH(k)$.\footnote{\label{fn:t-tensor-compatible} The unit object $\un$ is non-negative and 
 the tensor product respect non-negative objects.}

The canonical functor $\piAS 0:\SH(k) \rightarrow \HM(k)$
 induces an equivalence of categories $\SH(k)^\heartsuit \rightarrow \HM(k)$.
 We will denote by $\uH:\HM(k) \rightarrow \SH(k)$ the $\infty$-functor obtained
 by composition of the reciprocal equivalence and the natural inclusion of the heart.
\end{theo}
This implies that $\HM(k)$ is a monoidal Grothendieck abelian category.
 The formula for the tensor product is: $F_* \otimes^H G_*=\piAS 0\big(\uH(F_*) \otimes \uH(G_*)\big)$.

\begin{rema}\label{rem:htp-non-perfect}
The results of this theorem can be extended to the case when $k$ is a non perfect field of positive
 characteristic $p$, up to inverting $p$ in $\iSH(k)$. One can reduce to the perfect case
 by a limit argument and by using Lemma B.3 proved by \cite{LYZR}.
\end{rema}

\begin{exem}\label{ex:piSA}
All this theory shows that the motivic stable stem admits a strong
 structure. First, as a functor on function fields over $k$,
 it can be organized into a Milnor-Witt cycles module in the sense of Example \ref{ex:MW-mod}.
 Moreover, this uniquely corresponds to a homotopy module. This is valid for any
 motivic spectrum in place of the motivic sphere spectrum $\un=\Sus S^0$.
 This connects with the unstable computations in Example \ref{ex:piA-spheres}.
\begin{enumerate}
\item  We deduce from Theorem \ref{thm:Morel-stable} the following fundamental result:
$$
\piAS n(\un)=\begin{cases}
0 & \text{if } n<0, \\
\uKMW_* & \text{if } n=0.
\end{cases}
$$
It follows that $\uKMW_*$ is the unit of the monoidal structure on $\HM(k)$.
 It is therefore a (commutative) monoid object, and every object of $\HM(k)$
 admits a canonical $\uKMW_*$-module structure.
\item  One deduces from elementary vanishing in motivic cohomology that $\HMot \ZZ$ is a non-negative spectrum
 for the homotopy $t$-structure.\footnote{This follows from the fact homotopy modules are unramified
 as $\ZZ$-graded sheaves and from the vanishing $\HMot^{n,i}(k,\ZZ)=0$ if $n>i$,
 which is proved by \cite[Lemma 3.2(2)]{SV-BK}.}
 We have already mentioned that $\piAS 0(\HMot\ZZ)=\uKM_*$.
 Moreover, the unit $o:\un \rightarrow \HMot\ZZ$ of the ring spectrum $\HMot\ZZ$ induces the canonical projection map 
$$\uKMW_*=\piAS 0(\un) \xrightarrow{o_*} \piAS 0(\HMot\ZZ)=\uKM_*=\uKMW_*/(\eta)$$
explaining Remark \ref{rem:DM-modulo-eta}.

More generally, it can be seen
 that the right adjoint $\HMot:\iDM(k) \rightarrow \iSH(k)$ is $t$-exact for Voevodsky's homotopy $t$-structure
 on the left-hand side, and the induced functor on the heart is fully faithful with essential image
 the homotopy module on which $\eta$ acts trivially. Using Feld's equivalence of categories,
 as stated in Example \ref{ex:MW-mod}, the category of homotopy modules with a trivial action of $\eta$ can be identified with
 that of cycle modules, as defined by \cite{Rost}.\footnote{See \cite{Deg10} for all these assertions.}
\item The previous result can be extended to other oriented ring spectra.
 The unit of the cobordism ring spectrum $o:\un \rightarrow \MGL$
 does induce an isomorphism on the $0$-th stable motivic sheaves: $\piAS 0(\MGL) \simeq \uKM_*$.
 In characteristic $0$, this result was extended by \cite[Theorem 3.6.3]{YakMSL} to the special-linear cobordism ring spectrum $\MSL$
 showing that:
$$
\piAS 0(\MSL) \simeq \uKMW_*.
$$
\item The case of algebraic K-theory is different, as $\KGL$ is a $S^{2,1}$-periodic
 motivic ring spectrum. In particular, it is unbounded with respect to the homotopy
  $t$-structure and one has for any integer $i \in \ZZ$
$$
\piAS i(\KGL)=\underline \K_*\tw i,
$$
where the right-hand side is the unramified K-theory sheaf.
\item  Similarly, higher Grothendieck-Witt groups are $S^{8,4}$-periodic and one gets in particular:
$$
\piAS{i+4}(\sGW)=\piAS{i}(\sGW)\{4\}.
$$
\end{enumerate}
\end{exem}

\subsubsection{Slice filtration}\label{num:slice}
Note that the situation described in point 4 of the above example contrasts with the topological context,
 where the complex K-theory spectrum in degree $0$ coincides with the Eilenberg-MacLane
 spectrum, therefore leading to the (topological) Atiyah-Hirzebruch spectral sequence.

This fact led Voevodsky to introduce the \emph{slice filtration} on a motivic spectrum $\E$,
 an avatar of the cellular filtration in the motivic world except that cells are given
 by the motivic sphere $\PP^1_k$:\footnote{beware that this tower is possibly infinite in both directions;}
$$
\hdots \rightarrow f_n\E \rightarrow \hdots \rightarrow f_1\E \rightarrow f_0\E \rightarrow f_{-1}\E \rightarrow \hdots
$$
 He defined the $n$-th \emph{slice} $s_n(\E)$ of $\E$ as the homotopy cofiber of $f_{n+1}\E \rightarrow f_n\E$.

Then, \cite[Conjectures 7, 10]{VoeOpen}
 postulated that  the $0$-slice of both $\KGL$ and the sphere spectrum $\un$ are given by the the Eilenberg-MacLane
 motivic spectrum: 
$$s_0(\un)=\HMot, \ s_n(\KGL)=\HMot(n)[2n].$$
 The first conjecture was proved by Voevodsky in characteristic $0$. 
 Then, \cite{LevTow} proved both conjectures over a perfect infinite base field $k$.\footnote{The assumption
 that $k$ is infinite can be removed thanks to \cite{HKGab}.}
 This lead to a new definition of the motivic version of the Atiyah-Hirzebruch spectral sequence\footnote{whose existence
 was conjectured by \cite[\textsection 5, B.]{BeiHP};} for a smooth $k$-scheme $X$:
 it can be described as the spectral sequence associated to the slice filtration on $\KGL$, as conjectured by Voevodsky and proved
 in \cite{LevTow}. See also \cite[Th. 8.5, 8.7]{HoyoisHM} for another approach.

Note that another construction of the motivic Atiyah-Hirzebruch spectral sequence was already done by \cite{FS-AH}.
 As of now, the agreement of the latter with the slice spectral sequence of the K-theory spectrum $\KGL$
 is not known. Also, the convergence of the slice spectral sequence in general is a difficult matter,
 solved in \cite{LevConv}. See also \cite[Th. 8.12]{HoyoisHM}.

\subsubsection{Morel's $\pi_1$-conjecture and beyond}\label{sec:veff-slice}
Before stating the last main computation of this section,
 let us recall that
 \cite{OSveff} considered the interaction between Morel's homotopy $t$-structure and Voevodsky's slice filtration.
 This led them to introduce an important ring spectrum, called the \emph{very effective hermitian K-theory}.\footnote{Beware
 about Remark \ref{rem:GW=KQ} concerning the terminology used here.}
 It is defined as:
$$
\kq=f_0(\tau_{\geq 0}\KQ)
$$
where $\tau_{\geq 0}$ is the truncation functor with respect to Morel's homotopy $t$-structure and $f_0$ 
 the first stage of the slice filtration. It admits a canonical structure of commutative algebra in the $\infty$-category
 $\iSH(k)$ (i.e., it is an $E_\infty$-spectrum). It is proposed in \emph{loc. cit.} that $\kq$ plays a role analogous
 to the connective cover of real topological $K$-theory in algebraic topology.

The following result of \cite{ORS} emerged from a far-reaching conjecture of Morel,
 and was the motivation for many works in motivic homotopy theory.
 It was stimulated by Conjecture~\ref{conj:AF} stated by Asok and Fasel,
 and utimately led to its proof in characteristic zero; see \ref{sec:P1-freudenthal}.
\begin{theo}\label{thm:first-second-smstem}
Assume the characteristic exponent $c$ of $k$ is different from $2$.
\begin{enumerate}
\item  The unit $u:\un \rightarrow \kq$ of the very effective hermitian K-theory
 induces a short exact sequence of homotopy modules after inverting $e$:
$$
0  \rightarrow \uKM_*/24\tw 2 \rightarrow \piAS 1(\un) \xrightarrow{u_*} \piAS 1(\kq) \rightarrow 0
$$
Looking at the zero-th graded part and evaluating at $k$, the sequence becomes the following explicit
 computation of the motivic stable stem, which was conjectured by Morel:
$$
0 \rightarrow \KM_2(k)/24 \rightarrow \pi_{1,0}^k \rightarrow k^\times/2 \oplus \ZZ/2 \rightarrow 0.
$$
\item After application of the second motivic homotopy sheaf, $u$ induces the following short exact sequence of homotopy modules after inverting $e$:
$$
0  \rightarrow \piAS 1(\HMot\ZZ)_*/24\tw 2 \oplus \uKM_*/2\tw 4 \rightarrow \piAS 2(\un) \xrightarrow{u_*} \piAS 2(\kq).
$$
Looking at the first graded part and evaluating at $k$, one further obtains the split short exact sequence,
 after inverting $e$:
$$
0 \rightarrow \KM_3(k)/2 \rightarrow \pi_{3,1}^k \rightarrow \mu_{24}(k) \rightarrow 0.
$$
\end{enumerate}
\end{theo}
Based on fundamental results on the slice and very effective slice filtrations,
 such as Levine's convergence result for the former, the proof consists of a very involved computation
 of the slice spectral sequences for both the sphere and very effective hermitian K-theory spectra.

\section{Application to stable stems}

\subsection{Computing the stable stems: a short review}\label{sec:histo-stem}

Since the initial introduction of the fundamental group by \cite{Poin},
 and its subsequent extension to higher homotopy groups by his successors,
 the computations of homotopy groups has been a driving force in homotopy theory.
 The case of higher-dimensional spheres remains a central
 problem to this day.

While Poincaré initiated some of the first computations
 (via coverings and polyhedral decompositions of varieties), the main advancements in the first half
 of the 20th century were the following ones:
\begin{itemize}
\item \textcite{Hopf} constructed the \emph{Hopf map}, that we will denote by $\etat$,
 and carried out the first computation of higher homotopy groups: $\pi_3(S^2)=\ZZ.\etat$.
\item \textcite{Hurewicz}
 proved the theorem that now bears his name,
 allowing the computation of higher homotopy groups in terms of homology groups.
 As mentioned by Hurewicz in \emph{op. cit.}, this result implies:
$$
\pi_i(S^n)=\begin{cases}
0 & \text{if } i<n, \\
\ZZ& \text{if }i=n.
\end{cases}
$$
\item For an integer $i \geq 0$, \textcite{Freudenthal} established the theorem that bears his name,
 proving that the suspension map is an isomorphism in the following cases:
\begin{equation}\label{eq:stem}
\pi_{i+n}(S^n) \xrightarrow \sim \pi_{i+n+1}(S^{n+1}) \text{ if } n>i+1.
\end{equation}
This implies the existence of the \emph{$i$-stem},
 which we have already denoted by $\piS i$ in Section~\ref{sec:stab-stem}.
\end{itemize}
A new era was inaugurated by the PhD thesis of \cite{SerrePhD},
 which introduced fibrations (notably the path fibration) and spectral sequences
 in this area. It provided a systematic method for computing homotopy groups of
 spheres and demonstrated the finiteness of most of these groups, particularly the non-zero stable stems.\footnote{Around the same time,
 \cite{WhiteheadEHP1} and \cite{WhiteheadEHP2}, building on Blakers and Massey's notion of triads,
 introduced the \emph{EHP sequence}, which also provides a general inductive method for computing homotopy groups of spheres.}

From this background, a wealth of techniques and theories have emerged,
 partly motivated by the problem of unveiling the stable stem. Let us indicate some landmark
 results, and provide relevant details for this note thereafter.
\begin{itemize}
\item Motivated by the Kervaire invariant problem and the aim of improving the computatibility of Serre's spectral sequences,
 \textcite{AdamsSSp} introduced the Adams spectral sequence,
 which uses the homological algebra of the Steenrod algebra to compute the $p$-torsion
 in the stable stems.
\item \textcite{NovikovSSp} extended Adams's construction to any (appropriate) spectrum in place of the Eilenberg-MacLane spectrum $\sH\FF_p$
 and advocated for the use of the complex cobordism spectrum $\MU$.
\item The final point we wish to highlight is not a single work but rather a collection of results
 and a unifying philosophy,
 now referred to as \emph{chromatic homotopy theory}.\footnote{See also the Bourbaki talks n\textsuperscript{o} 728, 1005, 1029.}
 Numerous contributors have shaped this area,
 with cornerstone by \cite{QuillenFGL}.
 A pioneering work, rooted in Morava's contributions --- including his localization theorem and the introduction of Morava K-theories ---
 was made by \textcite{MilRavWil}, where the term ``chromatic'' was introduced.
\end{itemize}
In order to introduce the reader to the techniques used in the sequel,
 we will now review the main spectral sequences that were mentioned above.
 We will later recall a general construction to obtain all of them, in Example~\ref{ex:AdamsSSq}.

\subsubsection{The Adams spectral sequence}\label{sec:Adams-topo}
 Let us fix a prime $p$, which will be implicit in all subsequent notation.
 The Adams spectral sequence at the prime $p$ is a very efficient tool to compute the $p$-adic completion
 $\hpiS *$ of the stable stems which, according to Serre's fundamental theorem, are given by the formula:
$$
\hpiS i=\piS i \otimes_\ZZ \ZZ_p
=\begin{cases}
\ZZ_p \text{ ($p$-adic integers)} & i=0, \\
(\pi_i^S) \otimes_\ZZ \ZZ_{(p)}=(\pi_i^S)[p^\infty]=(\pi_i^S)_{p-\text{tor}} & i>0, \\
0 & i<0.
\end{cases}
$$
The Adams spectral sequence\footnote{Historically, the construction and results are all due to Adams.
 A classical reference is \cite[III.15]{AdamsBook}.
 The following review is largely based on the excellent account made by \cite{RavenGreen}.}
 is multiplicative and takes the form
\begin{equation}\label{eq:Adams}
E_2^{s,t}=\Ext^{s,t}_{A_\topo}(\FF_p,\FF_p) \Rightarrow \hpiS {t-s}\,.
\end{equation}
Let us clarify the notation.
\begin{itemize}
\item The grading follows the conventions from algebraic topology: differentials on the $E_r$-page 
 have bidegree $(r,r-1)$:
$$
d_r^{s,t}:E_r^{s,t} \rightarrow E_r^{s+r,t+r-1}.
$$
\item $A_\topo$ is the classical Steenrod algebra at the prime $p$,
 made by the cohomological stable operations in (singular) $\FF_p$-cohomology.
  This is a $\ZZ$-graded Hopf algebra over $\FF_p$, which in degree $n$ can be defined as:
$$
A_\topo^n=[\sH\FF_p,S^n \wedge \sH\FF_p]_{\SH}=[\sH\FF_p,\sH\FF_p[n]]_{\SH}.
$$
\item $\Ext^{s,t}_{A_\topo}$ denote the $s$-th extension group computed in the
 category of $\ZZ$-graded $A_\topo$-modules, and the integer $t$ refers to the 
 internal $r$-th grading.
\item The filtration on the abutment has the following form:
$$
\hpiS *=F^1 \hpiS * \supset F^2 \hpiS * \supset \hdots
$$
See \cite{BousLoc}, page 275, for an explicit formula.
Moreover, the $E_2$-page is concentrated in the following region:\footnote{These
 vanishing conditions can be improved in many ways. See \cite[Theorem 1.1]{AdamsPeriod} to begin with.}
$$
E_2^{s,t}=0 \text{ if } \begin{cases} s<0, \text{or } t<s, \\
 \text{or } 0<s<t<2s-3.
\end{cases}
$$
This implies that the above filtration is finite in each degree,
 that the spectral sequence converges (strongly) and that, for
 $r>\max\big(s,\frac 12(t-3s+2)\big)$,
$$
E_\infty^{s,t}=E_r^{s,t}.
$$
\end{itemize}

\subsubsection{Computing with the Adams spectral sequence}\label{sec:cl-computing}
Since its introduction, the above spectral sequence has been a successful tool
 for computing the stable stem, by focusing on $p$-primary parts for each prime $p$.
 There are three major challenges to overcome in this computation:
\begin{enumerate}[label=Step~\arabic*., leftmargin=*]
\item \emph{Determining the $E_2$-term.} It is customary to represent $E_2^{s,t}$
 in a diagram called an \emph{Adams chart},
 where $s$  (the \emph{Adams filtration}) is plotted on the vertical axis, and $f=t-s$ (the \emph{stem})
 on the horizontal axis.
\item \emph{Understanding enough of the differentials on each page to deduce the $E_\infty$-term.}
 An element $a$ in a given page $E_r^{s,t}$ that induces a non-trivial element in the subquotient $E_\infty^{s,t}$
 is called a \emph{permanent cycle}.
\item \emph{Solving the extension problem.} It allows one to reconstruct the whole filtered $\FF_p$-vector space
 from its associated graded (under a finite filtration).
\end{enumerate}
The goal of this section is in particular to explain how motivic homotopy theory has introduced new tools
 that allows improving calculations at each stage (especially steps 2 and 3).
 Since the main difficulty lies in the first and only even prime $p=2$, we will focus on this case.
 To give the reader a sense of the progress achieved, we will state a few general results.
 According to \cite[Th. 3.4.1]{RavenGreen},
 the first four nontrivial rows in the Adams chart are given by the $\FF_2$-algebras with generators and relations
 as follows:
\begin{itemize}
\item $E_2^{0,*}=\FF_2[0]$,
\item $E_2^{1,*}=\FF_2\langle h_i, i \geq 0\rangle$, where $\deg(h_i)=2^i$,
\item $E_2^{2,*}=\FF_2\langle h_ih_j \rangle/(h_ih_j-h_jh_i, h_ih_{i+1})$,
\item $E_2^{3,*}=\FF_2\langle h_ih_jh_k, 0 \leq i<j<k; c_i, i \geq 0\rangle/(h_ih_{i+2}^2,h_i^2h_{i+2})$,
 $\deg(c_i)=11.2^i$.
\end{itemize}
The family of elements $(h_i)_{i \geq 0}$ in $Ext^1_{A_\topo}(\FF_p,\FF_p)$ corresponds to the family of generators $(\Sq_{2^i})_{i \geq 0}$
 of the Steenrod algebra (at the prime $2$). The element $c_i$ is detected by a Massey product $c_i \in \langle h_{i+1},h_i,h_{i+2}^2 \rangle$,
 which is well-defined thanks to the relations satisfied by the family of elements $(h_i)_{i \geq 0}$.

\begin{exem}
As a basic example of resolution in Step 2 above, one knows that 
 $h_0$, $h_1$, $h_2$, $h_3$ are all permanent cycles, and detect respectively $2 \in \hpiS 0$, and the Hopf elements
 $\eta_\topo \in \hpiS 1=\ZZ/2$, $\nu \in \hpiS 3=\ZZ/8$, $\sigma \in \hpiS 7=\ZZ/16$.
 It is a celebrated theorem of \cite{AdamsHopf} that none of the others
 $h_i$ are permanent --- in fact, $d_2^{1,2^i}{h_i}=h_0h_{i-1}^2 \neq 0$ for $i>3$.
\end{exem}

\begin{rema}\label{rem:ssp->E2}
The resolution of the above Step 1 involves the use of other kinds of algebraic spectral sequences,
 which converge to the $E_2$-term to be computed.
 There are many such spectral sequences: Cartan-Serre, May, algebraic Adams-Novikov, chromatic, etc.
 These tools are also used in the computations of \textcite{IWX},
 and will briefly appear at the end of this lecture (see Theorem \ref{thm:tau-cofib}).
 We refer the reader to \cite{RavenGreen} (Chap. 3, \textsection 2 for May spectral sequences,
 Chap. 5 for the chromatic spectral sequence),
 as well as to \textcite{Miller} for interactions between these various spectral sequences.
\end{rema}

\subsubsection{The Adams-Novikov spectral sequences}\label{sec:Adams-Novikov}
One of the reasons to consider the so-called Adams-Novikov spectral sequences is that they allow
 determining certain differentials in the Adams spectral sequence.
 This is based on the fact that one can replace the ring spectrum $\sH\FF_p$ in the preceding construction
 by an arbitrary ring spectrum $\E$ (see next section).
 The resulting construction is functorial in $\E$.

Let us clarify the premises of chromatic homotopy theory.
 The theory of characteristic classes can be expressed in stable homotopy by a very simple structure:
 a (complex) orientation on a given ring spectrum $\E$ is a class $c \in \tilde \E^{2}(\CC\PP^\infty)=\E^{2}(BU)$
 in the associated reduced cohomology of the infinite projective space which restricts to $1 \in \tilde \E^{2}(\CC\PP^1)\simeq \E^0(*)$.
 The beauty of stable homotopy is that this single class determines higher Chern classes satisfying all classical
 properties, except that the first Chern class of a tensor product of two line bundles is not always given by addition,
 but in general is expressed via a well-defined (commutative) formal group law $F_\E(x,y)$,
 with coefficients in the base ring $\E_*=\E_*(*)$,
 and depending only on the oriented ring spectrum $(\E,c)$.

A fundamental observation made by \textcite{QuillenFGL} is that the complex cobordism ring spectrum $\MU$ admits a canonical orientation $c_\MU$
 such that $(\MU,c_\MU)$ is the universal oriented ring spectrum and the associated formal group
 $\big(\MU_*,F_\MU(x,y)\big)$ is the universal formal group law, defined by Lazard.\footnote{In particular,
 the ring of coefficients $\MU_*$ is isomorphic to the Lazard ring $\mathbbm L=\ZZ[x_1,x_2,\hdots]$.
 See \cite[Theorems 4.1.6 and A2.1.10]{RavenGreen}.}
 As $\FF_p$-linear singular cohomology is oriented,\footnote{Naturally, the associated formal group law is the additive one.}
 it inherits a ring map $c:\MU \rightarrow \sH\FF_p$. This yields the (first) Adams-Novikov spectral
 sequence:
\begin{equation}\label{eq:Adams-Novikov-MU}
E_{2,\MU}^{s,t}=\Ext^{s,t}_{\MU_*\MU}(\MU_*,\MU_*) \Rightarrow \piS{t-s}\,,
\end{equation}
a multiplicative spectral sequence which converges to the integral stable stem.
 In contrast to the Adams spectral sequence, we have used the graded $\ZZ$-algebra
 $\MU_*\MU$, which is the $\ZZ$-dual of the algebra of stable operations $\MU^*\MU$
 on complex cobordism. The pair $(\MU_*,\MU_*\MU)$ forms what is called
 a (graded) \emph{Hopf algebroid} over the ring $R=\ZZ$: a groupoid in the category of
 (graded) $R$-algebras.\footnote{We refer the reader to \cite[Appendix 1]{RavenGreen} for more information,
 which attributes the terminology to Haynes Miller.}
 The $E_2$-term is given by the extension groups in the category
 of graded comodules over this graded Hopf algebroid, and as before, the index $t$ refers
 to the internal grading.

One can get a more useful spectral sequence by working in the
 $p$-local stable homotopy category $\SH_{(p)}$, obtained by inverting all primes
 except $p$. Indeed, Quillen showed that the $p$-local ring spectrum $\MU_{(p)}$
 splits into a direct sum of tensor products of a ring spectrum $\BP$ called
 the Brown-Perterson spectrum. This decomposition reflects the structure
 of the category of formal group laws over $\ZZ_{(p)}$, as $\BP$ is complex oriented
 and the associated formal group law $(\BP_*,F_\BP)$ is the universal $p$-typical
 one.\footnote{This structure is based on the Cartier isomorphism 
 for $p$-local formal group laws.
 We refer the reader to \cite[Appendix 2]{RavenGreen} for a concise exposition aimed at
 applications in homotopy theory. A more systematic treatment can be found in the classical monograph of \cite{Hazew}.
 Recall in particular that $\BP_*=\ZZ_{(p)}[v_1,v_2,\hdots]$, with $\deg(v_i)=2(p^i-1)$.}
 Applying the general construction to $\BP$, one gets the (second) Adams-Novikov
 spectral sequence:
\begin{equation}\label{eq:Adams-Novikov-BP}
E_{2,\BP}^{s,t}=\Ext^{s,t}_{\BP_*\BP}(\BP_*,\BP_*) \Rightarrow \piS{t-s} \otimes_\ZZ \ZZ_{(p)}\,,
\end{equation}
with similar properties to the first one, but converging to the $p$-local stable stem.
 In fact, the above $E_2$-term is precisely given by the $p$-component of \eqref{eq:Adams-Novikov-MU},
 according to \cite[Theorem 1.4.2]{RavenGreen}.

Since the formal group law of $\sH\FF_p$ is additive, therefore $p$-typical,
 one deduces that the ring map $c$ corresponding to the orientation of $\sH\FF_p$
 induces a ring map $c':\BP \rightarrow \sH\FF_p$.
 The map $c'$ induces a morphism of multiplicative spectral sequences
 $E_{r,\BP}^{s,t} \rightarrow E_{r}^{s,t}$. One can then combine information
 between both spectral sequences to obtain computations of the stable stem.\footnote{See
 \cite[\textsection 4]{RavenNovice} for a thorough study up to degree 18.}

\begin{rema}\label{rem:stacks}
As explained in a celebrated course by \cite{HopStacks},
 the Hopf algebroid $(\MU_*,\MU_*\MU)$ is an affine presentation
 of the stack $\mathcal M_{FG}$ of formal group laws with strict isomorphisms.
 One can define a line bundle $\omega$ on it by assigning to a formal group law over
 a ring $R$ its space of invariant differential forms.
 Then the $E_2$-term of the Adams-Novikov spectral sequence\eqref{eq:Adams-Novikov-MU}
 is concentrated in even degree $t$ and can be computed as:
$$
E_{2,\MU}^{s,2t}=H^s\big(\mathcal M_{FG},\omega^{\otimes t}\big).
$$
This beautiful formula (proved in \cite[\textsection 3.2, (3.5)]{GoerssStacks})
 allows one to express chromatic homotopy theory
 as arising from the stratification of the $p$-localization of the stack
 $\mathcal M_{FG}$ induced by the height of $p$-typical formal group laws.
\end{rema}

\subsubsection{Spectral sequences in motivic and classical stable homotopy: a first link}
To provide the reader with an initial sense of the relevance
 of the motivic perspective in comparison to previous computational tools,
 we conclude this subsection with a striking comparison between spectral sequences.
 Fix an algebraically closed field $k$ of characteristic $0$.
 As recalled in Theorem~\ref{thm:LevinePi}, the motivic stable stems
 of $k$ in weight $0$ agree with the classical stable stems.
 Furthermore, as discussed in Section~\ref{num:slice},
 every motivic spectrum admits a canonical slice filtration.
 When applied to the motivic sphere spectrum over $k$,
 one deduces the \emph{slice spectral sequence for the motivic
 sphere spectrum in weight $0$}:
$$
E_{1,\mathrm{slice}}^{p,q}=[s_q(\un),\un[p+q]] \Rightarrow \pi_{-p-q,0}^k \simeq \piS{-p-q}\ .
$$
The main result of \cite{LevCSS} is the following comparison theorem.
\begin{theo}
Consider the above assumptions. Then,
 up to the following reindexing,
 the above slice spectral sequence is isomorphic to the Adams-Novikov spectral sequence:
$$
E_{r,\mathrm{slice}}^{p,q} \simeq E_{2r+1,\MU}^{3p+q,2p}.
$$
In particular, both induced filtrations on $\piS{-p-q}$ coincide.
\end{theo}
This result unveils a surprising link between the algebro-geometric content of motivic spectra
 and the basic structure of classical spectra.\footnote{Such a connection --- though not stated in this form ---
 was already anticipated by \cite{VoeOpen}, at the very end of Section~6.}

\subsection{Motivic cohomology with torsion coefficients}\label{sec:hmot-torsion}

We fix a prime $\ell$ invertible in $k$ and state in this paragraph
 the known results, all due to Voevodsky, about the motivic Eilenberg-MacLane spectrum $\HMot \FF_\ell$
 with $\FF_\ell$-coefficients (see Example \ref{ex:mot-spectra}(2), Remark \ref{rem:mot-modules}).
 To comply with \cite{DIAdams, IWX}, we let $\M_\ell$ be the bigraded $\FF_\ell$-algebra such that
 $\M_\ell^{n,i}=\HMot^{n,i}(k,\FF_\ell)$.

We also write $\Het\mu_\ell$ for the motivic ring spectrum which represents \'etale cohomology
 with coefficients in the $\ZZ$-graded $\FF_\ell$-torsion sheaf $\mu_\ell^{\otimes,*}$.
 According to the rigidity theorem of \cite[Corollary 6.4.2]{SV-BK}, it can be identified with the \'etale sheafification
 $a_\et(\HMot \FF_\ell)$ of $\HMot \FF_\ell$.\footnote{The cohomology represented by $a_\et(\HMot \FF_\ell)$ is usually called 
 the \emph{Lichtenbaum motivic cohomology} with $\FF_\ell$-coefficients.} This yields a canonical morphism of motivic
 ring spectra:
\begin{equation}\label{eq:HMotHet}
\HMot \FF_\ell \xrightarrow{\gamma_\et} \Het\mu_\ell.
\end{equation}
 Given that result, it is notable that the Beilinson-Lichtenbaum conjecture,
 proved by \cite{VoeBK},\footnote{See also the review of \cite{RiouBK}.}
 is equivalent to the following formulation stated purely in terms of the homotopy $t$-structure
 on $\iSH(k)$, as defined in Theorem \ref{thm:htp-t-struct}.
 In the case where $k$ is not perfect of characteristic $p$,
 the formulation still makes sense and remains valid by appealing to Remark~\ref{rem:htp-non-perfect}.
\begin{theo}\label{thm:BL}
Consider the above notation. Then the map \eqref{eq:HMotHet}
 induces an isomorphism of motivic ring spectra over $k$:
$$
\HMot \FF_\ell \rightarrow \tau_{\geq 0}\big(\Het\mu_\ell\big),
$$
using the (homological) truncation functor associated with the homotopy $t$-structure on $\iSH(k)$
 of Theorem \ref{thm:htp-t-struct}.
\end{theo}
We leave it to the reader to verify the equivalence of this statement
 with the classical formulation of the Beilinson-Lichtenbaum conjecture given in 
 \cite[\textsection~3, Conjecture~(Be4)]{SV-BK}.\footnote{Hint: use the properties
 of the homotopy $t$-structure as stated in Section~\ref{sec:stab-htp-t}.}

\subsubsection{Periodicity in Galois cohomology}
We will deduce from the preceding theorem a description of the
 stable motivic homotopy sheaves of $\HMot \FF_\ell$ (Definition~\ref{df-htp-mod}).
 According to the preceding result, let us consider \'etale unramified $\mu_q$-cohomology,
 a classical invariant in arithmetic, organized as a homotopy module (Definition \ref{df-htp-mod}).
 For any integer $i \in \ZZ$, we let $\mathcal H_\et^i(\mu_\ell)_*$ be the homotopy module over $k$
 whose $n$-th graded part is given by the sheaf
$$
\mathcal H_\et^i(\mu_\ell)_n=\mathcal H_\et^{i+n}(-,\mu_\ell^{\otimes n}),
$$
obtained by considering the Zariski, or equivalently Nisnevich, sheafification of the corresponding
 \'etale cohomology presheaf on $\Sm_k$. It is worth noting that
 the homotopy module $\mathcal H_\et^i(\mu_\ell)_*$ is equivalent to the Rost cycle module defined
 by the following Galois cohomology functor on function fields $E/k$:
$$
E \mapsto H^{i+*}(G_E,\mu_q^*).
$$
This follows from the equivalence of categories mentioned in \ref{ex:piSA}(2),
 the fact that $\eta$ acts trivially on $\Het \mu_\ell$
 and the identification of \'etale cohomology of fields with Galois cohomology.
 
By construction, we also have the relation $\piAS i(\Het\mu_\ell)=\mathcal H_\et^{-i}(\mu_\ell)_*$.
 The ring structure of $\Het\mu_\ell$, as well as the classical cup product in Galois cohomology,
 corresponds to the fact that $\mathcal H_\et^*(\mu_\ell)_*$ is
 a commutative monoid object in the category of homotopy modules. We will simply say that it is an algebra.
 It is striking that the stable motivic homotopy sheaves $\piAS i(\Het\mu_\ell)$ satisfy
 a periodicity property analogous to the $v_i$-periodicity observed in the stable stem.

Let us first recall that, for any integer $q$, the action of the absolute Galois group $G_k$ on the Galois module
 $\mu_\ell^{\otimes,q}$ is given by the $q$-th power $\chi^q$ of the cyclotomic character
 $\chi:G_k \rightarrow \FF_\ell^\times$.
 This implies that there exists a canonical isomorphism $\mu_\ell^{\otimes,\ell-1} \simeq \FF_\ell$
  of Galois modules (or \'etale sheaves), which can be immediately translated into a periodicity
 isomorphism in the algebra $H^*(G_E,\mu_q^*)$ for any extension field $E/k$.
 To get a motivic formulation that will be shortly stated, we apply the above theorem to deduce a canonical isomorphism:
$$
\M_\ell^{0,0} \simeq  H^0(G_k,\FF_\ell) \simeq H^0(G_k,\mu_\ell^{\otimes,\ell-1}) \simeq \M_\ell^{0,\ell-1}
$$
and we denote by $\tau' \in \M_\ell^{0,\ell-1}$ the image of $1$ under this isomorphism.

We can improve this periodicity if $k$ admits an $\ell$-th root of unity.
 Then the choice of $\zeta_\ell \in \mu_\ell(k)$ induces an isomorphism $\mu_\ell \simeq \FF_\ell$,
 and therefore a periodicity in $H^*(G_E,\mu_q^*)$ as above.
 Moreover, let us recall that one gets the following exact sequence in motivic cohomology:\footnote{One uses
 the exact sequence of abelian groups $0 \rightarrow \ZZ\xrightarrow{\ell} \ZZ \rightarrow \FF_\ell$,
 viewed for example in motivic complexes, and the known computations of motivic cohomology with twists $0$ and $1$.}
$$
0 \rightarrow \HMot^{0,1}(k,\FF_\ell) \rightarrow \HMot^{1,1}(k,\ZZ)=k^\times
 \xrightarrow{\ \ell\ } k^\times=\HMot^{1,1}(k,\ZZ) \rightarrow \HMot^{1,1}(k,\FF_\ell) \rightarrow 0.
$$
It induces a canonical isomorphism $\M_\ell^{0,1}(k) \simeq \mu_\ell(k)$.
 Then we denote generically by $\tau \in \M_\ell^{0,1}(k)$ the element which corresponds to $\zeta_\ell$
 via this isomorphism.\footnote{Observe that according to these definitions, $\tau'=\tau^{\ell-1}$.}

This study, paired with the previous theorem, gives the following corollary
 formulated in terms of homotopy modules.
\begin{coro}
Consider the above notation.
Then there exists canonical isomorphisms of commutative monoids in the monoidal abelian category of homotopy modules:
\begin{align*}
\piAS *(\HMot \FF_\ell)& \simeq \oplus_{i=0}^{l-2} \mathcal H_\et^{-i}(\mu_\ell)_*[\tau'] \\
\piAS *(\Het \mu_\ell)& \simeq \oplus_{i=0}^{l-2} \mathcal H_\et^{-i}(\mu_\ell)_*[\tau',\tau^{\prime -1}]
\end{align*}
where the left-hand side is seen as a polynomial algebra (resp. Laurent polynomial algebra) in the variable $\tau'$,
 also seen as an element of the respective left-hand sides.

If we assume that $k$ contains an $\ell$-th root of unity and let $\tau$ be given as above,
 then one deduces canonical isomorphisms of algebras in homotopy modules:
\begin{align*}
\piAS *(\HMot \FF_\ell)& \simeq \uKM_*/\ell[\tau] \\
\piAS *(\Het \mu_\ell)& \simeq \uKM_*/\ell[\tau,\tau^{-1}]
\end{align*}
with the same description of the right-hand sides as previously, but with respect to $\tau$.
\end{coro}
In both cases, the first isomorphism is induced by \eqref{eq:HMotHet} and follows from the periodicity properties studied earlier together
 with the preceding theorem; the second isomorphism is a reformulation of these periodicity properties.

\begin{rema}
\begin{enumerate}
\item Since we are mixing cohomological bidegrees and homological bidegrees for the homotopy $t$-structure,
 we make explicit the effect of multiplication by $\tau'$ (respectively $\tau$) in terms of the $\GG$-twist operation
 on homotopy modules (see Definition \ref{df-htp-mod}):
\begin{align*}
&\piAS i(\M_\ell).\tau'=\piAS {i+l-1}(\M_\ell)\tw{l-1} \\
\text{resp. } & \piAS i(\M_\ell).\tau=\piAS {i+1}(\M_\ell)\tw{1}.
\end{align*}
\item When $k$ does not contain an $\ell$-th root of unity, the stated periodicity is
 the best possible as for all $q \geq 0$, one gets:
$$
\piAS q(\HMot \FF_\ell)_q(k)=\M_\ell^{0,q} \simeq H^0(G_k,\mu_\ell^{\otimes,q})=
\begin{cases}
\FF_\ell & \text{if } q=0 \text{ mod } {\ell-1}, \\
0 & \text{otherwise.}
\end{cases}
$$
\end{enumerate}
\end{rema}

One deduces from the previous corollary the following result.
\begin{coro}
Under the assumptions of the Theorem \ref{thm:BL}, the morphism of motivic ring spectra over $k$, induced by \eqref{eq:HMotHet},
$$
\HMot\FF_\ell[\tau^{\prime -1}] \rightarrow \Het \mu_\ell
$$
is an isomorphism, where the left-hand side is obtained by internally inverting $\tau'$ (see Section \ref{sec:intern-comp&loc}).
\end{coro}
 This statement has a long history, starting with the pioneering work of \cite{ThoKEt} where an analogous
 result for algebraic K-theory was established. The above motivic version was in fact proved independently of the validity of the Beilinson-Lichtenbaum conjecture
 (which is stronger) by \cite[Theorem 6.2]{LevBott}.

Note that this corollary consists of two distinct assertions:
 first, an \'etale descent theorem for the $\tau'$-inverted theory; second,
 a computation of étale motivic cohomology in terms of classical étale cohomology.
 The latter is, as mentioned before Theorem \ref{thm:BL}, a rigidity statement (that also has a long history),
 and was established for motivic cohomology by \cite[Corollary 6.4.2]{SV-BK}.
 It has since been generalized in various contexts, the latest statement in motivic homotopy theory being \cite{BachRig}.
 The descent statement was likewise extended in motivic homotopy theory by \cite{ELSO} (for arbitrary $\MGL$-modules).

The following corollary plays a central role in \cite{IWX}.
\begin{coro}
Let $k$ be an algebraically closed field of characteristic prime to $\ell$.
 Then the $\FF_\ell$-linear motivic cohomology ring of coefficients over $k$ is the polynomial algebra:
$$
\M_\ell=\FF_\ell[\tau]
$$
where $\tau$ is an element of \emph{homological} bidegree $(0,-1)$ defined by the choice of an $\ell$-th root of unity in $k$.
\end{coro}
Indeed, as $k$ is separably closed, one obviously obtains an isomorphism of graded algebras $\KM_*(k)/\ell=\FF_\ell$, concentrated in degree $0$.

\subsubsection{The motivic Steenrod algebra}
The last ingredient that we will need is
 the motivic Steenrod algebra modulo $2$ over the field $\CC$.
 More generally, the mod~$\ell$ motivic Steenrod algebra over a field $k$ of characteristic prime to $\ell$
 is defined as the $\FF_\ell$\nobreakdash-algebra of stable cohomology operations on $\FF_\ell$-linear motivic cohomology:
 $$\mathcal A^{**}(k,\FF_\ell):=(\HMot\FF_\ell)^{**}(\HMot\FF_\ell),$$
 using the topological notation, keeping in mind the bidegree grading
 specific to the motivic context.

Instead of recalling all the details, we will give some key references to the reader.
 These operations were first computed by \cite{VoeOper} over a perfect
 field.\footnote{This assumption can easily be removed as the $\FF_\ell$-linear motivic cohomology
 spectrum is invariant under purely inseparable extensions.}
 Nice accounts, with corrections and complements, were given by \cite{RiouOperS} and \cite{HKO_Oper}.

As in the topological situation, explained in the previous section, 
 it is easier to work with the dual motivic Steenrod algebra modulo $\ell$,
 defined as follows:\footnote{While this bigraded $\FF_\ell$-algebra
 is $\FF_\ell$-dual to its cohomological counterpart, its algebraic structure involves the Cartan
 relations which are simpler than the Adem relations which describe the multiplication of $\mathcal A_\ell^k$.
 This is perfectly analogous to the situation in topology.}
$$
\mathcal A_{p,q}(k,\FF_\ell)=(\HMot\FF_\ell)_{p,q}(\HMot\FF_\ell)=[\un(q)[p],\HMot\FF_\ell \otimes \HMot \FF_\ell]_{\SH(k)}.
$$
For future reference, we will simply write $A=(\HMot\FF_2)_{**}(\HMot\FF_2)$ taken over the field $\CC$.
 We now recall the explicit presentation of this bigraded commutative $\FF_2$-algebra which also carries
 with a bigraded $\M_2$-algebra structure:
\begin{equation}
A=\M_2[\tau_0,\tau_1,\hdots,\xi_1,\xi_2,\hdots]/(\tau_i^2-\tau\xi_{i+1})
\end{equation}
where the generators are the motivic analogue of the Milnor basis:
\begin{itemize}
\item $\tau_i$ has bidegree $(2^{i+1}-1,2^i-1)$ and is dual to $\Sq^{2^i-1}$,
\item $\xi_i$ has bidegree $(2^{i+1}-2,2^i-1)$ and is dual to $\Sq^{2^i}\Sq^{2^{i-1}}\hdots\Sq^1$.
\end{itemize}
We refer the reader to \cite[Theorem 5.6]{HKO_Oper} for a precise and comprehensive description
 of all the algebraic structures of the Hopf algebroid $(\M_2,A)$.

\subsubsection{Complex realization}
We conclude by studying the complex case $k=\CC$, using the complex realization
 $\rho_\CC:\iSH(\CC) \rightarrow \iSH$ of Section \ref{sec:stable-real}.
 Fixing an $\ell$-th root $\zeta_\ell$ of unity in $\CC$, we get the class $\tau \in \M_\ell^{0,1}$ as above.

 Recall that we have in Example \ref{sec:stable-real} defined the $\FF_\ell$-linear Betti motivic spectrum:
 $\HB\FF_\ell:=\rho_{\CC*}(\sH\FF_\ell)$.
 Note that, apart from the fact that it has coefficients in a field of positive characteristic,
 it satisfies all the axioms of a Mixed Weil cohomology theory, as defined in \cite{CD2}.
 In particular, many of the results of \emph{loc. cit.} apply \emph{mutatis mutandis} to the resulting motivic ring spectrum $\HB \FF_\ell$,
 except that one needs to modify the discussion using the K-theory spectrum $\KGL$ at the end of \textsection 2.3.\footnote{Indeed,
 Theorem 2.3.23 of \emph{loc. cit.} fails in the $\FF_\ell$-linear context because of the existence of nontrivial Steenrod operations.
 This is why we use a different argument to get the map $\gamma_\mathrm B$ below.}
 More importantly, $\HB\FF_\ell$ is $(0,1)$-periodic and in fact, one can define a canonical \emph{Tate-twisting} $\FF_\ell$-vector space
 attached to this cohomology. Using the notation of \emph{loc. cit.}, 2.1.3, 2.1.5, one puts:
$$
\FF_\ell(1):=\tilde \HH_\mathrm B^{1}(\GG,\FF_\ell)
$$
using the reduced Betti cohomology of the sphere $\GG$. According to the axioms of a mixed Weil theory,
 this is a $1$-dimensional $\FF_\ell$-vector space so that we can define its $i$-th tensor power $\FF_\ell(i)$ for any integer $i$.
 Moreover, for any smooth complex scheme $X$, any pair of integers $(n,i) \in \ZZ^2$, one gets a canonical isomorphism:
$$
\HB^{n,i}(X,\FF_\ell) \simeq \HB^n(X,\FF_\ell) \otimes_{\FF_\ell} \FF_\ell(i).
$$
See \emph{loc. cit.} \textsection 2.1.7.
In fact, in the case of Betti cohomology, Tate-twists admit a canonical trivialization:
$$
\FF_\ell(1)=\tilde \HH^{1}(\GG(\CC),\FF_\ell) \simeq \HH^0(S^0,\FF_\ell)=\FF_\ell.
$$
We will still denote by $c \in \HB^{0,1}(\CC,\FF_\ell)$ the class defined by this isomorphism,\footnote{According
 to \emph{loc. cit.} 2.2.6 and 2.2.8, it induces an orientation of the motivic ring spectrum $\HB\FF_\ell$,
 which is unique as seen from the preceding isomorphism;} so that the $(0,1)$-periodicity of Betti cohomology is
 induced by multiplication by $c$.

According to the description of the motivic Eilenberg-MacLane spectrum using symmetric powers (see \cite[\textsection 6.1]{VoeICM}),\footnote{this works
 well since we are over a field of characteristic $0$,}
 one deduces a canonical isomorphism: $\rho_{\CC}(\HMot \FF_\ell) \simeq \sH\FF_\ell$. This allows one to apply the adjunction
 $(\rho_\CC,\rho_{\CC*})$ to get the following morphism of ring spectra:\footnote{There are of course several other ways to build this map.}
$$
\gamma_{\mathrm B}:\HMot\FF_\ell \rightarrow \rho_{\CC*}\rho_\CC(\HMot\FF_\ell) \simeq \rho_{\CC*}(\sH\FF_\ell)=\HB\FF_\ell.
$$
The functoriality properties of the induced map on the associated cohomologies and the construction of $\tau$ imply that
 $\gamma_{\mathrm B*}(\tau)=c$.

Finally, the comparison of $\FF_\ell$-linear \'etale cohomology with Betti cohomology induces a commutative diagram
 of motivic ring spectra:
\begin{equation}\label{eq:compat-real}
\begin{split}
\xymatrix@R=10pt@C=22pt{
\HMot\FF_\ell\ar^-{\gamma_\et}[r]\ar@{=}[d] & \Het\mu_\ell\ar^\tau_\sim[r] & \Het \FF_\ell\ar^\sim[d] \\
\HMot\FF_\ell\ar^{\gamma_\mathrm B}[rr] && \HB\FF_\ell
}
\end{split}
\end{equation}
where the right-hand side map is the isomorphism, obtained by the analytification map. The commutativity of the diagram
 can be obtained by using the analytification functor of \cite[\textsection 2]{AyoubAn}.
 Therefore, one can restate all the previous results in terms of complex realization as follows.
\begin{coro}\label{cor:F_l-compare-HM-HB}
Over the field $k=\CC$, the map $\gamma_\mathrm B$ induces isomorphisms of motivic ring spectra:
\begin{align*}
&\HMot \FF_\ell \simeq \tau_{\geq 0} \HB\FF_\ell, \\
&\HMot \FF_\ell[\tau^{-1}] \simeq \HB\FF_\ell.
\end{align*}
\end{coro}

The following result was first proved by \cite{DIAdams} (see in particular Corollary 2.9).
\begin{coro}\label{cor:mot-cl-Hopf-algebroid}
Consider the notation introduced above.
Then the complex realization functor $\rho_\CC$ induces an isomorphism of graded Hopf
 algebroid (see Section \ref{sec:Adams-Novikov})
 over the graded ring
 $\M_\ell[\tau^{-1}] \simeq \FF_\ell[\tau,\tau^{-1}]$:
$$
(\M_\ell,\mathcal A_{**}(\CC,\FF_\ell))[\tau^{-1}] \simeq (\FF_\ell,A_\topo) \otimes_{\FF_\ell} \M_\ell[\tau^{-1}].
$$
The left-hand side is obtained by inversion of the element $\tau$, while the right-hand side is obtained
 by scalar extensions.
\end{coro}
In fact, \emph{op. cit.} is slightly less precise, and considers only the case $\ell=2$,
 which is the most interesting one from a computational perspective.

\begin{rema}
This corollary is the first manifestation of the principle that 
 the weights in (torsion) motivic homotopy theory add a supplementary dimension to the classical invariants. 
 As an example, one should be aware that the cohomological operations on Betti cohomology modulo $\ell$
 are obtained from the Steenrod algebra by scalar extension along the Laurent polynomial $\FF_2$-algebra
 $\FF_2[c,c^{-1}]$. In particular, besides the Steenrod operations,
 we have the operation corresponding to multiplication by the scalar $c \in \HB^{0,1}(\CC,\FF_\ell)$,
 which expresses the $(0,1)$-periodicity of Betti cohomology.
\end{rema}

\subsection{The motivic Adams spectral sequence}

\subsubsection{Bousfield nilpotent completions and Adams towers}\label{sec:nilpotent}
We will now recall a general construction due to \cite[\textsection 5]{BousLoc},
 from which we adopt the notation. The aim is to define generalized Adams spectral sequences,
 compute their $E_2$-term as suitable extension groups, compute their abutment in terms of appropriate resolutions,
 and to elucidate their convergence properties.
 Nowadays, the construction can be made in an arbitrary stable presentable (symmetric) monoidal $\infty$-category $\iC$,
 which encompasses both the cases of $\SH$ and $\SH(k)$.\footnote{The reader can also consult \cite[Part 1]{MNNnilp},
 \cite[\textsection 2]{BOnb}, or \cite{Manto}.}

To be consistent with the previous notation, we let $[X,Y]_{\iC}$ be the morphisms
 between two objects of $\iC$ computed in its homotopy category.
 Given an integer $s \in \ZZ$, we define: $\pi_s(X,Y)=[X[s],Y]$.

We fix an object $\E$ in $\iC$ with a (not necessarily commutative)
 monoid structure\footnote{It is useful to avoid the commutativity assumptions.
 Indeed, it is easier to construct an $A_\infty$-structure rather than an $E_\infty$-structure
 on ring spectra, in the classical stable homotopy category. For example,
 it is a famous result of \cite{Lawson} that there cannot exist an $E_\infty$-ring structure on $\BP$.}
  $u:\un \rightarrow \E$ and $\mu:\E \otimes \E \rightarrow \E$
 in the associated homotopy category $\Ho\iC$. We slightly abuse notation and denote by
 $u:\un \rightarrow \E$ an arbitrary representative, as an element of $\iC_1$,
 of the homotopy class $u$.
 We then consider the homotopy fiber $\bar \E$ of the map $u$ and
 put $\bar \E^s=\bar \E^{\otimes s}$.
 One deduces homotopy exact sequences:
$$
\xymatrix@C=50pt@R=4pt{
\bar \E\ar^-{f_0=\epsilon}[r] & \un\ar^-u[r] & \E, \\
\bar \E^{s+1}\ar^-{f_{s+1}=\epsilon \otimes \bar \E^{s}}[r] & \bar \E^s\ar[r] & \E \otimes \bar \E^s,
}
$$
the second one being obtained from the first one by left-tensoring with $\bar \E^s$.
 We deduce a decreasing \emph{tower} of objects over $\un$:
$$
\hdots \rightarrow \bar \E^{s+1} \xrightarrow{f_{s+1}} \bar \E^{s}
 \rightarrow \hdots \rightarrow \bar \E \xrightarrow{f_0} \un
$$
or, in other words, an $\infty$-functor $\bar \E^\bullet:\ZZ_{\geq 0}^{\op} \rightarrow (\iC/\un)$,
 with values in the indicated comma $\infty$-category (see \cite[Définition 10.1]{Cisinski}
 for comma $\infty$-categories).
 Let us also define by $\bar \E_s$ the homotopy cofiber which fits into the homotopy exact sequence:
$$
\bar \E^{s+1} \rightarrow \un \rightarrow \bar \E_s.
$$
 Applying the octahedron axiom (see \cite[Theorem 1.1.2.14]{LurieHA}), to the preceding
 homotopy exact sequences, one deduces an octahedral diagram, in planar form:
\begin{equation}\label{eq:Bous-Adams}
\begin{split}
\xymatrix@=10pt{
& \bar \E^{s+1}\ar^{f_{s+1}}[rr]\ar[ld] && \bar \E^s\ar[rd]\ar[ld] & \\
\un\ar[rd]\ar@{}|/2pt/{(*)}[r]
 & & \E \otimes \bar \E^s\ar[ld]\ar@{-->}[lu]\ar@{}|/3pt/{(*)}[u]\ar@{}|/3pt/{(*)}[d]
 & & \un\ar@{}|/2pt/{(*)}[l]\ar[ld] \\
& \bar \E_{s}\ar_{g_{s+1}}[rr]\ar@{-->}[uu] && \bar \E_{s-1}\ar@{-->}[lu]\ar@{-->}[uu] &
}
\end{split}
\end{equation}
where a triangle containing $(*)$ is a homotopy exact sequence,
 a dashed arrow is a boundary map of such a homotopy exact sequence (therefore the target is implicitly suspended once),
and all other triangles are commutative
 (including the ones obtained by identifying the two objects $\un$ to a single vertex).
 We have therefore obtained a tower $\bar \E_\bullet:\ZZ_{\geq 0}^{\op} \rightarrow \un/\iC$ of objects
 under $\un$.\footnote{According to \cite{HopStacks}, one should call it the \emph{normalized $A$-Adams resolution of $\un$}.
 By tensoring with an arbitrary object $\mathbf X$, we obtain the \emph{normalized $\E$-Adams resolution of $\mathbf X$}.}

We can then apply any homological functor\footnote{i.e., $\mathcal A$ is an abelian category
 and $H$ sends a homotopy exact sequence to a long exact sequence;}
 $H:\Ho\iC \rightarrow \mathcal A$ to the diagram \eqref{eq:Bous-Adams},
 yielding two exact couples that fit into a Rees system in the terminology of \cite{EH-ec}.
 According to Theorem 7.10 of \emph{loc. cit.}, both exact couples give rise to the same spectral sequence.
 We apply this construction to the homological functor $[\un,-]_{\iC}$.\footnote{The general
 construction applies the functor $[\sX,(-) \otimes \sY]_{\iC}$ for arbitrary
 objects $\sX$, $\sY$.}
\begin{defi}
Consider the above assumptions.
 One defines the nilpotent $\E$-completion of an arbitrary object $\sX$ as the following homotopy limit:
$$
\hat{\sX}^\E:=\lim_{n\geq 0}\big(\sX \otimes \bar \E_n\big).
$$
The spectral sequence associated to the exact couples obtained from \eqref{eq:Bous-Adams}
 by replacing the entries $\un$ with $\hat \un^\E$ and then applying the homological functor $[\un,-]_{\iC}$:
$$
E_{1,\E}^{s,t}=\pi_{t-s}(\un,\E \otimes \bar \E^s)
 \Rightarrow \pi_{t-s}(\un,\hat \un^\E)
$$
is called the \emph{$\E$-Adams spectral sequence}.\footnote{Following the topological conventions,
 the differentials $d_1^{**}$ are of bidegree $(1,0)$.}
\end{defi}

\begin{exem}\label{ex:AdamsSSq}
\begin{enumerate}
\item Applying the above construction with $\iC=\iSH$,
 and with the associative algebra
 $\E=\sH\FF_p, \MU, \BP$ respectively, one obtains the three spectral sequences
 which appear in the previous section.
\item We will apply the above construction in the motivic case $\iC=\iSH(k)$.
 Note that in the motivic case, there is an additional grading coming from the Tate twist.
 In particular, instead of applying the above construction to a given ring spectrum $\E$,
 we will apply it to the graded object $\E(*)$. This implies that we will be working with
 the same object as above, but equipped with an additional grading,
 which will automatically be compatible with products.

The main example for us will be $\E=\HMot\FF_2$ over the field $k=\CC$.
 However, it is possible to consider the two cases $\E=\sH\FF_\ell, \MGL$.
 In addition, the analog of the Brown-Peterson spectrum exists in motivic homotopy,
 a ring spectrum denoted by $\BPGL$.
 See \cite{HKO_Adams}, or more generally \cite{NSO}.
\end{enumerate}
\end{exem}

\begin{rema}
The convergence of the $\E$-Adams spectral sequence is a delicate matter.
 If $E_1^{s,t}$ vanishes for $s>t$ (which will follow in all our applications),
 then one gets a weak form of convergence as explained in \cite{BousLoc},
 beginning of \textsection 6, and called \emph{conditional convergence}.
 According to \emph{loc. cit.} Proposition 6.3,
 the strong convergence\footnote{By ``strong convergence'', we mean that the filtration on the abutment
 is exhaustive, separated (Hausdorff) and complete;} is equivalent to the vanishing:
$\underset r {\lim}^1 \big(E_r^{s,t}\big)=0$.

In the applications, the spectral sequence will in fact
 be concentrated in a bounded region from the second page onward.
 This implies that for any $(s,t)$,
 the sequence $(E_r^{s,t})_{r\geq 1}$ stabilizes
 and that the filtration on the abutment is finite.
 (See the argument in Section \ref{sec:Adams-topo}.)
\end{rema}

\subsubsection{The cobar construction}\label{sec:cobar}
Let us consider the notation of the above definition,
 and assume in addition $\E$ is an $A_\infty$-object in $\iC$,
 or equivalently an associated algebra object as defined by \cite[\textsection 4.1.1, Definition 4.1.1.6]{LurieHA}.
 
The next step is to compute the $E_2$-term of the preceding spectral sequence.
 As in topology, we associate to $\E$ a homology theory on $\iC$ which to an object $X$ of
 $\iC$ associates the graded $\E_*$-module $\E_*X=\pi_*(\un,\E \otimes X)$,
 where $\E_*=\pi_*(\un,\E)$ is the associated \emph{ring of coefficients}.
 It follows as in Section \ref{sec:Adams-Novikov}, that the pair
 $(\E_*,\E_*\E)$ is a \emph{Hopf algebroid}.
 
To go further, we again follow the classical approach from topology
 highlighted by \cite[\textsection 5]{HopStacks}.
 Using the $A_\infty$-structure on $\E$,
 one defines a cosimplicial diagram $\CB^\bullet(\E):\Delta \rightarrow \iC$:
$$
\xymatrix@=30pt{
\E\ar@<3pt>[r]\ar@<-3pt>[r] & \E \otimes \E\ar[l]\ar@<6pt>[r]\ar[r]\ar@<-6pt>[r] & \E \otimes \E \otimes \E\ar@<3pt>[l]\ar@<-3pt>[l]\ar@{..>}[r] & \hdots
}
$$
such that $\CB^n(\E)=\E^{\otimes n+1}$; see \cite[Construction 2.7]{MNNnilp}.
 This is called the \emph{cobar construction} on $\E$.\footnote{Note from \emph{loc. cit.} 
 that it admits augmentation by $\un$.}
 According to a classical procedure (originally formalized by \cite{BKloc}),
 one can extract a tower
 $\ZZ_{\geq 0}^{op} \rightarrow \iC$ by taking partial limits:\footnote{\label{fn:Tot} In fact,
 the elegant Theorem 2.8 of \emph{op. cit.}, attributed to Lurie,
 asserts that the general $\infty$-functor
$$
\Tot_\bullet:\Fun(\Delta,\iC) \rightarrow \Fun(\ZZ_{\geq 0}^{\op},\iC)
$$
is an equivalence of stable $\infty$-categories.
 This is a stable $\infty$-categorical version of the Dold-Kan equivalence.}
$$
\Tot_n\CB^\bullet(\E)=\varprojlim_{i \leq n} \CB^i(\E).
$$
Let us state explicitly the following fundamental comparison result,
 proved in this form in \emph{op. cit.} Proposition 2.14.
\begin{prop}\label{prop:cobar&standard}
Under the above assumptions, there exists a canonical equivalence of towers:
 $\bar \E_\bullet \simeq \Tot_\bullet\CB^\bullet(\E)$.\footnote{We have neglected the
 natural augmentation with source $\un$.}

The tower $\Tot_\bullet\CB^\bullet(\E)$ will be called the \emph{tower $\E$-Adams
 resolution} attached to $\E$.\footnote{Here, it is understood that this is a resolution of $\un$.}
\end{prop}
One deduces that for any integer $t \in \ZZ$,
 the complexes $E_1^{*,t}$ from the $\E$-Adams spectral sequence take the following form
$$
\hdots 0 \rightarrow \E_t \rightarrow \E_{t-1}\E \rightarrow \E_{t-2}(\E^{\otimes 2}) \rightarrow \hdots
 \rightarrow \E_{t-s}(\E^{\otimes s}) \rightarrow \hdots
$$
where $\E_*$ sits in (cohomological) degree $s=0$,
 and the differentials are obtained as the alternating sum of the maps coming from the cobar-resolution of $\E$.
 Using the exterior pairing $\E*(X) \otimes_{\E_*} \E*(Y) \rightarrow \E_*(X \otimes Y)$,
 one deduces the following corollary.
\begin{coro}\label{cor:compute-E2}
Consider the above notation. Assume that $\E$ satisfies the following \emph{weak K\"unneth property}:
$$
\forall s>1, \text{  the exterior pairing induces an isomorphism }
(\E_*\E)^{\otimes_{\E_*}s} \rightarrow \E_*(\E^{\otimes s}).
$$
Then one can compute the $E_2$-term of the $\E$-Adams spectral sequence as
$$
E_2^{s,t}=\Ext^{s,t}_{\E_*\E}(\E_*,\E_*)
$$
the $\Ext$-group being computed in the category of comodules over the Hopf algebroid
 $(\E_*,\E_*\E)$.
\end{coro}
In our examples from topology, the weak form of the K\"unneth property
 follows from the assumption that $\E_*\E$ is flat over $\E_*$;
 see \cite[Proposition 5.7]{HopStacks}.

\subsubsection{Localization and completion}\label{sec:intern-comp&loc}
We finally recall the general tools to compute the abutment of
 the previous Adams spectral sequence, still following \cite{BousLoc}.
 According to the general philosophy developed by \cite{BKloc},
 this involves the $\infty$-categorical analogue of completion or localization in classical algebra.
 To recall these basic definitions, we will now assume that $\iC$
 is a presentable monoidal stable $\infty$-category.

We consider a $1$-morphism $f:L \rightarrow \un$,\footnote{Usually, it will be a homotopy class
 but the constructions will not depend on a choice of representative;} which will play the role
 of elements in the (motivic) stable stem, and $X$ be object of $\iC$.
\begin{enumerate}
\item For an integer $n>0$,
 one defines the $n$-th (homotopy) quotient $X/f^n$ of $X$ by $f$ as the homotopy cofiber
 of the map $X \otimes L^{\otimes n} \rightarrow X$.
\item One defines the completion of $X$ at $f$ as the following homotopy limit:\footnote{This is where
 we use the assumption that $\iC$ is presentable;}
$$
\hat X_f:=\holim_n \big(X/f^n\big).
$$
\item Assuming that $L$ is $\otimes$-invertible, one defines the localization of $X$ at $f$
 as the following homotopy limit:
$$
X[f^{-1}]:=\hocolim\big(X \xrightarrow f X \otimes L^{-1} \xrightarrow f X \otimes L^{-2} \rightarrow \hdots\big)
$$
where we have denoted by $L^{-n}$ the $n$-th tensor power of the $\otimes$-inverse of $L$,
 and simply by $f$ the morphism induced by $f$ after the obvious tensor product.
\end{enumerate}
These constructions can be extended to an $n$-tuple $I=(f_i:L_i \rightarrow \un)$ in an obvious way (left to the reader).
 We will (loosely) say that $I$ is an ideal of $\un$, and use the notation $\sX/I$, $\hat{\sX}_I$, $\sX[I^{-1}]$.
 
In order to state the next result,
 we will need to assume the existence of a $t$-structure $t$ on $\iC$.\footnote{Or, equivalently, its homotopy category $\Ho\iC$
 following standard usage in the literature.} We still use homological conventions, write $A \geq 0$ for non-negative objects
 and let $\upi_0$ be the associated (co)homological functor.
 We will assume that $\iC$ is \emph{left-complete}.\footnote{i.e., for any object $X$ in $\iC$, the canonical map
 $X \longrightarrow \underset{n \to -\infty}{\mathrm{colim}} \big(\tau_{\geq n}(X)\big)$ is an isomorphism.}
 We will also assume that $t$ is \emph{compatible with the tensor structure}
 in the usual sense (see footnote \ref{fn:t-tensor-compatible}, page \pageref{fn:t-tensor-compatible}).

We can now state the following pretty generalization of \cite[Theorems~6.5,6.6]{BousLoc},
 due to \cite{BOnb} and \cite{Manto}.
\begin{theo}
Consider the above notation.
 We let $\E$ be an $E_\infty$-algebra in $\iC$ such that $\E \geq 0$
 and $\sX$ be a connective object with respect to $t$.\footnote{In other words, homologically bounded below.}
 Let $I=(f_i:L_i \rightarrow \un)$ be an ideal of $\un$ as above.
\begin{enumerate}
\item Assume that for all $i$, $L_i$ is non-negative, strongly dualizable with a non-negative strong dual,
 and $\upi_0(\E) \simeq \upi_0\big(\un/I\big)$. Then: $\hat{\sX}_\E=\hat{\sX}_I$.
\item Assume that for all $i$, the tensor product with $L_i$ is a $t$-exact equivalence of $\infty$-categories
 and $\upi_0(\E) \simeq \upi_0\big(\un[I^{-1}]\big)$. Then: $\hat{\mathbf X}_\E=\sX[I^{-1}]$.
\end{enumerate}
\end{theo}
We refer the reader to \cite[Theorem 2.1+2.2]{BOnb},\footnote{The second point can easily be deduced from \emph{loc. cit.}
 Theorem 2.1.} or to \cite{Manto}.

\begin{exem}\label{ex:concrete-complete}
\begin{enumerate}
\item Let $\sX$ be a connective spectrum with respect to the
 canonical $t$-structure on $\iSH$. Then, for a prime $\ell \in k^\times$, one has:
$$
\hat{\sX}_{\sH\FF_\ell}=\hat{\sX}_\ell, \hat{\sX}_{\HH\ZZ}=\hat{\sX}_{\MU}=\sX, \hat{\sX}_\BP={\sX[\ell^{-1}]}.
$$
\item Let $\sX$ be a connective motivic spectrum with respect to the
 homotopy $t$-structure on $\iSH(k)$. Then, by applying the above theorem and Example \ref{ex:piSA}, one deduces:
$$
\hat{\sX}_{\MGL}=\hat{\sX}_{\HMot\ZZ}=\hat{\sX}_{\eta}, \hat{\sX}_{\HMot\FF_\ell}=\hat{\sX}_{(\ell,\eta)}, \hat{\sX}_{\BPGL}={\sX[\ell^{-1}]}.
$$
\item Consider $\sX$ as in the previous point.
 Assume in addition that $(-1)$ is a sum of squares in $k$ and that $\ell \neq 2$.
 Then: $\hat{\sX}_{(\ell,\eta)}=\hat{\sX}_{\ell}$.
 See \cite[Lemma 3.3.1]{Manto}.
\item Assume that $(-1)$ is a sum of squares in $k$ and that $k$ has finite $2$-cohomological dimension.
 Then the $\HMot\FF_2$-nilpotent completion of the motivic sphere spectrum agrees with
 its $2$-completion:
$$
\hat{\un}_{\HMot\FF_2}=\hat{\un}_2.
$$
This fact was first proved, over algebraically closed fields, by \cite{HKO_Adams}.
 The more general case is proved in \cite[Lemma 3.3.2]{Manto} (based on \emph{loc. cit.}).
\end{enumerate}
\end{exem}

\begin{rema}\label{rem:Bousloc}
We do not enter into the details, but it is important to note that another corollary of the above
 theorem is that, under the stated hypothesis, the nilpotent $\E$-localization also
 coincides with the (left Bousfield) localization with respect to $\E$, which consists in inverting
 maps inducing isomorphisms in $\E_*$-homology
 --- beware that in motivic homotopy, one has to consider the bigraded homology theory $\E_{**}$.
 The resulting localization functor $L_E$, following the general procedure described in Section \ref{sec:abstract-loc},
 has good properties --- for example, it respects $E_\infty$-ring spectra.
\end{rema}

Given all these preparations, we can now state the existence and form of the motivic Adams spectral
 sequence, first introduced by \cite{MorelAdams}, and more thoroughly studied by \cite{HKO_Adams}
 and \cite{DIAdams}. We state it only over the base field $\CC$ and modulo $2$ as this is our main focus.
 We leave to the reader the formulation of the other cases.
\begin{prop}
We work over the base field $k=\CC$, and modulo the prime $\ell=2$.
 Then there exists a \emph{weight graded} \emph{motivic Adams spectral sequence}, with a trigraded multiplicative structure,
 of the following form:
$$
E_2^{s,t,w}=\Ext^{s,t,w}_A(\M_2,\M_2) \Rightarrow \pi_{t-s,w}\big(\hat \un\big)=:\hat \pi_{t-s,w}^\CC
$$
with differentials on the $r$-th page of tri-degree $(r,r-1,0)$. 
\begin{itemize}
\item  the $\Ext$ group is computed in the category of bigraded comodules over the Hopf algebroid
 $(A,\M_2)$ (as defined in \ref{sec:hmot-torsion}), $s$ being the degree of the extension group,
 and $(t,w)$ being the internal degree of that category;
\item on the abutment, $\hat \un$ denotes the $2$-completion (defined in \ref{sec:intern-comp&loc})
 of the motivic sphere spectrum,
 or equivalently its nilpotent $\HMot\FF_2$-completion according to the preceding example.
\end{itemize}
In addition, the spectral sequence is strongly convergent, and the filtration
 on the abutment is finite in each \emph{motivic stem} $f=(t-s)$.
\end{prop}
Indeed, this is simply the $\HMot\FF_2$-Adams spectral sequence constructed above.
 To get the correct form of the $E_2$-term, one needs the appropriate weak K\"unneth property to apply Corollary \ref{cor:compute-E2}.
 This is proved in \cite[Proposition 7.5]{DIAdams}, but it follows more generally from the fact $\HMot \FF_\ell$ is cellular
 according to \cite{HoyoisHM} using \cite{DIcell}.
 The last assertion follows from \cite[Corollary 7.15]{DIAdams}, which establishes vanishing properties for the motivic Adams spectral sequence
 analogous to that of the Adams spectral sequence.
 
\begin{rema}
As in the topological case, \cite{IWX} (and all the other papers in this topic)
 uses a different grading convention to depict the $r$-th page of the motivic Adams spectral sequence.
 More precisely, the group  $E_r^{s,t,w}$ is displayed on a plane, following the conventions of the Adams charts,
 as described in Step 1 of Section~\ref{sec:cl-computing}:
\begin{itemize}
\item the lines are indexed by the integer $s$, which is called the \emph{Adams degree}. In particular, the picture is concentrated in degree $s \geq 0$;
\item the columns are indexed by the integer $f=t-s$, which is called the \emph{stem}.
\item the picture has to be thought of as a projection of a three-dimensional picture, along the axis represented by the index $w$,
 called the \emph{weight}. 
\end{itemize}
To summarize, a group of grading $(s,t,w)$ is pictured in the point of coordinates $(f=t-s,s)$.
 Then the differentials on the $r$-th page have tri-degree $(1,r-1,0)$.
 Moreover, when displaying the $E_r$-term for $r$ big enough,
 each column represents the gradings of the abelian group $\hat \pi_{f,*}^\CC$.
\end{rema}

There is a deep interplay between the motivic Adams spectral sequence and the (classical) Adams spectral
 sequence. Roughly speaking, one gets back the latter from the former by inverting the motivic element $\tau$.
 More precisely, building on Corollary \ref{cor:mot-cl-Hopf-algebroid}, Dugger and Isaksen proved
 the following remarkable result.\footnote{This is stated explicitly in \cite[Proposition 3.0.2]{IsakStems}.}
\begin{theo}\label{thm:DI}
Keep the assumptions of the preceding proposition.
 Then the complex realization functor induces an isomorphism of weight-graded
 $\M_2[\tau^{-1}]$-linear spectral sequences:
$$
E_{*,\HMot \FF_2}^{***}[\tau^{-1}] \xrightarrow \sim \E_{*,\sH\FF_2}^{**} \otimes_{\FF_2} \M_2[\tau^{-1}]
$$
where we have indicated the ring spectra for more clarity.
\end{theo}
In fact, one can derive a direct proof using the construction of $\E_*$-Adams spectral sequences
 from the cobar resolution and Corollary \ref{cor:F_l-compare-HM-HB}.

As both spectral sequences are strongly convergent, one deduces the following corollary.
\begin{coro}
Consider the above assumptions.
 Then the element $\tau$ lifts to a (homological) bidegree $(0,-1)$ element in 
 the $\hat \pi_{**}^\CC$ motivic stable homotopy groups of the $2$-completed motivic sphere,
 that we continue to denote by $\tau$.
 Moreover, the complex realization functor induces an isomorphism of bigraded $\M_2[\tau^{-1}]$-algebras:
$$
\hat \pi_{**}^\CC[\tau^{-1}] \simeq \piS * \otimes_{\ZZ} \ZZ_2[\tau,\tau^{-1}]
$$
where on the left-hand side, one has localized with respect to the element $\tau$,
 and the isomorphism maps the motivic class $\tau$ to the indeterminate $\tau$ on the right-hand side.
\end{coro}
In particular, we have obtained a map
\begin{equation}
\tau:\hat \un(-1) \rightarrow \hat \un
\end{equation}
 which lifts the previously defined map $\tau:\HMot\FF_2(-1) \rightarrow \HMot \FF_2$.
 This is a key player of the computations of \cite{IWX}.
 The reader may appreciate the direct and elegant construction of this lift due
 to \cite[Remark following Lemma 23]{HKO_Adams} and based on Morel's Theorem~\ref{thm:Morel-stable}.

\begin{rema}
One can check that the isomorphism of the above corollary is compatible with the one constructed by Levine
 in Theorem \ref{thm:LevinePi} for $k=\CC$, using the canonical map
 $\pi_{**}^\CC \rightarrow \hat \pi_{**}^\CC$.
\end{rema}

\subsection{Deforming homotopy theories via the motivic class $\tau$}

\subsubsection{The special fiber of $\tau$}
We learned in the end of the previous section that inverting the motivic element $\tau$
 --- which in particular kills all $\tau$-torsion elements ---
 allows one to recover the classical Adams spectral sequence from the
 motivic Adams spectral sequence.

On the other hand,  the pages of the motivic Adams spectral sequence may contain non-trivial $\tau$-torsion classes,
 and they sometimes hide extensions that still are visible in the abutment. To better understand this phenomenon,
 \cite[Chapter 5]{IsakStems}, introduces the \emph{cofiber of $\tau$} denoted
 by $C\tau$ and defined as the following homotopy cofiber in the $\infty$-category
 of $2$-complete motivic spectra\footnote{therefore, it can be computed by first taking the homotopy cofiber
 in the motivic homotopy category and then applying the $2$-completion functor;}
\begin{equation}\label{eq:tau-cofiber}
\hat \un(-1) \xrightarrow \tau \hat \un \xrightarrow i C\tau \xrightarrow \partial \hat \un(-1)[1].
\end{equation}
This procedure of killing homotopy classes is very classical.
 A remarkable result of \cite{Gheorghe} is that, in this particular case,
 the cofiber $C\tau$ not only acquires a ring structure in homotopy category, but can also be equipped
 with an $E_\infty$-ring structure compatible with the canonical map $i$.
 This allows one to define the $\infty$-category $C\tau-\mathrm{mod}$ of modules over $C\tau$,
 which can be thought of as \emph{$\tau$-torsion $2$-complete motivic spectra over $\CC$}.
 Although preliminary computations of \cite[Proposition 6.2.5]{IsakStems} may suggest such a connection,
 the following beautiful result proved by \textcite[Corollary 1.2]{GWX} reveals another surprising bridge between motivic and classical
 homotopical invariants.
\begin{theo}\label{thm:tau-cofib}
With the above notation, there exists a canonical equivalence of stable $\infty$-categories
$$
C\tau-\mathrm{mod}^{\mathrm{cell}} \simeq \mathrm{Stable}(\BP_*\BP-\mathrm{comod}^\mathrm{ev})
$$
where:
\begin{itemize}
\item the left-hand side is made by the cellular $C\tau$-modules, i.e.,
 the full sub-$\infty$-category spanned by colimits of $C\tau$-modules
 of the form $C(\tau)(q)[p]$ for any pair $(p,q) \in \ZZ^2$;
\item the right-hand side is Hovey's stable $\infty$-category of even comodules
 over the Hopf algebroid $(\BP_*,\BP_*\BP)$, obtained by localizing complexes of such comodules
 along homotopy isomorphisms.
\end{itemize} 
\end{theo}

\begin{rema}
\begin{enumerate}
\item This result can be reformulated in terms of stacks:
 the $\infty$-category of cellular $C\tau$-modules
 is equivalent to the ind-$\infty$-category of perfect complexes on the moduli stack of
 formal group laws $\mathcal M_{FGL}$ over $\ZZ_2$ (see also Remark \ref{rem:stacks}).
 This interpretation essentially underlies the proof of Corollary 1.2 given by \cite{GWX},
 as discussed after Remark 4.15 therein.
\item The preceding theorem was generalized by \cite{BJWX}
 through the construction of the so-called \emph{Chow-$t$-structure} on the stable motivic homotopy
 category $\iSH(k)$ for any base field $k$.

This $t$-structure can be described as the unique one whose non-negative objects
 are generated under colimits and extension by the Thom spaces $\Th(v)$ of a virtual vector bundle $v$
 over a smooth \emph{and proper} $k$-scheme $X$.\footnote{The name is likely inspired by the work of Bondarko,
 who defined the \emph{Chow weight structure} on the category of motivic complexes over $k$
 using similar types of generators. See, e.g., \cite[\textsection 7.1]{Bondarko} for effective motivic complexes.}
 The authors then relate the invariants associated with this $t$-structure --- such as truncations, its heart, and heart-valued objects ---
 to even comodules over appropriate Hopf algebroids, after inverting the characteristic exponent $e$ of the base field $k$.

As an example,
 they identify the cellular heart of the Chow-$t$-structure on $\SH(k)[e^{-1}]$
 with the $e$-localized $\infty$-category of $(\MU_*,\MU_*\MU)$-modules (see Theorem 1.12).
\end{enumerate}
\end{rema}

The important point from the computational perspective is the following corollary,
 which follows from the previous theorem (see \cite{GWX}, Theorem~1.3 for the statement
 and Part~2 for details).
\begin{coro}
There is an isomorphism of spectral sequences between the motivic Adams spectral sequence for
 $C\tau$ and the algebraic Adams-Novikov spectral sequence.\footnote{The latter
 is an algebraic spectral sequence based on resolutions of $\BP_*\BP$-comodules,
 and computes the $E_2$-term of the Adams-Novikov spectral sequence.
 See \emph{loc. cit.} \textsection9.1,
 and also Remark \ref{rem:ssp->E2} for a broader picture.}
\end{coro}

\subsubsection{Final procedure to compute motivic and classical stable stems}
The innovative ingredient introduced by \cite{IWX} for computing stable stems
 out of motivic stable stems is the following \emph{deformation diagram}
\begin{equation}\label{eq:def-tau}
\xymatrix@=40pt{
{}\hat \un[\tau^{-1}] & \hat \un\ar[l]\ar@<-3pt>_i[r] & C\tau\ar@<-3pt>_{\partial}[l]
}
\end{equation}
where the left-hand (resp. right-hand) side is thought as the generic (resp. special) fiber.
 Concretely, this diagram means two facts:
\begin{enumerate}
\item the motivic Adams spectral sequence is constrained along the maps $i$ and $\partial$
 by the algebraic Adams-Novikov spectral sequence, which corresponds according to the previous theorem
 to the Adams spectral of the special fiber $C\tau$;
\item the generic fiber computes the classical Adams spectral sequence,  according to
 Theorem \ref{thm:DI}.
\end{enumerate}

From a practical point of view, this leads to the following computational strategy,
 which both draws on and refines the classical approach to computing stable stems
 described in Section~\ref{sec:cl-computing}:
\begin{enumerate}[label=$\mathrm{Step}^\tau$\arabic*., leftmargin=*]
\item Begin by computing the $E_2$-term of the motivic Adams spectral sequence.
 This tri-graded object captures refined information, including torsion in the motivic weight.
 The computation relies on algebraic tools such as the May spectral sequence,
 and reduces to an algorithmic problem which can be handled by computer.

Compute the pages of the algebraic Adams-Novikov spectral sequence
 using similar algebraic methods, that can also be handled by computer.
\item The previous step gives the determination of the Adams spectral sequence for the cofiber of $\tau$.
 Then the right part of the deformation diagram allows one to determine
 differentials of the motivic Adams spectral sequence,
 and also additional refined information (notably, \emph{Toda brackets}). 
 Ultimately, this allows one to determine the Adams motivic $E_\infty$-page up to a fixed range.
\item Solve the extension problem for determining the motivic stable stem
 from the information of the $E_\infty$-term (as in Step 3 in the original strategy \ref{sec:cl-computing}),
 using again the information coming from the special fiber.
 This implies finding the so-called \emph{hidden extensions}\footnote{as precisely defined in \cite{IWX}, Definition 2.10;}
 and in particular with respect to the motivic element $\tau$.
\item The last step is obvious, and uses the left part of the deformation diagram:
 read off the classical stable stems (and information on the classical Adams differentials)
 from the computations obtained in the previous step.
\end{enumerate}

This motivic approach has led to a significant breakthrough in the effective computation of stable stems.
 Thanks to the method introduced by \cite{IWX}, it is now possible
 to push the determination of the $2$-primary component\footnote{this is by far the most difficult primary part of the stable stems;}
 of the stable homotopy groups of spheres from dimension 66 up to dimension 90.\footnote{Albeit a few remaining uncertainties: four unresolved differentials
 in the motivic Adams spectral sequence, see Table~9 in \emph{loc. cit.}}

\subsubsection{Towards synthetic homotopy}
As in topology, $\mathrm{Step}^\tau$2\&3 cannot be made algorithmic and carried out by a computer.
 In particular, one of the key technical tools that the authors had to use is the construction of
 an intermediate motivic ring spectrum over $\CC$, called the \emph{motivic modular forms spectrum}.
 It is the motivic analog of the famous \emph{topological modular forms spectrum} of Ando, Hopkins, Strickland
 (see e.g. \cite{GoerssB}).
 The way they construct this object is both interesting and relevant to conclude this lecture.

One still works, as above, in the stable monoidal $\infty$-category $\iSH(\CC)^\wedge_2$,
 the $2$-completed motivic stable homotopy category over $\CC$.\footnote{From what we have seen before,
 it can be described in full generality with the left Bousfield localization of $\iSH(\CC)$
 with respect to the motivic spectrum $\un/2$. Under an appropriate finiteness assumption,
 this is also the Bousfield localization with respect to $\HMot \FF_2$; see Remark \ref{rem:Bousloc}.}
 As in the previous theorem and following \cite{DIcell},
 one restricts to the full sub-$\infty$-category $\iSHcell(\CC)^\wedge_2$ of cellular $2$-complete motivic spectra;
 that is the full sub-$\infty$-category of $\iSH(\CC)^\wedge_2$ spanned by colimits of motivic sphere spectra $\hat \un(q)[p]$
 for $(p,q) \in \ZZ^2$.\footnote{The author emphasizes that, from a motivic perspective, cellular objects should also be viewed as \emph{ind-Artin–Tate objects},
 since we work over the algebraically closed field $\CC$.
 Artin–Tate motives play a key role in the theory of motivic complexes
 and in our current understanding of $L$-functions and periods over number fields.}

With the aim of describing the latter $\infty$-category, \textcite{GIKR} discovered the idea that one can
 work directly with the Adams resolutions of spectra, framed in the language of filtered spectra.
 This leads them to the $\infty$-category $\Fun(\ZZ^{\op},\iSH^\wedge_2)$ of \emph{filtered $2$-complete spectra},
 which is presentable, stable and monoidal (using the Day convolution product).
 Via a fully faithful embedding, one can consider the tower Adams resolutions of spectra defined in \ref{prop:cobar&standard}
 viewed as objects in this $\infty$-category. The importance of the Adams-Novikov spectral sequence (Section \ref{sec:Adams-Novikov})
 justifies specifically considering the tower Adams resolution associated with the $2$-completion $\wMU$ of $\MU$:
$$
T_\bullet:=\Tot_\bullet \CB^\bullet(\wMU).
$$
In order to make the next theorem work, one looks
 at a suitably truncated tower, denoted by $\Gamma_\star(S^0)$ in \emph{loc. cit.},\footnote{The symbol $\star$ here
 replaces the symbol $\bullet$ that we have used up to now to suggest the underlying filtration on objects.
 They serve exactly the same purpose.}
 whose $w$-th term is given by the following formula:
$$
\Gamma_w(S^0):=\Tot_\bullet(\tau_{\geq 2w} \CB^\bullet(\wMU))
$$
where $\tau_{\geq 2w}$ is the truncation functor associated with the canonical $t$-structure on $\iSH$,
 and we apply it term-wise to the cosimplicial object $\CB^\bullet(\wMU)$
 (see Definition 3.2 of \emph{loc. cit.}).\footnote{The effect of this construction is that the spectral
 sequence associated to $\Gamma_\bullet(S^0)$ is a truncation of the Adams-Novikov spectral sequence associated to $\MU$.
 See \emph{loc. cit.} Remark 3.4.}
 It is proved in \emph{loc. cit.} that $\Gamma_\star(S^0)$ is actually an $E_\infty$-object in filtered spectra.
 In particular, one can consider the $\infty$-category of modules $\Gamma_\star(S^0)\!-\!\mathrm{mod}$.
 The following result, Theorem 6.12 of \emph{loc. cit.},
 sheds light on the relation between classical and motivic invariant that we have met
 on several occasions in this section.
\begin{theo}
There is an explicit pair of mutually inverse equivalences of stable monoidal $\infty$-categories:
$$
\xymatrix@=30pt{
\Gamma_\star(S^0)\!-\!\mathrm{mod}\ar@<2pt>^\sim[r] & \iSHcell(\CC)^\wedge_2.\ar@<2pt>^\sim[l]
}
$$
\end{theo}
Thanks to this surprising equivalence,
 the authors were able to define the motivic modular forms spectrum $\mathrm{mmf}$ in the right-hand side $\infty$-category,
 by transporting a topological construction done in filtered $2$-complete spectra.
 In fact, one can also realize the fundamental deformation diagram \eqref{eq:def-tau} directly
 in $\Gamma_\star(S^0)\!-\!\mathrm{mod}$.

\subsubsection{Epilogue}
The construction and properties of $\Gamma_\star(S^0)\!-\!\mathrm{mod}$ have since been axiomatized under the name
 \emph{synthetic homotopy theory} by \cite{Pstra}. Though arising from motivic methods,
 this framework is now fully topological and expected to play a central role in the future of algebraic topology.
 This stands as a beautiful example of the deep and productive interactions
 between algebraic geometry and algebraic topology through motivic homotopy theory.

\printshorthands 

\printbibliography

\end{document}